\documentclass[12 pt]{article}
\usepackage[a4paper, total={6.5in, 9.5in}]{geometry}
\usepackage[T1]{fontenc}
\usepackage[utf8]{inputenc}
\usepackage[english]{babel}

\usepackage{graphicx}

\usepackage{amsmath}
\usepackage{amsfonts}
\usepackage{amssymb}
\usepackage{amsthm}
\usepackage[alphabetic]{amsrefs}

\allowdisplaybreaks

\usepackage{mathtools}

\usepackage[colorlinks=true,linkcolor=blue,urlcolor=blue,citecolor=blue]{hyperref}

\usepackage{wedn}
\usepackage[T1]{fontenc}
\newcommand{\XA}{{\mbox{\large{\wedn{a}}}}}
\newcommand{\XB}{{\mbox{\large{\wedn{b}}}}}
\newcommand{\XC}{{\mbox{\large{\wedn{c}}}}\,}
\newcommand{\XX}{\underline{\mbox{\large{\wedn{x}}}}\,}
\newcommand{\XY}{\underline{\mbox{\large{\wedn{y}}}}\,}
\newcommand{\XT}{\underline{\mbox{\large{\wedn{t}}}}\,}

\title{Local density of solutions \\ of time and space fractional equations\thanks{Supported by
the Australian Research Council Discovery Project 170104880 NEW ``Nonlocal
Equations at Work''. The authors are members of INdAM/GNAMPA.}}
\author{Alessandro Carbotti\thanks{Dipartimento di Matematica
e Fisica, Universit\`a del Salento,
Via Per Arnesano, 73100 Lecce, Italy. {\tt alessandro.carbotti@unisalento.it}}, Serena Dipierro\thanks{Department
of Mathematics and Statistics,
University of Western Australia,
35 Stirling Highway,
Crawley WA 6009, Australia. {\tt serena.dipierro@uwa.edu.au} },
and
Enrico Valdinoci\thanks{Department of Mathematics and Statistics,
University of Western Australia,
35 Stirling Highway,
Crawley WA 6009, Australia, 
and Istituto di Matematica Applicata e Tecnologie Informatiche,
Consiglio Nazionale delle Ricerche,
Via Ferrata 1, 27100 Pavia, Italy,
and Dipartimento di Matematica, Universit\`a degli studi di Milano,
Via Saldini 50, 20133 Milan, Italy. {\tt enrico@mat.uniroma3.it} }}

\newtheorem{theorem}{Theorem}[section]
\newtheorem{remark}[theorem]{Remark}
\newtheorem{lemma}[theorem]{Lemma}
\newtheorem{proposition}[theorem]{Proposition}

\newtheorem{corollary}[theorem]{Corollary}
\numberwithin{equation}{section}

\begin{document}
\maketitle
\begin{abstract}
We prove that 
any given function can be smoothly approximated by
functions lying in the kernel of a linear operator
involving at least one fractional component.
The setting in which we work is very general, since it takes into account
anomalous diffusion, with possible
fractional components
in both space and time. The operators studied comprise
the case of
the sum of classical and fractional Laplacians, possibly of different
orders, in the space variables, and classical or fractional derivatives
in the time variables.

This type of approximation results
shows that space-fractional and time-fractional
equations exhibit a variety of solutions which is much richer and more
abundant than in the case of classical diffusion.
\end{abstract}

\tableofcontents

\section{Introduction and main results}
\label{s:first}
In this paper we prove the local density of functions
which annihilate a linear operator built by classical and 
fractional derivatives, both in space and time.\medskip

Nonlocal operators of fractional type
present a variety of challenging problems in pure mathematics,
also in connections with long-range phase transitions and nonlocal
minimal surfaces,
and are nowadays commonly exploited in a large number of models
describing complex phenomena related
to anomalous diffusion and boundary reactions
in physics, biology and material sciences (see e.g.~\cite{claudia}
for several examples, for instance in atom dislocations in crystals and
water waves models).
Furthermore, anomalous diffusion in the space variables can be seen
as the natural counterpart of discontinuous
Markov processes, thus providing important connections
with problems in probability and statistics, and several applications to
economy and finance (see e.g.~\cite{MR0242239,MR3235230} for pioneer works relating
anomalous diffusion and financial models).\medskip

On the other hand, the development of time-fractional derivatives
began at the end of the seventeenth century, also in view of
contributions by mathematicians
such as Leibniz, Euler, Laplace, Liouville and many others,
see e.g.~\cite{ferrari} and the references therein for several
interesting scientific and historical discussions.
{F}rom the point of view of the applications, time-fractional derivatives 
naturally provide a model to comprise memory effects in the description of the
phenomena under consideration. \medskip

In this paper, the time-fractional derivative will be mostly described
in terms of the so-called
Caputo fractional derivative (see~\cite{MR2379269}),  which
induces a natural ``direction'' in
the time variable, distinguishing between ``past'' and ``future''. In particular, the time direction encoded in this setting
allows the analysis of ``non anticipative systems'', namely phenomena in which
the state at a given time depends on past events, but not on future ones.
The Caputo derivative is also related to other types of time-fractional derivatives,
such as the Marchaud fractional derivative, which has
applications in modeling anomalous time
diffusion, see e.g.~\cite{MR3488533, AV, ferrari}.
See also~\cite{MR1219954, MR1347689} for more details on fractional
operators and several applications.\medskip

In this article, we will take advantadge of the nonlocal structure
of a very general linear operator containing fractional derivatives
in some variables (say, either time, or space, or both),
in order to approximate, in the smooth sense and with arbitrary precision,
any prescribed function. Remarkably, {\em no structural assumption}
needs to be taken on the prescribed function: therefore
this approximation property reveals a {\em truly nonlocal behaviour},
since it is
in contrast with the rigidity of the functions that lie in the kernel
of classical linear
operators (for instance, harmonic functions cannot approximate 
a function with interior maxima or minima, functions with null first derivatives 
are necessarily constant, and so on).

The approximation
results with solutions of nonlocal operators have been first introduced
in~\cite{MR3626547}
for the case of the fractional Laplacian,
and since then widely studied under different perspectives,
including harmonic analysis,
see~\cite{MR3774704, 2016arXiv160909248G, 2017arXiv170804285R, 2017arXiv170806294R, 2017arXiv170806300R}.
The approximation result for the
one dimensional case of a fractional derivative of Caputo type
has been considered in~\cite{MR3716924, CDV18}, and
operators involving classical time derivatives and additional classical derivatives
in space have been studied in~\cite{DSV1}.
\medskip

The great flexibility of solutions of fractional problems established
by this type of approximation results
has also consequences that go beyond
the purely mathematical curiosity. 
For example, these results 
can be applied to study the evolution of
biological populations, showing how a nonlocal hunting or dispersive
strategy can be 
more convenient than one based on classical diffusion,
in order to avoid waste of resources and optimize the search for food
in sparse environment, see~\cite{MR3590678, MR3579567}.
Interestingly, the theoretical descriptions
provided in this setting can be compared with a series of concrete
biological data and real world experiments, confirming
anomalous diffusion behaviours in many biological species, see~\cite{ALBA}.

\medskip

Another interesting application of time-fractional derivatives
arises in neuroscience, for instance in view of the anomalous diffusion
which has been experimentally measured in neurons, see e.g.~\cite{SANTA}
and the references therein.
In this case, the anomalous diffusion could be seen as the effect
of the highly ramified structure of the biological cells
taken into account, see~\cite{DV1}. 
\medskip

In many applications, it is also natural to consider the case in
which different types of diffusion take 
place in different variables: for instance, classical diffusion
in space variables could be
naturally combined to anomalous diffusion with respect to variables
which take into account genetical information, see~\cite{GEN1, GEN2}.
\medskip

Now, to state the main results of this paper, we introduce some notation.
In what follows, we will denote the ``local variables''
with the symbol $x$, the ``nonlocal variables''
with $y$, the ``time-fractional variables''
with $t$.
Namely, we consider the variables
\begin{equation}\label{1.0}\begin{split}
&x=\left(x_1,\ldots,x_n\right)\in\mathbb{R}^{p_1}\times\ldots\times\mathbb{R}^{p_n},
\\&
y=\left(y_1,\ldots,y_M\right)\in\mathbb{R}^{m_1}
\times\ldots\times\mathbb{R}^{m_M}\\
{\mbox{and }}\;&
t=\left(t_1,\ldots,t_l\right)\in\mathbb{R}^l,\end{split}\end{equation}
for some $p_1,\dots,p_n$, $M$, $m_1,\dots,m_M$, $l \in\mathbb{N}$, and we let
$$\left(x,y,t\right)\in\mathbb{R}^N,\qquad{\mbox{
where }}\;N:=p_1+\ldots+p_n+m_1+\ldots+m_M+l.$$ 
When necessary, we will use the notation $B_R^k$ to denote the $k$-dimensional
ball of radius $R$, centered at the origin in $\mathbb{R}^k$; otherwise, when there are no ambiguities, we will use the usual notation $B_R$.

Fixed $r=\left(r_1,\ldots,r_n\right)\in\mathbb{N}^{p_1}\times\ldots\times\mathbb{N}^{p_n}$,
with~$|r_i|\ge1$ for each~$i\in\{1,\dots, n \}$,
and~$\XA=\left(\XA_1,\ldots,\XA_n\right)\in\mathbb{R}^n$, we consider the local operator
acting on the variables~$x=(x_1,\dots,x_n)$ given by
\begin{equation}\label{ILPOAU-1}
\mathfrak{l}:=\sum_{i=1}^n {\XA_i\partial^{r_i}_{x_i}}.
\end{equation}
where the multi-index notation has been used.

Furthermore, given $\XB=\left(\XB_1,\ldots,\XB_M\right)\in\mathbb{R}^M$ and
$s=\left(s_1,\ldots,s_M\right)\in\left(0,+\infty\right)^M$, we consider the operator
\begin{equation}\label{ILPOAU-2}
\mathcal{L}:=\sum_{j=1}^M {\XB_j(-\Delta)^{s_j}_{y_j}},
\end{equation}
where each operator~$(-\Delta)^{s_j}_{y_j}$
denotes the fractional Laplacian of order $2s_j$ acting
on the set of space variables~$y_j\in\mathbb{R}^{m_j}$. More precisely,
for any~$j\in\{1,\dots,M\}$,
given $h_j\in\mathbb{N}$ and $s_j\in\left(0,h_j\right)$,
in the spirit of~\cite{AJS2}, we consider
the operator
\begin{equation}
\label{nonlocop}
(-\Delta)^{s_j}_{y_j}u\left(x,y,t\right):=\int_{\mathbb{R}^{m_j}} 
{\frac{\left(\delta_{h_j} u\right)
\left(x,y,t,Y_j\right)}{|Y_j|^{m_j+2s_j}} \,dY_j},
\end{equation}
where
\begin{equation}\label{898989ksdc}
\left(\delta_{h_j} u\right)\left(x,y,t,Y_j\right):=
\sum_{k=-h_j}^{h_j} {\left(-1\right)^k \binom{2h_j}{h_j-k}u
\left(x,y_1,\ldots,y_{j-1},y_j+kY_j,y_{j+1},\ldots,y_M,t\right)}.\end{equation}

In particular, when $h_j:=1$, this setting comprises the case
of the fractional Laplacian~$\left(-\Delta\right)^{s_j}_{y_j}$ of order~$2s_j\in(0,2)$, given by
\begin{equation*}\begin{split}
\left(-\Delta\right)^{s_j}_{y_j}u\left(x,y,t\right) 
&\, := c_{m_j,s_j}\; 
\int_{\mathbb{R}^{m_j}} 
\Big(2u(x,y,t)-
u(x,y_1,\ldots,y_{j-1},y_j+Y_j,y_{j+1},\ldots,y_M,t)\\&\qquad\qquad-
u(x,y_1,\ldots,y_{j-1},y_j-Y_j,y_{j+1},\ldots,y_M,t)
\Big)\;
\frac{dY_j}{|Y_j|^{m_j+2s_j}},\end{split}
\end{equation*}
where $s_j\in(0,1)$ and $c_{m_j,s_j}$ denotes a multiplicative normalizing constant
(see e.g. formula~(3.1.10) in~\cite{claudia}).

It is interesting to recall that
if $h_j=2$ and $s_j=1$ the setting in~\eqref{nonlocop}
provides a nonlocal representation for the classical Laplacian,
see \cite{AV}.
\medskip

In our general framework,
we take into account also nonlocal operators of time-fractional type.
To this end,
for any~$\alpha>0$, letting $k:=[\alpha]+1$ and $a\in\mathbb{R}\cup\left\{-\infty\right\}$,
one can introduce the left\footnote{In the literature, one often finds also
the notion of right Caputo fractional derivative, defined for $t<a$ by 
$$ \frac{(-1)^k}{\Gamma\left(k-\alpha\right)}\int^{a}_{t}
\frac{\partial_{t}^{k} u(\tau)}{\left(\tau-t\right)^{\alpha-k+1}}
\, d\tau.$$
Since the right time-fractional derivative boils down to the left one
(by replacing~$t$ with~$2a-t$), in this paper we focus only
on the case of left derivatives.

Also, though there are several time-fractional derivatives that are studied
in the literature under different perspectives,
we focus here on the Caputo derivative, since it possesses well-posedness properties
with respect to classical initial
value problems, differently than other time-fractional derivatives, such as the
Riemann-Liouville derivative,
in which the initial value setting involves data containing
derivatives of
fractional order.}
Caputo fractional derivative of order $\alpha$
and initial point $a$, defined, for~$t>a$,
as
\begin{equation}
\label{defcap}
D_{t,a}^{\alpha}u(t):=\frac{1}{\Gamma(k-\alpha)}\int_a^t \frac{\partial_{t}^{k} u\left(\tau\right)}{(t-\tau)^{\alpha-k+1}} d\tau,
\end{equation}
where\footnote{For notational simplicity, we will often denote~$\partial_t^k u
=u^{(k)}$.} $\Gamma$ denotes the Euler's Gamma function. 

%%%%			For sufficiently smooth functions, a simple integration by parts allows to write the Caputo fractional derivative, when $\alpha\in(0,1)$, in the following equivalent form
%%%%			\begin{equation}
%%%%			\Gamma(1-\alpha)D_{t,a}^\alpha u(t):=\frac{u(t)-u(a)}{(t-a)^\alpha}+\alpha\int_a^t \frac{u(t)-u(\tau)}{(t-\tau)^{\alpha+1}} d\tau;
%%%%			\end{equation}
%%%%			notice that this definition takes into account only the values $t>a$, but if we assume that $u(t)=u(a)$ for any $t<a$, then we have
%%%%			\begin{equation}
%%%%			\begin{split}
%%%%			\label{march}
%%%%			&\frac{u(t)-u(a)}{(t-a)^\alpha}+\alpha\int_a^t \frac{u(t)-u(\tau)}{(t-\tau)^{\alpha+1}} d\tau=\alpha\int_{-\infty}^a \frac{u(t)-u(a)}{(t-\tau)^{\alpha+1}} d\tau+\alpha\int_a^t \frac{u(t)-u(\tau)}{(t-\tau)^{\alpha+1}} d\tau= \\
%%%%			&\alpha\int_{-\infty}^a \frac{u(t)-u(\tau)}{(t-\tau)^{\alpha+1}} d\tau+\alpha\int_a^t \frac{u(t)-u(\tau)}{(t-\tau)^{\alpha+1}} d\tau=\alpha\int_{-\infty}^t \frac{u(t)-u(\tau)}{(t-\tau)^{\alpha+1}} d\tau.
%%%%			\end{split}
%%%%			\end{equation}
%%%%			This formulation of \eqref{defcap} is better known as Marchaud fractional derivative with initial point $-\infty$; notice that \eqref{march} recalls the structure of the one-variable fractional Laplacian, but in this case, for the well-posedness of the boundary value problems, we need only a left-sided prescription of the solution. \\
In this framework,
fixed $\XC=\left(\XC_1,\ldots,\XC_l\right)\in\mathbb{R}^l$,
$\alpha=(\alpha_1,\dots,\alpha_l)\in(0,+\infty)^l$
and~$a=(a_1,\dots,a_l)\in(\mathbb{R}\cup\left\{-\infty\right\})^l$, 
we set
\begin{equation}\label{ILPOAU-3}
\mathcal{D}_a:=\sum_{h=1}^l \XC_h \,D_{t_h, a_h}^{\alpha_h}\,.
\end{equation} 
Then, in the notation introduced in~\eqref{ILPOAU-1}, \eqref{ILPOAU-2}
and~\eqref{ILPOAU-3}, we consider here the superposition of the
local, the nonlocal, and the time-fractional operators, that is, we set
\begin{equation}\label{1.6BIS}
\Lambda_a:=\mathfrak{l}+\mathcal{L}+\mathcal{D}_a.
\end{equation}
With this, the statement of our main result goes as follows:

\begin{theorem}\label{theone}
Suppose that 
\begin{equation}\label{NOTVAN}\begin{split}&
{\mbox{either there exists~$i\in\{1,\dots,M\}$ such that~$\XB_i\ne0$
and~$s_i\not\in{\mathbb{N}}$,}}\\
&{\mbox{or there exists~$i\in\{1,\dots,l\}$ such that~$\XC_i\ne0$ and $\alpha_i\not\in{\mathbb{N}}$.}}\end{split}
\end{equation}
Let $\ell\in\mathbb{N}$, $f:\mathbb{R}^N\rightarrow\mathbb{R}$,
with $f\in C^{\ell}\big(\overline{B_1^N}\big)$. Fixed $\epsilon>0$,
there exist
\begin{equation}\label{EX:eps}\begin{split}&
u=u_\epsilon\in C^\infty\left(B_1^N\right)\cap C\left(\mathbb{R}^N\right),\\
&a=(a_1,\dots,a_l)=(a_{1,\epsilon},\dots,a_{l,\epsilon})
\in(-\infty,0)^l,\\ {\mbox{and }}\quad&
R=R_\epsilon>1\end{split}\end{equation} such that
\begin{equation}\label{MAIN EQ}\left\{\begin{matrix}
\Lambda_a u=0 &\mbox{ in }\;B_1^N, \\
u=0&\mbox{ in }\;\mathbb{R}^N\setminus B_R^N,
\end{matrix}\right.\end{equation}
and
\begin{equation}\label{IAzofm}
\left\|u-f\right\|_{C^{\ell}(B_1^N)}<\epsilon.
\end{equation}
\end{theorem}

We observe that the initial points of the Caputo type operators
in Theorem~\ref{theone} also depend on~$\epsilon$, as detailed in~\eqref{EX:eps}
(but the other parameters, such as the orders of the operators
involved, are fixed arbitrarily).\medskip

We also stress that condition~\eqref{NOTVAN} requires that
the operator~$\Lambda_a$ contains at least one nonlocal operator
among its building blocks in~\eqref{ILPOAU-1}, \eqref{ILPOAU-2}
and~\eqref{ILPOAU-3}. This condition cannot be avoided, since
approximation results in the same spirit of Theorem~\ref{theone}
cannot hold for classical differential operators.\medskip

Theorem~\ref{theone} comprises, as particular cases,
the nonlocal approximation results established in the recent literature of this
topic. Indeed, when
\begin{eqnarray*}
&&\XA_1=\dots=\XA_n=\XB_1=\dots=\XB_{M-1}=\XC_1=\dots=\XC_l=0,\\
&& \XB_M=1,\\
{\mbox{and }}&& s\in(0,1)
\end{eqnarray*}
we see that Theorem~\ref{theone} recovers the main result
in~\cite{MR3626547}, giving the local density of $s$-harmonic functions
vanishing outside a compact set.\medskip

Similarly, when
\begin{eqnarray*}&&
\XA_1=\dots=\XA_n=\XB_1=\dots=\XB_{M}=\XC_1=\dots=\XC_{l-1}=0,\\
&&\XC_l=1,\\{\mbox{and }}&&\mathcal{D}_a=D_{t,a}^{\alpha},
\quad{\mbox{ for some~$\alpha>0$, $a<0$}}\end{eqnarray*}
we have that
Theorem~\ref{theone} reduces to
the main results in~\cite{MR3716924} for~$\alpha\in(0,1)$
and~\cite{CDV18} for~$\alpha>1$, in which such approximation result
was established for Caputo-stationary functions, i.e, functions that annihilate
the Caputo fractional derivative.\medskip

Also, when
\begin{eqnarray*}
&&p_1=\dots=p_{n}=1,\\
&&\XC_1=\dots=\XC_{l}=0,\\ {\mbox{and }}
&&s_j\in(0,1),\quad{\mbox{for every~$j\in\{1,\dots,M\},$}}\end{eqnarray*}
we have that
Theorem~\ref{theone} recovers the cases taken into account in~\cite{DSV1},
in which approximation results have been established
for the superposition of a local operator
with a superposition of fractional Laplacians of order~$2s_j<2$.\medskip

In this sense, not only Theorem~\ref{theone} comprises the existing literature,
but it goes beyond it, since it combines 
classical derivatives, fractional Laplacians and Caputo fractional derivatives
altogether.
In addition, it comprises the cases in which the space-fractional Laplacians 
taken into account are of order
greater than~$2$. 

As a matter of fact, this point is also a novelty
introduced by Theorem~\ref{theone} here
with respect to the previous literature.

Theorem~\ref{theone} was announced in~\cite{CDV18}, and
we have just received the very interesting preprint~\cite{2018arXiv181007648K}
which also considered the case of
different, not necessarily fractional, powers of the Laplacian,
using a different and innovative methodology.
\medskip

The rest of the paper is organized as follows. 
Sections~\ref{s:second}
and~\ref{s:grf0} focus on time-fractional operators.
More precisely, in Sections~\ref{s:second}
and~\ref{s:grf0}
we study the boundary behaviour of the eigenfunctions of the
Caputo derivative and of functions with vanishing Caputo derivative, respectively,
detecting their singular
boundary behaviour in terms of
explicit representation formulas. These type of results are also
interesting in themselves and can find further applications. 

Sections \ref{s:grf}--\ref{s:hwb} are devoted to some properties
of the higher order fractional Laplacian. 
More precisely, Section~\ref{s:grf} provides
some representation formula of the solution
of~$(-\Delta)^s u=f$ in a ball, with~$u=0$ outside this ball, for all~$s>0$,
and extends
the Green formula methods introduced
in~\cite{MR3673669} and \cite{AJS3}.

Then, in Section~\ref{SEC:eigef} we 
study the boundary behaviour of the first Dirichlet
eigenfunction of higher order fractional equations, and in Section~\ref{sec5}
we give some precise asymptotics
at the boundary for the first Dirichlet eigenfunction
of~$(-\Delta)^s$ for any~$s>0$.

Section~\ref{s:hwb} is devoted to
the analysis of the asymptotic behaviour of $s$-harmonic
functions, with a ``spherical bump function'' as exterior Dirichlet datum.

Section~\ref{s:fourthE} contains an auxiliary statement, namely
Theorem~\ref{theone2}, which will imply
Theorem~\ref{theone}. This is technically convenient,
since the operator~$\Lambda_a$ depends in principle on the initial
point~$a$: this has the disadvantage that if~$\Lambda_a u_a=0$
and~$\Lambda_b u_b=0$ in some domain, the function~$u_a+u_b$
is not in principle a solution of any operator, unless~$a=b$.
To overcome such a difficulty, in Theorem~\ref{theone2}
we will reduce to the case in which~$a=-\infty$, exploiting
a polynomial extension that we have introduced and used in~\cite{CDV18}.

In Section~\ref{s:fourth0}
we make the main step towards the proof of Theorem~\ref{theone2}.
In this section, we prove that functions in the kernel
of nonlocal operators such as the one in~\eqref{1.6BIS}
span with their derivatives a maximal Euclidean space.
This fact is special for the nonlocal case
and its proof is based on the boundary analysis of the fractional operators in both
time and space.
Due to the general form of the operator in~\eqref{1.6BIS},
we have to distinguish here several cases,
taking advantage of either the time-fractional
or the space-fractional components of the operators.

We conclude the paper with Section~\ref{s:fourth} in which we
complete the proof of
Theorem~\ref{theone2}, using the previous approximation
results and suitable rescaling arguments.

\section{Sharp boundary behaviour for the time-fractional eigenfunctions}\label{s:second}

In this section we show that the eigenfunctions of the Caputo fractional derivative 
in~\eqref{defcap}
have
an explicit representation via the Mittag-Leffler function.
For this,
fixed $\alpha$, $\beta\in\mathbb{C}$ with $\Re\left(\alpha\right)>0$,
for any $z$ with $\Re\left(z\right)>0$, we recall that
the Mittag-Leffler function is defined as
\begin{equation}\label{Mittag}
E_{\alpha,\beta}\left(z\right):=\sum_{j=0}^{+\infty} {\frac{z^j}{\Gamma
\left(\alpha j+\beta\right)}}.
\end{equation} 
The Mittag-Leffler function plays an important role in equations
driven by the Caputo derivatives, replacing the exponential function
for classical differential equations, as given by the following well-established result
(see \cite{MR3244285} and the references therein):

\begin{lemma}\label{lemma1}
Let~$\alpha\in(0,1]$,
$\lambda\in{\mathbb{R}}$, and
$a\in\mathbb{R}\cup\left\{-\infty\right\}$.
Then, the unique solution of the boundary value problem
\begin{equation*}\left\{
\begin{matrix}
D_{t,a}^{\alpha}u(t)=\lambda\, u(t) &
\mbox{ for any }t\in (a,+\infty),\\
u(a)=1 &
\end{matrix}\right.
\end{equation*}
is given by $E_{\alpha,1}\left(\lambda \left(t-a\right)^\alpha\right)$.
\end{lemma}

Lemma~\ref{lemma1} can be actually generalized\footnote{It is easily seen that
for~$k:=1$ Lemma~\ref{MittagLEMMA} boils down
to Lemma~\ref{lemma1}.} to any
fractional order of differentiation~$\alpha$:

\begin{lemma}\label{MittagLEMMA}
Let $\alpha\in(0,+\infty)$, with~$\alpha\in(k-1,k]$ and~$k\in\mathbb{N}$,
$a\in\mathbb{R}\cup\left\{-\infty\right\}$, and~$\lambda\in{\mathbb{R}}$. Then,
the unique 
continuous solution of the boundary value problem
\begin{equation}\label{CHE:0}\left\{
\begin{matrix}
D_{t,a}^{\alpha}u(t)=\lambda \,u(t) &
\mbox{ for any }t\in (a,+\infty),\\
u(a)=1 ,\\
\partial^m_t u(a)=0&\mbox{ for any }
m\in\{1,\dots,k-1\}
\end{matrix}\right.
\end{equation}
is given by $u\left(t\right)=E_{\alpha,1}\left(\lambda \left(t-a\right)^\alpha\right)$.

\begin{proof}
For the sake of simplicity we take $a=0$.
Also, the case in which~$\alpha\in{\mathbb{N}}$ can be checked with a direct computation,
so we focus on the case~$\alpha\in(k-1,k)$, with~$k\in{\mathbb{N}}$.

We
let $u\left(t\right):=E_{\alpha,1}\left(\lambda t^\alpha\right)$.
It is straightforward to see that~$ u(t)=1+{\mathcal{O}}(t^k)$ and therefore
\begin{equation}\label{CHE:1}
u(0)=1 \qquad{\mbox{and}}\qquad
\partial^m_t u(0)=0\;\mbox{ for any }\;
m\in\{1,\dots,k-1\}.
\end{equation}
We also claim that
\begin{equation}\label{CHE:2}
D_{t,a}^{\alpha}u(t)=\lambda \,u(t) \;
\mbox{ for any }\;t\in (0,+\infty).
\end{equation}
To check this,
we recall~\eqref{defcap} and~\eqref{Mittag} (with~$\beta:=1$), 
and we have that
\begin{eqnarray*}&&
D_{t,a}^{\alpha} u\left(t\right) \\&= &
\frac{1}{\Gamma\left(k-\alpha\right)}\int_0^t {\frac{u^{(k)}\left(\tau\right)}{
\left(t-\tau\right)^{\alpha-k+1}}\, d\tau} \\
&=& \frac{1}{\Gamma\left(k-\alpha\right)}
\int_0^t {\left(\sum_{j=1}^{+\infty} {\lambda^j\frac{\alpha j\left(\alpha j
-1\right)\ldots\left(\alpha j-k+1\right)}{\Gamma\left(\alpha j
+1\right)} \tau^{\alpha j-k}}\right)\frac{d\tau}{\left(t-\tau\right)^{\alpha-k+1}}} \\
&=& \sum_{j=1}^{+\infty} {\lambda^j\,\frac{\alpha j\left(\alpha j-1\right)
\ldots\left(\alpha j-k+1\right)}{\Gamma\left(k-\alpha\right)
\Gamma\left(\alpha j+1\right)}}
\int_0^t {\tau^{\alpha j-k}\left(t-\tau\right)^{k-\alpha-1} \,d\tau}.
\end{eqnarray*}
Hence, using the change of variable $\tau=t\sigma$, we obtain that
\begin{equation}\label{H:1}
D_{t,a}^{\alpha} u\left(t\right) =
\sum_{j=1}^{+\infty} {\lambda^j\,
\frac{\alpha j\left(\alpha j-1\right)\ldots\left(\alpha j-k+1\right)}{
\Gamma\left(k-\alpha\right)\Gamma\left(\alpha j+1\right)}}
t^{\alpha j-\alpha}\int_0^1 {\sigma^{\alpha j-k}\left(1-\sigma\right)^{k-\alpha-1}
\, d\tau}.\end{equation}
On the other hand, from the basic properties of the Beta function, it is known
that if~$\Re(z)$, $\Re(w)>0$, then
\begin{equation}\label{H:2} \int_0^1 {\sigma^{z-1}\left(1-\sigma\right)^{w-1}\, dt}
=\frac{\Gamma\left(z\right)\Gamma\left(w\right)}{\Gamma\left(z+w\right)}.
\end{equation}
In particular, taking~$z:=\alpha j-k+1\in(\alpha-k+1,+\infty)\subseteq(0,+\infty)$
and~$w:=k-\alpha\in(0,+\infty)$, and substituting~\eqref{H:2} into~\eqref{H:1},
we conclude that
\begin{equation}\label{H:3}\begin{split}
D_{t,a}^{\alpha} u\left(t\right)=
& \sum_{j=1}^{+\infty} {\lambda^j\frac{\alpha j\left(\alpha j-1\right)\ldots\left(\alpha j-k+1\right)}{\Gamma\left(k-\alpha\right)\Gamma\left(\alpha j+1\right)}}\frac{\Gamma\left(\alpha j-k+1\right)\Gamma\left(k-\alpha\right)}{\Gamma\left(\alpha j-\alpha+1\right)}\,t^{\alpha j-\alpha} \\
=& \sum_{j=1}^{+\infty} {\lambda^j\frac{\alpha j\left(\alpha j-1\right)\ldots\left(\alpha j-k+1\right)}{\Gamma\left(\alpha j+1\right)}}\frac{\Gamma\left(\alpha j-k+1\right)}{\Gamma\left(\alpha j-\alpha+1\right)}\,t^{\alpha j-\alpha}.
\end{split}\end{equation}
Now we use the fact that $z\Gamma\left(z\right)=\Gamma\left(z+1\right)$
for any $z\in\mathbb{C}$ with $\Re\left(z\right)>-1$, so, we have
\begin{equation*}
\alpha j\left(\alpha j-1\right)\ldots\left(\alpha j-k+1\right)\Gamma\left(\alpha j-k+1\right)=\Gamma\left(\alpha j+1\right).
\end{equation*}
Plugging this information into~\eqref{H:3}, we thereby find that
\begin{equation*}
D_{t,a}^{\alpha} u\left(t\right)=
\sum_{j=1}^{+\infty} {\frac{\lambda^j}{\Gamma\left(\alpha j-\alpha+1\right)}t^{\alpha j-\alpha}}=\sum_{j=0}^{+\infty} {\frac{\lambda^{j+1}}{\Gamma\left(\alpha j+1\right)}t^{\alpha j}}=\lambda u(t).
\end{equation*}
This proves~\eqref{CHE:2}.

Then, in view of~\eqref{CHE:1} and~\eqref{CHE:2} we obtain that~$u$
is a solution of~\eqref{CHE:0}.
Hence, to complete the proof of the desired result, we have to show
that such a solution is unique. To this end, supposing that we have two
solutions of~\eqref{CHE:0}, we consider their difference~$w$, and
we observe that~$w$ is a solution of
\begin{equation*}\left\{
\begin{matrix}
D_{t,0}^{\alpha}w(t)=\lambda \,w(t) &
\mbox{ for any }t\in (0,+\infty),\\
\partial^m_t w(0)=0&\mbox{ for any }
m\in\{0,\dots,k-1\}.
\end{matrix}\right.
\end{equation*}
By Theorem~4.1 in~\cite{MR3563609}, it follows that~$w$ vanishes
identically, and this proves the desired uniqueness result.
\end{proof}
\end{lemma}

%%%%			Another useful property we take into account is the scaling property of the Caputo derivative.
%%%%			\begin{proposition}
%%%%			\label{capsca}
%%%%			Let $r>0$ and $u_r(t):=u(rt)$. Then
%%%%			\begin{equation}
%%%%			D_{t,a}^{\alpha}u_r(t)=r^\alpha \,D_{t,ra}^{\alpha}u (rt).
%%%%			\end{equation}
%%%%			\begin{proof}
%%%%			With a simple computation, we have
%%%%			\begin{align*}
%%%%			D_{t,a}^{\alpha}u_r(t)&=\frac{1}{\Gamma(k-\alpha)}\int_a^t \frac{u_r^{(k)}(\tau)}{(t-\tau)^{\alpha-k+1}}\, d\tau=\frac{r^k}{\Gamma(k-\alpha)}\int_a^t \frac{u^{(k)}(r\tau)}{(t-\tau)^{\alpha-k+1}} \,d\tau \\
%%%%			&=\frac{r^{k-1}}{\Gamma(k-\alpha)}\int_{ra}^{rt} \frac{u^{(k)}(w)}{\left(t-\frac{w}{r}\right)^{\alpha-k+1}}\, dw=\frac{r^\alpha}{\Gamma(k-\alpha)}\int_{ra}^{rt} \frac{u^{(k)}(w)}{\left(rt-w\right)^{\alpha-k+1}} \,dw \\
%%%%			&=r^\alpha \,D_{t,ra}^{\alpha}u(rt),
%%%%			\end{align*}
%%%%			where we have also used the change of variable $w:=r\tau$.
%%%%			\end{proof}
%%%%			\end{proposition}

The boundary behaviour of the Mittag-Leffler function for different values
of the fractional parameter~$\alpha$ is depicted in Figure~\ref{FIG1}.
In light of~\eqref{Mittag},
we notice in particular that, near~$z=0$,
$$ E_{\alpha,\beta}\left(z\right)=
\frac{1}{\Gamma\left(\beta\right)}+\frac{z}{\Gamma\left(\alpha +\beta\right)}+O(z^2)$$
and therefore, near~$t=a$,
\begin{equation*}
E_{\alpha,1}\left(\lambda \left(t-a\right)^\alpha\right)
=1+\frac{\lambda \left(t-a\right)^\alpha}{\Gamma\left(\alpha +1\right)}+O\big(\lambda^2\left(t-a\right)^{2\alpha}\big).
\end{equation*}

\begin{figure}[h]
\centering
\includegraphics[width=8 cm]{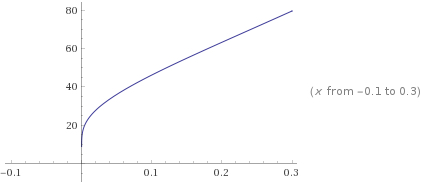}
\includegraphics[width=8 cm]{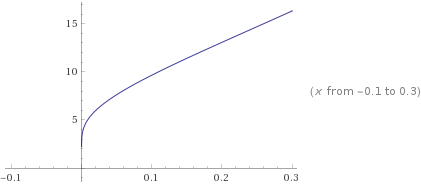}\\
\bigskip
\includegraphics[width=8 cm]{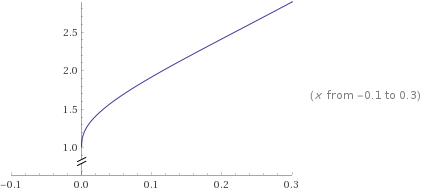}
\includegraphics[width=8 cm]{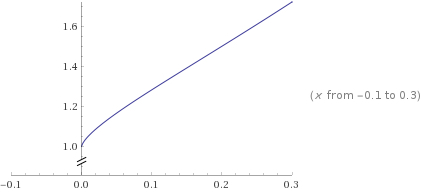}\\
\bigskip
\includegraphics[width=8 cm]{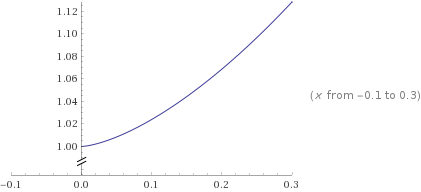}
\includegraphics[width=8 cm]{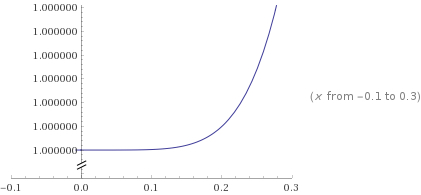}
\caption{\footnotesize\it Behaviour of the
Mittag-Leffler
function $E_{\alpha,1}\left(t^\alpha\right)$ near the origin
for $\alpha=\frac{1}{100}$, $\alpha=\frac{1}{20}$, $\alpha=\frac{1}{3}$, 
$\alpha=\frac{2}{3}$
$\alpha=\frac{3}{2}$ and~$\alpha=\frac{11}{2}$.}
\label{FIG1}
\end{figure}

%\section{Scale invariance for fractional Laplacian}\label{s:third}
%
%In this section we show that the operators taken into account in $\mathcal{L}$
%in~\eqref{ILPOAU-2}
%are scale-invariant. Namely, we point out the details
%of the following elementary observation:
%
%\begin{lemma}
%Let~$s_j\in(0,+\infty)$ and~$u\in C^\infty({\mathbb{R}}^{m_j})$.
%Fixed~$r>0$,
%let~$u_r(y_j):=u(ry_j)$. Then
%$$ (-\Delta)^{s_j}_{y_j} u_r(y_j)= r^{2s_j} (-\Delta)^{s_j}_{y_j}u(ry_j).$$
%\end{lemma}
%
%\begin{proof}
%{F}rom~\eqref{nonlocop} and~\eqref{898989ksdc}, we have that
%\begin{eqnarray*}
%(-\Delta)^{s_j}_{y_j} u_r(y_j) &=&
%\int_{\mathbb{R}^{m_j}} 
%\frac{\left(\delta_{h_j} u_r\right)(y_j,Y_j)}{|Y_j|^{m_j+2s_j}}\; \,dY_j
%\\ &=&
%\sum_{k=-h_j}^{h_j}
%(-1)^k \binom{2h_j}{h_j-k}
%\int_{\mathbb{R}^{m_j}} 
%\frac{u_r(y_j+kY_j)}{|Y_j|^{m_j+2s_j}}\, \,dY_j
%\\ &=& \sum_{k=-h_j}^{h_j}
%(-1)^k \binom{2h_j}{h_j-k}
%\int_{\mathbb{R}^{m_j}} 
%\frac{u_r(y_j+kY_j)}{|Y_j|^{m_j+2s_j}}\, \,dY_j.
%\end{eqnarray*}
%Then, using the change of variable~$rY_j=:Z_j$, we obtain that
%\begin{eqnarray*}
%(-\Delta)^{s_j}_{y_j} u_r(y_j) &=&
%r^{m_j+2s_j}\sum_{k=-h_j}^{h_j}
%(-1)^k \binom{2h_j}{h_j-k}
%\int_{\mathbb{R}^{m_j}} 
%\frac{u(ry_j+kZ_j)}{|Z_j|^{m_j+2s_j}}\, \,\frac{dZ_j}{r^{m_j}}
%\\ &=& r^{2s_j}\,
%\sum_{k=-h_j}^{h_j}
%(-1)^k \binom{2h_j}{h_j-k}
%\int_{\mathbb{R}^{m_j}} 
%\;\frac{u(ry_j+kZ_j)}{\left|Z_j\right|^{m_j+2s_j}}
%\,dZ_j\\
%&=& r^{2s_j}\,
%\int_{\mathbb{R}^{m_j}} 
%\frac{\left(\delta_{h_j} u\right)(ry_j,Z_j)}{\left|Z_j\right|^{m_j+2s_j}}\;\,dZ_j\\
%&=& r^{2s_j} (-\Delta)^{s_j}_{y_j} u(ry_j),
%\end{eqnarray*}
%as desired.
%\end{proof}

\section{Sharp boundary behaviour for the time-fractional harmonic functions}\label{s:grf0}

In this section, we detect the optimal boundary behaviour of time-fractional
harmonic functions and of their derivatives.
The result that we need for our purposes is the following:

\begin{lemma}\label{LF}
Let~$\alpha\in(0,+\infty)\setminus\mathbb{N}$.
There exists a function~$\psi:\mathbb{R}\to\mathbb{R}$ such
that~$\psi\in C^\infty((1,+\infty))$ and
\begin{eqnarray}
\label{LAp1}&&D^\alpha_0\psi(t)=0 \qquad{\mbox{for all }}t\in(1,+\infty),\\
\label{LAp2}{\mbox{and }}&&\lim_{\epsilon\searrow0}
\epsilon^{\ell-\alpha} \partial^\ell\psi(1+\epsilon t)=\kappa_{\alpha,\ell}\, t^{\alpha-\ell},
\qquad{\mbox{for all }}\ell\in\mathbb{N},
\end{eqnarray}
for some~$\kappa_{\alpha,\ell}\in\mathbb{R}\setminus\{0\}$, where~\eqref{LAp2}
is taken in the sense of distribution for~$t\in(0,+\infty)$.
\end{lemma}

\begin{proof} We use Lemma~2.5
in~\cite{CDV18}, according to which
(see in particular formula~(2.16)
in~\cite{CDV18}) the claim in~\eqref{LAp1}
holds true. Furthermore (see formulas~(2.19)
and~(2.20) in~\cite{CDV18}), we can write that, for all~$t>1$,
\begin{equation}\label{VBVBHJSnb}
\psi(t)=-\frac1{\Gamma(\alpha)\Gamma([\alpha]+1-\alpha)}
\iint_{[1,t]\times[0,3/4]}\partial^{[\alpha]+1}\psi_0(\sigma)\,(\tau-\sigma)^{[\alpha]-\alpha}\,
(t-\tau)^{\alpha-1}\,d\tau\,d\sigma,
\end{equation}
for a suitable~$\psi_0\in C^{[\alpha]+1}([0,1])$.

In addition, by Lemma~2.6 in~\cite{CDV18}, we can write that
\begin{equation}\label{0oLAMJA}
\lim_{\epsilon\searrow0}
\epsilon^{ -\alpha} \psi(1+\epsilon)=\kappa,
\end{equation}
for some~$\kappa\ne0$.
Now we set
$$ (0,+\infty)\ni t\mapsto f_\epsilon(t):=
\epsilon^{\ell-\alpha} \partial^\ell\psi(1+\epsilon t).$$
We observe that, for any~$\varphi\in C^\infty_0((0,+\infty))$,
\begin{equation}\label{UHAikAJ678OKA}
\begin{split}& \int_0^{+\infty} f_\epsilon(t)\,\varphi(t)\,dt=
\epsilon^{\ell-\alpha} \int_0^{+\infty}\partial^\ell\psi(1+\epsilon t)\varphi(t)\,dt\\
&=
\epsilon^{-\alpha} \int_0^{+\infty}\frac{d^\ell}{dt^\ell}\big(\psi(1+\epsilon t)\big)\varphi(t)\,dt
=(-1)^\ell\,
\epsilon^{-\alpha} \int_0^{+\infty}\psi(1+\epsilon t)\,\partial^\ell\varphi(t)\,dt.
\end{split}\end{equation}
Also, in view of~\eqref{VBVBHJSnb},
\begin{eqnarray*}&&
\epsilon^{-\alpha}|\psi(1+\epsilon t)|\\
&=&\left|\frac{\epsilon^{-\alpha}}{\Gamma(\alpha)\Gamma([\alpha]+1-\alpha)}
\iint_{[1,1+\epsilon t]\times[0,3/4]}\partial^{[\alpha]+1}\psi_0(\sigma)\,(\tau-\sigma)^{[\alpha]-\alpha}\,
(1+\epsilon t-\tau)^{\alpha-1}\,d\tau\,d\sigma\right|
\\&\le&C\,\epsilon^{-\alpha}\,
\int_{[1,1+\epsilon t]}(1+\epsilon t-\tau)^{\alpha-1}\,d\tau
\\ &=& Ct^\alpha,
\end{eqnarray*}
which is locally bounded in~$t$, where~$C>0$ here above may vary from line to line.

As a consequence, we can pass to the limit in~\eqref{UHAikAJ678OKA}
and obtain that
$$\lim_{\epsilon\searrow0} \int_0^{+\infty} f_\epsilon(t)\,\varphi(t)\,dt
=
(-1)^\ell\, \int_0^{+\infty} \lim_{\epsilon\searrow0}
\epsilon^{-\alpha}
\psi(1+\epsilon t)\,\partial^\ell\varphi(t)\,dt.
$$
This and~\eqref{0oLAMJA} give that
$$\lim_{\epsilon\searrow0} \int_0^{+\infty} f_\epsilon(t)\,\varphi(t)\,dt
=
(-1)^\ell\, \kappa\,\int_0^{+\infty} t^\alpha\,
\partial^\ell\varphi(t)\,dt=\kappa\,\alpha\dots(\alpha-\ell+1)\int_0^{+\infty} t^{\alpha-\ell}\,\varphi(t)\,dt,$$
which establishes~\eqref{LAp2}.
\end{proof}

\section{Green representation formulas and solution of $(-\Delta)^s u=f$ in $B_1$ with homogeneous Dirichlet datum}
\label{s:grf}
Our goal is to provide some representation results on the solution
of~$(-\Delta)^s u=f$ in a ball, with~$u=0$ outside this ball, for all~$s>0$.
Our approach is an extension of the Green formula methods introduced
in~\cite{MR3673669} and \cite{AJS3}:
differently from the previous literature, we are not assuming here that~$f$
is regular in the whole of the ball, but merely that it is H\"older continuous
near the boundary and sufficiently integrable inside. Given the type of singularity
of the Green function, these assumptions are sufficient to obtain meaningful
representations, which in turn will be useful to deal with
the eigenfunction problem in the subsequent Section~\ref{SEC:eigef}.

\subsection{Solving $(-\Delta)^s u=f$ in $B_1$ for discontinuous~$f$ vanishing near $\partial B_1$}

In this subsection, we want to extend the representation
results of \cite{MR3673669} and \cite{AJS3} to
the case in which the right hand side is not H\"older continuous,
but merely in a Lebesgue space, but it has the additional
property of vanishing near the boundary of the domain.
To this end, fixed~$s>0$,
we consider the polyharmonic Green
function in $B_1\subset{\mathbb{R}}^n$, given,
for every~$x\ne y\in\mathbb{R}^n$, by 
\begin{equation}\label{GREEN}
\begin{split}& 
\mathcal{G}_s\left(x,y\right):=\frac{k(n,s)}{\left|x-y\right|^{n-2s}}\,
\int_0^{r_0\left(x,y\right)} 
\frac{\eta^{s-1}}{\left(\eta+1\right)^{\frac{n}{2}}} \,d\eta,
\\ {\mbox{where }}\quad&r_0\left(x,y\right):=
\frac{\left(1-\left|x\right|^2\right)_+\left(1-\left|y\right|^2\right)_+}{\left|x-y\right|^2}, \\
{\mbox{with }}\quad&
k(n,s):=\frac{\Gamma\left(\frac{n}{2}\right)}{
\pi^{\frac{n}{2}}\,4^s\Gamma^2\left(s\right)}. 
\end{split}
\end{equation}
Given~$x\in B_1$, we also set
\begin{equation}\label{GREEN2} d(x):=1-|x|.\end{equation}
In this setting, the main result of this subsection is the following:

\begin{proposition}\label{LONTANO}
Let~$r\in(0,1)$ and~$f\in L^2(B_1)$, with~$f=0$ in~${\mathbb{R}}^n\setminus B_r$. Let
\begin{equation}\label{DEF uG} u(x):=
\begin{cases}
\displaystyle\int_{B_1} \mathcal{G}_s\left(x,y\right)\,f(y)\,dy & {\mbox{ if }}x\in B_1,\\
0&{\mbox{ if }}x\in{\mathbb{R}}^n\setminus B_1.
\end{cases}
\end{equation}
Then:
\begin{equation}
\label{LON1}
u\in L^1(B_1), {\mbox{ and }} \|u\|_{L^1(B_1)}\le C\,\|f\|_{L^1(B_1)},
\end{equation}
\begin{equation}\label{LON2}
{\mbox{for every $R\in(r,1)$, }} \sup_{x\in B_1\setminus B_R}
d^{-s}(x)\,|u(x)|\le C_R\,\|f\|_{L^1(B_1)},\end{equation}
\begin{equation}\label{LON3}
{\mbox{$u$ satisfies }}(-\Delta)^s u=f{\mbox{ in }}B_1 {\mbox{ in the sense of
distributions,}} 
\end{equation}
and
\begin{equation}\label{LON4}
u\in W^{2s,2}_{loc}(B_1).
\end{equation}
Here above, 
$C>0$ is a constant depending on~$n$, $s$ and~$r$,
$C_R>0$ is a constant depending on~$n$, $s$, $r$ and~$R$
and $C_\rho>0$ is a constant depending on~$n$, $s$, $r$ and~$\rho$.
\end{proposition}

When~$f\in C^\alpha(B_1)$ for some~$\alpha\in(0,1)$,
Proposition~\ref{LONTANO} boils down to the main results of \cite{MR3673669}
and~\cite{AJS3}.

\begin{proof}[Proof of Proposition~\ref{LONTANO}] 
We recall the following useful estimate, see Lemma~3.3 in 
\cite{AJS3}:
for any~$\epsilon\in\left(0,\,\min\{n,s\}\right)$, and any~$\bar R$, $\bar r>0$,
$$ \frac1{\bar R^{n-2s}}\,\int_0^{\bar r/\bar R^2}\frac{\eta^{s-1}}{
\left(\eta+1\right)^{\frac{n}{2}}} \,d\eta\le
\frac{2}{s}\;\frac{\bar r^{s-(\epsilon/2)}}{\bar R^{n-\epsilon}},$$
and so, by~\eqref{GREEN} and~\eqref{GREEN2},
for every~$x$, $y\in B_1$,
$$ \mathcal{G}_s\left(x,y\right)\le
\frac{C\,d^{s-(\epsilon/2)}(x)\,d^{s-(\epsilon/2)}(y)}{|x-y|^{n-\epsilon}}
$$
for some~$C>0$. Hence, recalling~\eqref{DEF uG},
\begin{eqnarray*}
\int_{B_1} |u(x)|\,dx&\le& \int_{B_1}
\left(\int_{B_1} \mathcal{G}_s\left(x,y\right)\,|f(y)|\,dy\right)\,dx
\\ &\le& C \int_{B_1}
\left(\int_{B_1} \frac{|f(y)|}{|x-y|^{n-\epsilon}}\,dy\right)\,dx\\&=&
C \int_{B_1}
\left(\int_{B_1} \frac{|f(y)|}{|x-y|^{n-\epsilon}}\,dx\right)\,dy\\&=&
C \int_{B_1}|f(y)|\,dy,
\end{eqnarray*}
up to renaming~$C>0$ line after line, and this proves~\eqref{LON1}.

Now, if~$x\in B_1\setminus B_R$ and~$y\in B_r$, with~$0<r<R<1$, we have that
$$|x-y|\ge |x|-|y|\ge R-r$$
and accordingly
$$ r_0\left(x,y\right)\le
\frac{2d(x)}{(R-r)^2},$$
which in turn implies that
$$ \mathcal{G}_s\left(x,y\right)\le\frac{k(n,s)}{\left|x-y\right|^{n-2s}}\,
\int_0^{{2d(x)}/{(R-r)^2}} 
\frac{\eta^{s-1}}{\left(\eta+1\right)^{\frac{n}{2}}} \,d\eta,
\le {C\,d^s(x)},$$
for some~$C>0$.
As a consequence,
since~$f$ vanishes outside~$B_r$, we see that, for any~$x\in B_1\setminus B_R$,
\begin{eqnarray*}
|u(x)|\le
\int_{B_r} \mathcal{G}_s\left(x,y\right)\,|f(y)|\,dy\le C\,d^s(x)
\int_{B_r} |f(y)|\,dy,
\end{eqnarray*}
which proves~\eqref{LON2}.

Now, we fix~$\hat r\in(r,1)$ and consider a mollification of~$f$,
that we denote by~$f_j\in C^\infty_0(B_{\hat r})$, with~$f_j\to f$
in~$L^2(B_1)$ as~$j\to+\infty$. 
We also write~$\mathcal{G}_s * f$ as a short notation for the right hand side
of~\eqref{DEF uG}. Then, by \cite{MR3673669} and \cite{AJS3},
we know that~$u_j:=\mathcal{G}_s * f_j$ is a (locally smooth, hence distributional) solution of~$(-\Delta)^s u_j=f_j$.
Furthermore, if we set~$\tilde u_j:=u_j-u$ and~$\tilde f_j:=f_j-f$
we have that
$$ \tilde u_j=\mathcal{G}_s * (f_j-f)=\mathcal{G}_s * \tilde f_j,$$
and therefore, by~\eqref{LON1},
$$ \|\tilde u_j\|_{L^1(B_1)}\le C\,\|\tilde f_j\|_{L^1(B_1)},$$
which is infinitesimal as~$j\to+\infty$. This says that~$u_j\to u$
in~$L^1(B_1)$ as~$j\to+\infty$, and consequently, for any~$\varphi\in C^\infty_0(B_1)$,
\begin{eqnarray*}
&&\int_{B_1} u(x)\,(-\Delta)^s\varphi(x)\,dx=\lim_{j\to+\infty}
\int_{B_1} u_j(x)\,(-\Delta)^s\varphi(x)\,dx
\\&&\qquad=\lim_{j\to+\infty}
\int_{B_1} f_j(x)\,\varphi(x)\,dx=\int_{B_1} f(x)\,\varphi(x)\,dx,
\end{eqnarray*}
thus completing the proof of~\eqref{LON3}.

Now, to prove~\eqref{LON4}, we can suppose that~$s\in(0,+\infty)\setminus{\mathbb{N}}$,
since the case of integer~$s$ is classical, see e.g. \cite{MR1814364}.
First of all, we claim that
\begin{equation}\label{08}
{\mbox{\eqref{LON4} holds true for every~$s\in(0,1)$.}}
\end{equation}
For this, we first claim that
if~$g\in C^\infty(B_1)$ and~$v$ is a (locally smooth) solution of~$(-\Delta)^s v=g$
in~$B_1$, with~$v=0$ outside~$B_1$, then~$v\in W^{2s,2}_{loc}(B_1)$,
and, for any~$\rho\in(0,1)$, 
\begin{equation}\label{BIX}
\|v\|_{W^{2s,2}(B_\rho)}\le C_\rho\,\|g\|_{L^2(B_1)}.
\end{equation}
This claim can be seen as a localization of
Lemma~3.1 of~\cite{MR2863859},
or a quantification of the last claim in Theorem~1.3 of~\cite{MR3641649}.
To prove~\eqref{BIX}, we let~$R_-<R_+\in(\rho,1)$,
and consider~$\eta\in C^\infty_0(B_{R_+})$ with~$\eta=1$ in~$B_{R_-}$.
We let~$v^*:=v\eta$, and we recall formulas~(3.2), (3.3)
and~(A.5) in~\cite{MR3641649}, according to which
\begin{equation*}\begin{split}&
(-\Delta)^s v^*-\eta(-\Delta)^s v=g^* \quad{\mbox{ in }}{\mathbb{R}}^n,\\
{\mbox{with }}\quad&\| g^*\|_{L^2({\mathbb{R}}^n)}\le C\,\|v\|_{W^{s,2}({\mathbb{R}}^n)}
,\end{split}\end{equation*}
for some~$C>0$. 

Moreover, 
using a notation taken from \cite{MR3641649}
we denote by~$W^{s,2}_0(\overline{B_1})$
the space of functions in~$W^{s,2}({\mathbb{R}}^n)$ vanishing outside~$B_1$
and we consider the dual space~$
W^{-s,2}_0(\overline{B_1})$. We remark that if~$h\in L^2(B_1)$
we can naturally identify~$h$ as an element of~$
W^{-s,2}_0(\overline{B_1})$ by considering the action of~$h$
on any~$\varphi\in W^{s,2}_0(\overline{B_1})$ as defined by
$$ \int_{B_1} h(x)\,\varphi(x)\,dx.$$
With respect to this, we have that
\begin{equation}\label{DUE} \|h\|_{W^{-s,2}_0(\overline{B_1})}=\sup_{{\varphi\in
W^{s,2}_0(\overline{B_1})}\atop{\|\varphi\|_{W^{s,2}_0(\overline{B_1})}=1}}
\int_{B_1}h(x)\,\varphi(x)\,dx\le\|h\|_{L^2(B_1)}.\end{equation}
We notice also that
$$ \|v\|_{W^{s,2}({\mathbb{R}}^n)}\le C\,\| g\|_{W^{-s,2}(\overline{B_1})},$$
in light of Proposition~2.1 of~\cite{MR3641649}. This and~\eqref{DUE}
give that
$$ \|v\|_{W^{s,2}({\mathbb{R}}^n)}\le C\,\| g\|_{L^2(B_1)}.$$
Then, by Lemma~3.1 of~\cite{MR2863859}
(see in particular formula~(3.2) there, applied here with~$\lambda:=0$),
we obtain that
\begin{equation}\label{BIX0}
\begin{split}
\|v^*\|_{W^{2s,2}({\mathbb{R}}^n)}\,&\le C\,\|
\eta(-\Delta)^s v+
g^*\|_{L^2({\mathbb{R}}^n)}\\
&\le C\,\big( \| (-\Delta)^s v\|_{L^2(B_{R_+})}+\|g^*\|_{L^2({\mathbb{R}}^n)}\big)
\\&=C\,\big( \| g\|_{L^2(B_{R_+})}+\|g^*\|_{L^2({\mathbb{R}}^n)}\big)\\
&\le C\,\big(\| g\|_{L^2(B_{1})}+
\|v\|_{W^{s,2}({\mathbb{R}}^n)}\big)\\
&\le C\, \| g\|_{L^2(B_{1})}
,\end{split}\end{equation}
up to renaming~$C>0$ step by step.
On the other hand, since~$v^*=v$ in~$B_\rho$,
$$ \|v\|_{W^{2s,2}(B_\rho)}=
\|v^*\|_{W^{2s,2}(B_\rho)}\le \|v^*\|_{W^{2s,2}({\mathbb{R}}^n)}.$$
{F}rom this and~\eqref{BIX0}, we obtain \eqref{BIX}, as desired.

Now, we let~$f_j$, $\tilde f_j$, $u_j$ and~$\tilde u_j$
as above and make use
of~\eqref{BIX} to write
\begin{equation}\label{qwe89:BBA}
\begin{split}
&\|u_j\|_{W^{2s,2}(B_\rho)}\le C_\rho\,\|f_j\|_{L^2(B_1)}
\\ {\mbox{and }}\quad&
\|\tilde u_j\|_{W^{2s,2}(B_\rho)}\le C_\rho\,\|\tilde f_j\|_{L^2(B_1)}.
\end{split}
\end{equation}
As a consequence,
taking the limit as~$j\to+\infty$
we obtain that
$$ \|u\|_{W^{2s,2}(B_\rho)}\le C_\rho\,\|f\|_{L^2(B_1)},$$
that is~\eqref{LON4} in this case, namely the claim in~\eqref{08}.

Now, to prove~\eqref{LON4}, we argue by induction on
the integer part of~$s$. When the integer part of~$s$
is zero, the basis of the induction is warranted by~\eqref{08}.
Then, to perform the inductive step, given~$s\in(0,+\infty)\setminus{\mathbb{N}}$,
we suppose that~\eqref{LON4} holds true for~$s-1$,
namely
\begin{equation}\label{0lL:AN1}
\mathcal{G}_{s-1} * f
\in W^{2s-2,2}_{loc}(B_1).
\end{equation}
Then, following~\cite{AJS3},
it is convenient to introduce the notation
$$ [x,y]:=\sqrt{|x|^2|y|^2-2x\cdot y+1}$$
and consider
the auxiliary kernel given, for every~$x\ne y\in B_1$, by
\begin{equation}
\label{aux}
P_{s-1}(x,y):=\frac{(1-|x|^2)^{s-2}_+(1-|y|^2)^{s-1}_+(1-|x|^2|y|^2)}{[x,y]^n}.
\end{equation} 
We point out that if~$x\in B_r$ with~$r\in(0,1)$,
then
\begin{equation}
\label{klop}
[x,y]^2\ge|x|^2|y|^2-2|x|\,| y|+1=(1-|x|\,|y|)^2\ge(1-r)^2>0.
\end{equation}
Consequently, since~$f$ is supported in~$B_r$,
\begin{equation}\label{0lL:AN2}
P_{s-1}*f\in C^\infty({\mathbb{R}}^n).\end{equation}
Then,
we recall that
\begin{equation}\label{0lL:AN3}
-\Delta_x\mathcal{G}_s(x,y)=\mathcal{G}_{s-1}(x,y)-CP_{s-1}(x,y),\end{equation}
for some~$C\in{\mathbb{R}}$, see Lemma~3.1 in \cite{AJS3}.

As a consequence, in view of~\eqref{0lL:AN1}, \eqref{0lL:AN2}, \eqref{0lL:AN3},
we conclude that
$$ -\Delta (\mathcal{G}_s*f)= (-\Delta_x\mathcal{G}_s)*f
\in W^{2s-2,2}_{loc}(B_1).$$
This and the classical elliptic regularity theory (see e.g. \cite{MR1814364})
give that~$\mathcal{G}_s*f\in W^{2s,2}_{loc}(B_1)$, which
completes the inductive proof and establishes~\eqref{LON4}.
\end{proof}

\subsection{Solving $(-\Delta)^s u=f$ in $B_1$ for~$f$ H\"older continuous near $\partial B_1$}

The goal of this subsection is
to extend the representation
results of \cite{MR3673669} and \cite{AJS3} to
the case in which the right hand side is not H\"older continuous
in the whole of the ball,
but merely in a neighborhood of the boundary.
This result is obtained here by superposing the
results in \cite{MR3673669} and \cite{AJS3}
with Proposition~\ref{LONTANO} here, taking advantage of the linear
structure of the problem.

\begin{proposition} \label{LEJOS}
Let~$f\in L^2(B_1)$.
Let~$\alpha$, $r\in(0,1)$ and assume that
\begin{equation}\label{CHlaIA}
f\in C^\alpha(B_1\setminus B_r).\end{equation}
In the notation of~\eqref{GREEN}, let
\begin{equation} \label{0olwsKA}
u(x):=
\begin{cases}
\displaystyle\int_{B_1} \mathcal{G}_s\left(x,y\right)\,f(y)\,dy & {\mbox{ if }}x\in B_1,\\
0&{\mbox{ if }}x\in{\mathbb{R}}^n\setminus B_1.
\end{cases}
\end{equation}
Then, in the notation of~\eqref{GREEN2}, we have that:
\begin{equation}\label{VIC2}
{\mbox{for every $R\in(r,1)$, }} \sup_{x\in B_1\setminus B_R}
d^{-s}(x)\,|u(x)|\le C_R\,\big(\|f\|_{L^1(B_1)}+
\|f\|_{L^\infty(B_1\setminus B_r)}\big)
,\end{equation}
\begin{equation}\label{VIC3}
{\mbox{$u$ satisfies }}(-\Delta)^s u=f{\mbox{ in }}B_1 {\mbox{ in the sense of
distributions,}} \end{equation}
and
\begin{equation}\label{VIC4} 
u\in W^{2s,2}_{loc}(B_1).
\end{equation}
Here above, 
$C>0$ is a constant depending on~$n$, $s$ and~$r$,
$C_R>0$ is a constant depending on~$n$, $s$, $r$ and~$R$
and $C_\rho>0$ is a constant depending on~$n$, $s$, $r$ and~$\rho$.
\end{proposition}

\begin{proof} We take~$r_1\in(r,1)$ and~$\eta\in C^\infty_0(B_{r_1})$
with~$\eta=1$ in~$B_r$.
Let also
$$f_1:=f\eta\qquad{\mbox{and}}\qquad f_2:=f-f_1.$$
We observe that~$f_1\in L^2(B_1)$, and that~$f_1=0$ outside~$B_{r_1}$.
Therefore, we are in the position of applying Proposition~\ref{LONTANO}
and find a function~$u_1$ (obtained by convolving~$\mathcal{G}_s$
against~$f_1$) such that
\begin{eqnarray}
&& \label{XLON2}
{\mbox{for every $R\in(r_1,1)$, }} \sup_{x\in B_1\setminus B_R}
d^{-s}(x)\,|u_1(x)|\le C_R\,\|f_1\|_{L^1(B_1)},\\ 
\label{XLON3}
&& {\mbox{$u_1$ satisfies }}(-\Delta)^s u_1=f_1{\mbox{ in }}B_1 {\mbox{ in the sense of
distributions,}} 
\\
\label{XLON4}{\mbox{and }}
&& u_1\in W^{2s,2}_{loc}(B_1).
\end{eqnarray}
On the other hand, we have that~$f_2=f(1-\eta)$
vanishes outside~$B_1\setminus B_r$
and it is H\"older continuous. Accordingly,
we can apply Theorem~1.1 of~\cite{AJS3}
and find a function~$u_2$ (obtained by convolving~$\mathcal{G}_s$
against~$f_2$) such that
\begin{eqnarray}
&& \label{YLON2}
{\mbox{for every $R\in(r_1,1)$, }} \sup_{x\in B_1\setminus B_R}
d^{-s}(x)\,|u_2(x)|\le C_R\,\|f_2\|_{L^\infty(B_1)},\\ 
\label{YLON3}
&& {\mbox{$u_2$ satisfies }}(-\Delta)^s u_2=f_2{\mbox{ in }}B_1 {\mbox{ in the sense of
distributions,}} 
\\
\label{YLON4}
{\mbox{and }}&& u_2\in C^{2s+\alpha}_{loc}(B_1).
\end{eqnarray}
Then, $f=f_1+f_2$, and thus,
in view of~\eqref{0olwsKA}, we have that~$
u=u_1+u_2$. Also, $u$ satisfies~\eqref{VIC2},
thanks to~\eqref{XLON2} and~\eqref{YLON2},
\eqref{VIC3},
thanks to~\eqref{XLON3} and~\eqref{YLON3}, and~\eqref{VIC4},
thanks to~\eqref{XLON4} and~\eqref{YLON4}.
\end{proof}

\section{Existence and regularity for the first eigenfunction of the higher order fractional Laplacian}\label{SEC:eigef}

The goal of these pages is
to study the boundary behaviour of the first Dirichlet
eigenfunction of higher order fractional equations.

For this, writing~$s=m+\sigma$, with~$m\in{\mathbb{N}}$ and~$\sigma\in(0,1)$,
we define the energy space 
\begin{equation}
\label{energy}
H_0^s\left(B_1\right):=\left\{u\in H^s\left(\mathbb{R}^n\right);\; u=0 \;\text{in}\; \mathbb{R}^n\setminus B_1\right\},
\end{equation}
endowed with the Hilbert norm
\begin{equation}
\label{energynorm} \left\|u\right\|_{H_0^s\left(B_1\right)}:=
\left(\sum_{\left|\alpha\right|\leq m} {\left\|\partial^\alpha u
\right\|^2_{L^2\left(B_1\right)}}+\mathcal{E}_s
\left(u,u\right)\right)^{\frac{1}{2}},\end{equation}
where
%%		$$\mathcal{E}_s\left(u,v\right):=\begin{cases}
%%		\mathcal{E}_{\sigma}\left(\Delta^{\frac{m}{2}}u,\Delta^{\frac{m}{2}}v\right),\hspace{0,4 cm}\text{if}\hspace{0.1 cm}m\hspace{0.1 cm}\text{is even} \\
%%		\sum_{k=1}^N {\mathcal{E}_{\sigma}\left(\partial_k\Delta^{\frac{m-1}{2}}u,\partial_k\Delta^{\frac{m-1}{2}}v\right)},\hspace{0,4 cm}\text{if}\hspace{0.1 cm}m\hspace{0.1 cm}\text{is odd}.
%%		\end{cases}$$
%%		In both cases, it reduces again to
\begin{equation}\label{ENstut}
\mathcal{E}_s\left(u,v\right)=\int_{\mathbb{R}^n} {
\left|\xi\right|^{2s}\mathcal{F}u\left(\xi\right)\overline{\mathcal{F}v\left(\xi\right)} \,
d\xi},\end{equation}
being~$\mathcal{F}$ the Fourier transform and using the notation~$\overline{z}$ to denote the complex conjugated
of a complex number~$z$.

In this setting, we consider~$u\in H^s_0(B_1)$ to be 
such that
\begin{equation}\label{dirfun}\begin{cases}
\left(-\Delta\right)^s u=\lambda_1 u &\quad\text{ in }B_1, \\
u=0 &\quad\text{ in } \mathbb{R}^n\setminus\overline{B}_1,
\end{cases}\end{equation}
for every~$s>0$, with~$\lambda_1$ as small as possible.

The existence of solutions of \eqref{dirfun} is ensured
via variational techniques, as stated in the following result:

\begin{lemma}\label{VARIA}
The functional~$\mathcal{E}_s\left(u,u\right)$
attains its minimum~$\lambda_1$ on the functions in~$H^s_0(B_1)$
with unit norm in~$L^2(B_1)$.

The minimizer satisfies~\eqref{dirfun}.

In addition, $\lambda_1>0$.

\begin{proof} The proof is based on the direct method
in the calculus of variations. We provide some details for completeness.
Let~$s=m+\sigma$,
with~$m\in\mathbb{N}$ and~$\sigma\in(0,1)$.
Let us consider a minimizing sequence~$u_j\in H^s_0(B_1)\subseteq H^m({\mathbb{R}}^n)$
such that~$\|u_j\|_{L^2(B_1)}=1$ and
$$ \lim_{j\to+\infty}\mathcal{E}_s\left(u_j,u_j\right)=\inf_{{u\in H^s_0(B_1)}\atop{
\|u\|_{L^2(B_1)}=1}}\mathcal{E}_s\left(u,u\right).$$
In particular, we have that~$u_j$ is bounded in~$H^s_0(B_1)$ uniformly in~$j$,
so, up to a subsequence, it converges to some~$u_\star$
weakly in~$H^s_0(B_1)$ and strongly in~$L^2(B_1)$
as~$j\to+\infty$.

The weak lower semicontinuity of the seminorm $\mathcal{E}_s\left(\cdot,\cdot\right)$
then implies that~$u_\star$ is the desired minimizer.

Then, given~$\phi\in C^\infty_0(B_1)$, we have that
$$ \mathcal{E}_s\left(u_\star+\epsilon\phi,u_\star+\epsilon\phi\right)
\ge\mathcal{E}_s\left(u_\star,u_\star\right),$$
for every~$\epsilon\in{\mathbb{R}}$, and this gives that~\eqref{dirfun} is
satisfied in the sense of distributions,
and also in the classical sense by the elliptic regularity theory.

Finally, we have that~$\mathcal{E}_s\left(u_\star,u_\star\right)>0$,
since~$u_\star$ (and thus~${\mathcal{F}}u_\star$) does not vanish identically.
Consequently,
$$ \lambda_1=\frac{ \mathcal{E}_s
\left(u_\star,u_\star\right)}{\|u_\star\|_{L^2(B_1)}^2}=
\mathcal{E}_s\left(u_\star,u_\star\right)>0,$$
as desired.
\end{proof}
\end{lemma}

Our goal is now to apply Proposition~\ref{LEJOS}
to solutions of~\eqref{dirfun}, taking~$f:=\lambda u$. To this end,
we have to check that condition~\eqref{CHlaIA}
is satisfied, namely that solutions of~\eqref{dirfun} are
H\"older continuous in $B_1\setminus B_r$, for any $0<r<1$.

To this aim, we prove that polyharmonic operators of any order~$s>0$ 
always admit
a first eigenfunction in the ball which does not change sign
and which is radially symmetric. For this, we start discussing
the sign property:

\begin{lemma}\label{ikAHHPKAK}
There exists a nontrivial solution of~\eqref{dirfun} that does not change sign.

\begin{proof} We exploit a method explained in details in
Section~3.1 of~\cite{MR2667016}. As a matter of fact,
when~$s\in{\mathbb{N}}$, the desired result is exactly Theorem~3.7
in~\cite{MR2667016}.

Let~$u$ be as in
Lemma~\ref{VARIA}.
If either~$u\ge0$ or~$u\le0$, then the desired result is proved.
Hence, we argue by contradiction,
assuming that~$u$ attains strictly positive and strictly negative values.
We define
$$ {\mathcal{K}}:=\{ w:{\mathbb{R}}^n\to{\mathbb{R}} {\mbox{ s.t. 
$\mathcal{E}_s\left(w,w\right)<+\infty$, and
$w\ge0$ in $B_1$}} \}.$$
Also, we set
$$ {\mathcal{K}}^\star
:=\{ w\in H^s_0(B_1) {\mbox{ s.t. }}
\mathcal{E}_s\left(w,v\right)\le0
{\mbox{ for all }}v\in {\mathcal{K}}\}.$$
We claim that
\begin{equation}\label{kstar po}
{\mbox{if~$w\in{\mathcal{K}}^\star$, then $w\le0$}}.
\end{equation}
To prove this, we recall the notation in~\eqref{GREEN},
take~$\phi\in C^\infty_0(B_1)\cap{\mathcal{K}}$,
and let
$$ v_\phi(x):=
\begin{cases}
\displaystyle\int_{B_1} \mathcal{G}_s\left(x,y\right)\,\phi(y)\,dy & {\mbox{ if }}x\in B_1,\\
0&{\mbox{ if }}x\in{\mathbb{R}}^n\setminus B_1.
\end{cases}
$$
Then~$v_\phi\in{\mathcal{K}}$ and it satisfies~$
\left(-\Delta\right)^s v_\phi=\phi$ in~$B_1$, thanks
to~\cite{MR3673669} or~\cite{AJS3}.

Consequently, we can write, for every~$x\in B_1$,
$$ \phi(x)={\mathcal{F}}^{-1}(|\xi|^{2s}{\mathcal{F}}v_\phi)(x).$$
Hence, for every~$w\in{\mathcal{K}}^\star$,
\begin{eqnarray*}
0&\ge&\mathcal{E}_s\left(w,v_\phi\right)\\
&=&
\int_{\mathbb{R}^n} 
\left|\xi\right|^{2s}\mathcal{F}v_\phi\left(\xi\right)
\overline{\mathcal{F}w\left(\xi\right)} \,d\xi
\\&=&
\int_{\mathbb{R}^n} {\mathcal{F}}^{-1}(
\left|\xi\right|^{2s}\mathcal{F}v_\phi)(x)\,
w\left(x\right) \,dx\\&=&
\int_{B_1} {\mathcal{F}}^{-1}(
\left|\xi\right|^{2s}\mathcal{F}v_\phi)(x)\,
w\left(x\right) \,dx\\&=&
\int_{B_1}\phi(x)\,w\left(x\right) \,dx.
\end{eqnarray*}
Since~$\phi$ is arbitrary and nonnegative, this gives that~$w\le0$,
and this establishes~\eqref{kstar po}.

Furthermore, by Theorem~3.4 in~\cite{MR2667016}, we can write
$$ u=u_1+u_2,$$
with~$u_1\in {\mathcal{K}}\setminus\{0\}$,
$u_2\in{\mathcal{K}}^\star\setminus\{0\}$, and~$\mathcal{E}_s\left(u_1,u_2\right)=0$.

We observe that
$$ \mathcal{E}_s\left(u_1-u_2,u_1-u_2\right)=
\mathcal{E}_s\left(u_1,u_1\right)+\mathcal{E}_s\left(u_2,u_2\right)
+2\mathcal{E}_s\left(u_1,u_2\right)=\mathcal{E}_s\left(u_1,u_1\right)+\mathcal{E}_s\left(u_2,u_2\right).$$
In the same way,
$$ \mathcal{E}_s\left(u,u\right)=
\mathcal{E}_s\left(u_1+u_2,u_1+u_2\right)=
\mathcal{E}_s\left(u_1,u_1\right)+\mathcal{E}_s\left(u_2,u_2\right),$$
and therefore
\begin{equation}\label{7yhbAxcvTFV}
\mathcal{E}_s\left(u_1-u_2,u_1-u_2\right)=\mathcal{E}_s\left(u,u\right).
\end{equation}
On the other hand,
\begin{eqnarray*}
\| u_1-u_2\|_{L^2(B_1)}^2-\| u\|_{L^2(B_1)}^2&=&\| u_1-u_2\|_{L^2(B_1)}^2-\| u_1+u_2\|_{L^2(B_1)}^2\\
&=& -4\int_{B_1} u_1(x)\,u_2(x)\,dx.
\end{eqnarray*}
As a consequence, since~$u_2\le0$ in view of~\eqref{kstar po},
we conclude that
$$ \| u_1-u_2\|_{L^2(B_1)}^2-\| u\|_{L^2(B_1)}^2\ge0.$$
This and~\eqref{7yhbAxcvTFV} say that the function~$u_1-u_2$
is also a minimizer for the variational problem in Lemma~\ref{VARIA}.
Since now~$u_1-u_2\ge0$, the desired result follows.
\end{proof}
\end{lemma}

Now, we define the spherical mean of a function~$v$ by
$$ v_\sharp(x):=
\frac{1}{\left|\mathbb{S}^{n-1}\right|}
\int_{\mathbb{S}^{n-1}} v({\mathcal{R}}_\omega\,x)
\,d{\mathcal{H}}^{n-1}(\omega)
$$
where~${\mathcal{R}}_\omega$ is the rotation corresponding to the solid angle~$\omega
\in{\mathbb{S}^{n-1}}$, ${\mathcal{H}}^{n-1}$ is the standard
Hausdorff measure, and~$\left|\mathbb{S}^{n-1}\right|=
{\mathcal{H}}^{n-1}(\mathbb{S}^{n-1})$.
Notice that~$v_\sharp(x)=v_\sharp({\mathcal{R}}_\varpi x)$
for any~$\varpi \in\mathbb{S}^{n-1}$, that is~$v_\sharp$
is rotationally invariant.

Then, we have:

\begin{lemma}
\label{lapsfercom}
Any positive power of the Laplacian commutes
with the spherical mean, that is
$$ \big( (-\Delta)^s v\big)_\sharp(x)=(-\Delta)^s v_\sharp(x).$$
\begin{proof} By density,
we prove the claim for a function~$v$ in the
Schwartz space of smooth and rapidly decreasing functions.
In this setting, writing~${\mathcal{R}}_\omega^T$
to denote the transpose of the rotation~${\mathcal{R}}_\omega$,
and changing variable~$\eta:={\mathcal{R}}_\omega^T\,\xi$,
we have that
\begin{equation}\label{RFA2}
\begin{split} (-\Delta)^s v({\mathcal{R}}_\omega\,x)\,&=
\int_{{\mathbb{R}}^n} |\xi|^{2s} {\mathcal{F}}v(\xi)\,
e^{2\pi i {\mathcal{R}}_\omega\,x\cdot\xi}\,d\xi\\ &=
\int_{{\mathbb{R}}^n} |\xi|^{2s} {\mathcal{F}}v(\xi)\,
e^{2\pi i x\cdot{\mathcal{R}}_\omega^T\,\xi}\,d\xi\\
&=
\int_{{\mathbb{R}}^n} |\eta|^{2s}
{\mathcal{F}}v({\mathcal{R}}_\omega\,\eta)\,
e^{2\pi i x\cdot\eta}\,d\eta.
\end{split}\end{equation}
On the other hand, using the substitution~$y:={\mathcal{R}}_\omega^T\,x$,
\begin{eqnarray*}
{\mathcal{F}}v({\mathcal{R}}_\omega\,\eta)
&=&\int_{{\mathbb{R}}^n} v(x)\,
e^{-2\pi i x\cdot{\mathcal{R}}_\omega\,\eta}\,dx\\
&=& \int_{{\mathbb{R}}^n} v(x)\,
e^{-2\pi i {\mathcal{R}}_\omega^T\,x\cdot\eta}\,dx\\
&=& \int_{{\mathbb{R}}^n} v({\mathcal{R}}_\omega\,y)\,
e^{-2\pi i y\cdot\eta}\,dy,
\end{eqnarray*}
and therefore, recalling~\eqref{RFA2},
$$ (-\Delta)^s v({\mathcal{R}}_\omega\,x)=
\iint_{{\mathbb{R}}^n\times{\mathbb{R}}^n} |\eta|^{2s}
v({\mathcal{R}}_\omega\,y)\,
e^{2\pi i (x-y)\cdot\eta}\,dy\,d\eta.$$
As a consequence,
\begin{eqnarray*}
\big( (-\Delta)^s v\big)_\sharp(x)&=&
\frac{1}{\left|\mathbb{S}^{n-1}\right|}
\int_{\mathbb{S}^{n-1}} (-\Delta)^s v({\mathcal{R}}_\omega\,x)
\,d{\mathcal{H}}^{n-1}(\omega)
\\ &=&
\frac{1}{\left|\mathbb{S}^{n-1}\right|}
\iiint_{\mathbb{S}^{n-1}\times{\mathbb{R}}^n\times{\mathbb{R}}^n} 
|\eta|^{2s}
v({\mathcal{R}}_\omega\,y)\,
e^{2\pi i (x-y)\cdot\eta}
\,d{\mathcal{H}}^{n-1}(\omega)\,dy\,d\eta\\
&=& \iint_{{\mathbb{R}}^n\times{\mathbb{R}}^n} 
|\eta|^{2s}
v_\sharp(y)\,
e^{2\pi i (x-y)\cdot\eta}\,dy\,d\eta\\&=&
\int_{{\mathbb{R}}^n} 
|\eta|^{2s} {\mathcal{F}} (v_\sharp)(\eta)\,
e^{2\pi i x\cdot\eta}\,d\eta\\&=&(-\Delta)^s v_\sharp(x),
\end{eqnarray*}
as desired.
\end{proof}
\end{lemma}

It is also useful to observe that the spherical mean is compatible with the energy bounds.
In particular we have the following observation:

\begin{lemma}
We have that
\begin{equation}\label{ENblaoe1}
\mathcal{E}_s\left(v_\sharp,v_\sharp\right)\le
\mathcal{E}_s\left(v,v\right).\end{equation}
Moreover,
\begin{equation}\label{ENblaoe2}
{\mbox{if~$v\in H^s_0(B_1)$, then so does~$v_\sharp$.}}\end{equation}

\begin{proof} We see that
\begin{eqnarray*}
{\mathcal{F}}(v_\sharp)(\xi)&=&
\int_{\mathbb{R}^n} v_\sharp(x)\,e^{-2\pi ix\cdot\xi}\,dx\\
&=&
\frac{1}{\left|\mathbb{S}^{n-1}\right|}
\iint_{\mathbb{S}^{n-1}\times{\mathbb{R}^n}} 
v({\mathcal{R}}_\omega\,x)\,e^{-2\pi ix\cdot\xi}
\,d{\mathcal{H}}^{n-1}(\omega)\,dx
\end{eqnarray*}
and therefore, taking the complex conjugated,
$$ \overline{ {\mathcal{F}}(v_\sharp)(\xi) }=
\frac{1}{\left|\mathbb{S}^{n-1}\right|}
\iint_{\mathbb{S}^{n-1}\times{\mathbb{R}^n}} 
v({\mathcal{R}}_\omega\,x)\,e^{2\pi ix\cdot\xi}
\,d{\mathcal{H}}^{n-1}(\omega)\,dx.$$
Hence, by~\eqref{ENstut}, and exploiting the changes of variables~$y:=
{\mathcal{R}}_{\omega}\,x$ and~$\tilde y:=
{\mathcal{R}}_{\tilde\omega}\,\tilde x$,
\begin{eqnarray*}&&
\mathcal{E}_s\left(v_\sharp,v_\sharp\right)\\&=&\int_{\mathbb{R}^n} 
\left|\xi\right|^{2s}\mathcal{F}(v_\sharp)\left(\xi\right)
\overline{\mathcal{F}(v_\sharp)\left(\xi\right)} \,d\xi\\
&=&\frac{1}{\left|\mathbb{S}^{n-1}\right|^2}
\iiiint\!\!\!\int_{\mathbb{S}^{n-1}\times\mathbb{S}^{n-1}\times{\mathbb{R}^n}\times{\mathbb{R}^n}\times{\mathbb{R}^n}} 
|\xi|^{2s}v({\mathcal{R}}_\omega\,x)\,v({\mathcal{R}}_{\tilde\omega}\,\tilde x)\,e^{2\pi i (\tilde x-x)\cdot\xi}
\,d{\mathcal{H}}^{n-1}(\omega)
\,d{\mathcal{H}}^{n-1}(\tilde\omega)\,dx\,d\tilde x\,d\xi\\
&=& \frac{1}{\left|\mathbb{S}^{n-1}\right|^2}
\iiiint\!\!\!\int_{\mathbb{S}^{n-1}\times\mathbb{S}^{n-1}\times{\mathbb{R}^n}\times{\mathbb{R}^n}\times{\mathbb{R}^n}} 
|\xi|^{2s}v(y)\,v(\tilde y)\,e^{2\pi i \tilde y\cdot{\mathcal{R}}_{\tilde\omega}\,\xi}
e^{-2\pi i y\cdot{\mathcal{R}}_\omega\,\xi}
\,d{\mathcal{H}}^{n-1}(\omega)
\,d{\mathcal{H}}^{n-1}(\tilde\omega)\,dy\,d\tilde y\,d\xi\\
&=& \frac{1}{\left|\mathbb{S}^{n-1}\right|^2}
\iiint_{\mathbb{S}^{n-1}\times\mathbb{S}^{n-1}\times{\mathbb{R}^n}} 
|\xi|^{2s}\,{\mathcal{F}}v({\mathcal{R}}_\omega\,\xi)\,\overline{
{\mathcal{F}}v({\mathcal{R}}_{\tilde\omega}\,\xi)}\,
\,d{\mathcal{H}}^{n-1}(\omega)
\,d{\mathcal{H}}^{n-1}(\tilde\omega)\,d\xi.
\end{eqnarray*}
Consequently, using the Cauchy-Schwarz Inequality,
and the substitutions~$\eta:=
{\mathcal{R}}_{\omega}\,\xi$ and~$\tilde\eta:=
{\mathcal{R}}_{\tilde\omega}\,\xi$,
\begin{eqnarray*}
\mathcal{E}_s\left(v_\sharp,v_\sharp\right)&\le&
\frac{1}{\left|\mathbb{S}^{n-1}\right|^2}
\iiint_{\mathbb{S}^{n-1}\times\mathbb{S}^{n-1}\times{\mathbb{R}^n}} 
|\xi|^{2s}\,\big|{\mathcal{F}}v({\mathcal{R}}_\omega\,\xi)\big|
\,\big|
{\mathcal{F}}v({\mathcal{R}}_{\tilde\omega}\,\xi)\big|\,
\,d{\mathcal{H}}^{n-1}(\omega)
\,d{\mathcal{H}}^{n-1}(\tilde\omega)\,d\xi
\\ &\le&
\frac{1}{\left|\mathbb{S}^{n-1}\right|^2}
\left(
\iiint_{\mathbb{S}^{n-1}\times\mathbb{S}^{n-1}\times{\mathbb{R}^n}} 
|\xi|^{2s}\,\big|{\mathcal{F}}v({\mathcal{R}}_\omega\,\xi)\big|^2
\,d{\mathcal{H}}^{n-1}(\omega)
\,d{\mathcal{H}}^{n-1}(\tilde\omega)\,d\xi\right)^{\frac12}\\&&\qquad\cdot
\left(
\iiint_{\mathbb{S}^{n-1}\times\mathbb{S}^{n-1}\times{\mathbb{R}^n}} 
|\xi|^{2s}\,
\big|{\mathcal{F}}v({\mathcal{R}}_{\tilde\omega}\,\xi)\big|^2
\,d{\mathcal{H}}^{n-1}(\omega)
\,d{\mathcal{H}}^{n-1}(\tilde\omega)\,d\xi\right)^{\frac12}\\
\\ &=&
\frac{1}{\left|\mathbb{S}^{n-1}\right|^2}
\left(
\iiint_{\mathbb{S}^{n-1}\times\mathbb{S}^{n-1}\times{\mathbb{R}^n}} 
|\eta|^{2s}\,\big|{\mathcal{F}}v(\eta)\big|^2
\,d{\mathcal{H}}^{n-1}(\omega)
\,d{\mathcal{H}}^{n-1}(\tilde\omega)\,d\eta\right)^{\frac12}\\&&\qquad\cdot
\left(
\iiint_{\mathbb{S}^{n-1}\times\mathbb{S}^{n-1}\times{\mathbb{R}^n}} 
|\tilde\eta|^{2s}\,\big|
{\mathcal{F}}v(\tilde\eta)\big|^2
\,d{\mathcal{H}}^{n-1}(\omega)
\,d{\mathcal{H}}^{n-1}(\tilde\omega)\,d\tilde\eta\right)^{\frac12}
\\ &=&
\left(
\int_{{\mathbb{R}^n}} 
|\eta|^{2s}\,\big|{\mathcal{F}}v(\eta)\big|^2
\,d\eta\right)^{\frac12}
\left(
\int_{\mathbb{R}^{n}} 
|\tilde\eta|^{2s}\,\big|
{\mathcal{F}}v(\tilde\eta)\big|^2
\,d\tilde\eta\right)^{\frac12}
\\&=&
\mathcal{E}_s\left(v,v\right).
\end{eqnarray*}
This proves~\eqref{ENblaoe1}.

Now, we prove~\eqref{ENblaoe2}. For this, we observe that
$$ \frac{\partial^\ell v_\sharp}{\partial x_{j_1}\dots\partial x_{j_\ell}}(x)=
\frac{1}{\left|\mathbb{S}^{n-1}\right|}\sum_{k_1,\dots,k_\ell=1}^n
\int_{\mathbb{S}^{n-1}} 
\frac{\partial^\ell v}{\partial x_{k_1}\dots\partial x_{k_\ell}}
({\mathcal{R}}_\omega\,x)\;{\mathcal{R}}_\omega^{k_1j_1}\dots{\mathcal{R}}_\omega^{k_\ell j_\ell}
\;d{\mathcal{H}}^{n-1}(\omega),$$
for every $\ell\in{\mathbb{N}}$ and~$j_1,\dots,j_\ell\in\{1,\dots,n\}$,
where~${\mathcal{R}}_\omega^{jk}$ denotes the $(j,k)$ component of the matrix~${\mathcal{R}}_\omega$.
In particular,
$$ \left|\frac{\partial^\ell v_\sharp}{\partial x_{j_1}\dots\partial x_{j_\ell}}(x)\right|\le
C\,\sum_{k_1,\dots,k_\ell=1}^n
\int_{\mathbb{S}^{n-1}} 
\left|\frac{\partial^\ell v}{\partial x_{k_1}\dots\partial x_{k_\ell}}
({\mathcal{R}}_\omega\,x)\right|
\;d{\mathcal{H}}^{n-1}(\omega),$$
for some~$C>0$ only depending on~$n$ and~$\ell$, and hence
\begin{eqnarray*} \left\|\frac{\partial^\ell v_\sharp}{\partial x_{j_1}\dots
\partial x_{j_\ell}}(x)\right\|_{L^2(B_1)}^2&\le&
C\,\sum_{k_1,\dots,k_\ell=1}^n
\iint_{\mathbb{S}^{n-1}\times B_1} 
\left|\frac{\partial^\ell v}{\partial x_{k_1}\dots\partial x_{k_\ell}}
({\mathcal{R}}_\omega\,x)\right|^2
\;d{\mathcal{H}}^{n-1}(\omega)\,dx\\&=&
C\,\sum_{k_1,\dots,k_\ell=1}^n
\iint_{\mathbb{S}^{n-1}\times B_1} 
\left|\frac{\partial^\ell v}{\partial x_{k_1}\dots\partial x_{k_\ell}}
(y)\right|^2
\;d{\mathcal{H}}^{n-1}(\omega)\,dy\\&=&
C\,\sum_{k_1,\dots,k_\ell=1}^n
\left\|\frac{\partial^\ell v}{\partial x_{k_1}\dots\partial x_{k_\ell}}
\right\|_{L^2(B_1)}^2,
\end{eqnarray*}
up to renaming~$C$.

This, together with~\eqref{energynorm} and~\eqref{ENblaoe1},
gives~\eqref{ENblaoe2}, as desired.
\end{proof}
\end{lemma}

With this preliminary work, we can now find a nontrivial,
nonnegative
and radial solution of~\eqref{dirfun}.

\begin{proposition}\label{2.52.5}
There exists a solution of~\eqref{dirfun} in~$H^s_0(B_1)$
which is radial, nonnegative
and with unit norm in~$L^2(B_1)$.

\begin{proof} Let~$u\ge0$ be a nontrivial solution of~\eqref{dirfun},
whose existence is warranted by Lemma~\ref{ikAHHPKAK}.

Then, we have that~$u_\sharp\ge0$.
Moreover,
\begin{eqnarray*}&& \int_{B_1}u_\sharp(x)\,dx=
\frac{1}{\left|\mathbb{S}^{n-1}\right|}
\iint_{\mathbb{S}^{n-1}\times B_1} u({\mathcal{R}}_\omega\,x)
\,d{\mathcal{H}}^{n-1}(\omega)\,dx\\&&\qquad=
\frac{1}{\left|\mathbb{S}^{n-1}\right|}
\iint_{\mathbb{S}^{n-1}\times B_1} u(y)
\,d{\mathcal{H}}^{n-1}(\omega)\,dy=
\int_{ B_1} u(y)
\,dy>0,
\end{eqnarray*}
and therefore~$u_\sharp$ does not vanish identically.

As a consequence, we can define
$$u_\star:=\frac{u_\sharp}{\|u_\sharp\|_{L^2(B_1)}}.$$
We know that~$u_\star\in H^s_0(B_1)$, due to~\eqref{ENblaoe2}.
Moreover, in view of Lemma~\ref{lapsfercom},
$$ (-\Delta)^s u_\star=
\frac{(-\Delta)^s u_\sharp}{\|u_\sharp\|_{L^2(B_1)}}=
\frac{\big((-\Delta)^s u\big)_\sharp}{\|u_\sharp\|_{L^2(B_1)}}
=\frac{\lambda_1\,u_\sharp}{\|u_\sharp\|_{L^2(B_1)}}
=\lambda_1\,u_\star,$$
which gives the desired result.
\end{proof}
\end{proposition}

Now, we are in the position of proving the following result.

\begin{lemma}
\label{onesob}
Let $s\ge1$ and~$r\in(0,1)$. If $u\in H_0^s\left(B_1\right)$ and $u$ is radial, then $u\in C^\alpha\left({\mathbb{R}}^n\setminus B_r\right)$ for any $\alpha\in\left[0,\frac{1}{2}\right]$.
\begin{proof}
We write
\begin{equation}
u\left(x\right)=u_0\left(\left|x\right|\right),\qquad\mbox{for some }
\;u_0:[0,+\infty)\rightarrow\mathbb{R}
\end{equation}
and we observe that~$u\in H_0^s\left(B_1\right)\subset H^1\left({\mathbb{R}}^n\right) $.

Accordingly, for any $0<r<1$, we have
\begin{equation}
\label{funz}
\infty>\int_{{\mathbb{R}}^n\setminus B_r} |u(x)|^2 \,dx
=\int_r^{+\infty} |u_0(\rho)|^2\rho^{n-1} \,d\rho\geq 
r^{n-1}\int_r^{+\infty}|u_0(\rho)|^2 \,d\rho
\end{equation}
and
\begin{equation}
\label{grad}
\infty>\int_{{\mathbb{R}}^n\setminus B_r} {|\nabla u(x)|^2 \,dx}=\int_r^{+\infty}
{\left|u_0' (\rho)\right|^2\rho^{n-1} \,d\rho}
\geq r^{n-1}\int_r^{+\infty}{\left|u_0' (\rho)\right|^2 \,d\rho}.
\end{equation}
Thanks to \eqref{funz} and \eqref{grad} we have that $u_0\in H^1\left(\left(r,+\infty\right)\right)$,
with~$u_0=0$ in~$[1,+\infty)$.

Then, from 
the Morrey Embedding Theorem, it follows that
$u_0\in C^\alpha\left(\left(r,+\infty\right)\right)$ for any
$\alpha\in\left[0,\frac{1}{2}\right]$, which leads to the desired result.
\end{proof}
\end{lemma}

\begin{corollary}\label{QUESTCP} Let~$s\in(0,+\infty)$.
There exists a radial, nonnegative
and nontrivial solution of \eqref{dirfun}
which belongs to~$H^s_0(B_1)\cap C^{\alpha}({\mathbb{R}}^n\setminus B_{1/2})$,
for some~$\alpha\in(0,1)$.

\begin{proof} If~$s\in(0,1)$, the desired claim follows from 
Corollary~8 in~\cite{DSV1}.

If instead~$s\ge1$, we obtain the desired result as a consequence of
Proposition \ref{2.52.5} and Lemma \ref{onesob}.
\end{proof}
\end{corollary}

\section{Boundary asymptotics of the first eigenfunctions of~$(-\Delta)^s$}
\label{sec5}

In Lemma 4 of~\cite{DSV1}, some precise asymptotics
at the boundary for the first Dirichlet eigenfunction
of~$(-\Delta)^s$ have been established in the range~$s\in(0,1)$.

Here, we obtain a related expansion in the range~$s>0$
for the eigenfunction provided in Corollary~\ref{QUESTCP}.
The result that we obtain is the following:

\begin{proposition}
\label{sharbou}
There exists a nontrivial solution $\phi_*$ of \eqref{dirfun}
which belongs to~$H^s_0(B_1)\cap C^{\alpha}
({\mathbb{R}}^n\setminus B_{1/2})$,
for some~$\alpha\in(0,1)$, and such that, for every~$e\in\partial B_1$
and~$\beta=(\beta_1,\dots,\beta_n)\in{\mathbb{N}}^n$,
\[\lim_{\epsilon\searrow 0}\epsilon^{\left|\beta\right|-s}\partial^\beta\phi_*\left(e+\epsilon X\right)=\left(-1\right)^{\left|\beta\right|}k_*\, s\left(s-1\right)\ldots\left(s-\left|\beta\right|+1\right)e_1^{\beta_1}\ldots e_n^{\beta_n}\left(-e\cdot X\right)_+^{s-\left|\beta\right|},\]
in the sense of distribution, with~$|\beta|:=\beta_1+\dots+\beta_n$
and~$k_*>0$.\end{proposition}

The proof of Proposition~\ref{sharbou} relies on Proposition~\ref{LEJOS}
and some auxiliary computations on the Green function in~\eqref{GREEN}.
We start with the following result:

\begin{lemma}
\label{lklkl}
Let $0<r<1$, $e\in\partial B_1$, $s>0$,
$f\in C^\alpha(\mathbb{R}^n\setminus B_r)\cap L^2(\mathbb{R}^n)$ 
for some $\alpha\in(0,1)$, and $f=0$ outside $B_1$. Then the integral
\begin{equation}
\label{I1I2}
\int_{B_1} f(z)\frac{(1-|z|^2)^s}{s|z-e|^n} \,dz
\end{equation}
is finite.
\begin{proof}
We denote by~$I$ the integral in~\eqref{I1I2}. We let
$$ I_1:=
\int_{B_1\setminus B_r} f(z)\frac{(1-|z|^2)^s}{s|z-e|^n} \,dz
\qquad{\mbox{and}}\qquad I_2:=
\int_{B_r} f(z)\frac{(1-|z|^2)^s}{s|z-e|^n} \,dz.$$
Then, we have that
\begin{equation}\label{SsalKAM:1} I=I_1+I_2.\end{equation}
Now, if $z\in B_1\setminus B_r$, we have that
\begin{equation*}
f(z)\leq|f(z)-f(e)|\leq C|z-e|^\alpha,
\end{equation*}
therefore
\begin{equation}\label{SsalKAM:2}
I_1\leq\int_{B_1\setminus B_r}\frac{(1-|z|^2)^s}{s|z-e|^{n-\alpha}} \,dz<\infty.
\end{equation}
If instead~$z\in B_r$,
\begin{equation*}
|z-e|\geq 1-r>0,
\end{equation*}
and consequently
\begin{equation}\label{SsalKAM:3}
I_2\leq \frac{1}{s\,(1-r)^n}\int_{B_r}f(z)\, dz<\infty.
\end{equation}
The desired result follows from~\eqref{SsalKAM:1},
\eqref{SsalKAM:2} and~\eqref{SsalKAM:3}.
\end{proof}
\end{lemma}

Next result gives
a precise boundary behaviour of the Green function
for any $s>0$ (the case in which~$s\in(0,1)$ and~$f\in
C^\alpha(\mathbb{R}^n)$ was considered in Lemma~6 of~\cite{DSV1},
and in fact the proof presented here also simplifies the one in
Lemma~6 of~\cite{DSV1} for the setting considered there).

\begin{lemma}
\label{lemsix}
Let $e$, $\omega\in\partial B_1$, $\epsilon_0>0$ and~$r\in(0,1)$. 
Assume that
\begin{equation}\label{CHIAmahdfn}
e+\epsilon\omega\in B_1,\end{equation}
for any $\epsilon\in(0,\epsilon_0]$. Let $f\in C^\alpha(\mathbb{R}^n\setminus B_r)\cap L^2(\mathbb{R}^n)$ for some $\alpha\in(0,1)$, with $f=0$ outside $B_1$. \\
Then
\begin{equation}
\lim_{\epsilon\searrow 0}\epsilon^{-s}\int_{B_1} f(z)
\mathcal{G}_s(e+\epsilon\omega,z) \,dz=k(n,s)\,
\int_{B_1} f(z)\frac{(-2e\cdot\omega)^s(1-|z|^2)^s}{s|z-e|^n}\, dz,
\end{equation}
for a suitable normalizing constant~$k(n,s)>0$.

\begin{proof} In light of~\eqref{CHIAmahdfn}, we have that
\begin{equation*}1>
|e+\epsilon\omega|^2=1+\epsilon^2+2\epsilon e\cdot\omega,
\end{equation*}
and therefore
\begin{equation}\label{RGHHCicj}
-e\cdot\omega>\frac{\epsilon}{2}>0.
\end{equation}
Moreover, if $r_0$ is as given in \eqref{GREEN}, we have that, for all~$z\in B_1$,
\begin{equation}
\label{kfg}
r_0(e+\epsilon\omega,z)=\frac{\epsilon(-\epsilon-2e\cdot\omega)(1-|z|^2)}{|z-e-\epsilon\omega|^2}\leq\frac{3\epsilon}{|z-e-\epsilon\omega|^2}.
\end{equation}
Also, a Taylor series representation allows us to write, for any~$t\in(-1,1)$,
\begin{equation}\label{7uJAJMMA.aA}
\frac{t^{s-1}}{(t+1)^{\frac{n}{2}}}=\sum_{k=0}^\infty \binom{-n/2}{k}t^{k+s-1}.
\end{equation}
We also notice that
\begin{equation}
\label{bound}\begin{split}&
\left|\binom{-n/2}{k}\right|=
\left|
\frac{-\frac{n}2
\left(-\frac{n}2-1\right)\,...\,\left(-\frac{n}2-k+1\right)
}{k!}
\right|
=
\frac{\frac{n}2
\left(\frac{n}2+1\right)\,...\,\left(\frac{n}2+(k-1)\right)
}{k!}\\
&\qquad\le\frac{n
\left(n+1\right)\,...\,\left(n+(k-1)\right)
}{k!}
\le\frac{\left(n+(k-1)\right)!
}{k!}= (k+1)\,...\,\left(n+(k-1)\right)\\&\qquad
\le (n+k+1)^{n+1}.
\end{split}\end{equation}
This and the Root Test
give that the series in~\eqref{7uJAJMMA.aA}
is uniformly convergent on compact sets in $(-1,1)$.

As a consequence, if we set
\begin{equation}
\label{min}
r_1(x,z):=\min\left\{\frac{1}{2},r_0(x,z)\right\},
\end{equation}
we can switch integration and summation signs and obtain that
\begin{equation}
\label{mmno}
\int_0^{r_1(x,z)} \frac{t^{s-1}}{(t+1)^{\frac{n}{2}}}\, dt=\sum_{k=0}^\infty c_k(r_1(x,z))^{k+s},
\end{equation}
where
\begin{equation*}
c_k:=\frac{1}{k+s}\binom{-n/2}{k}.
\end{equation*}
Once again, the bound in~\eqref{bound}, together with~\eqref{min},
give that the series in~\eqref{mmno} is convergent.

Now, we omit for simplicity the normalizing constant~$k(n,s)$
in the definition of the Green function in~\eqref{GREEN},
and
we define
\begin{equation}\label{GNAKDDEF}
\mathcal{G}(x,z):=|z-x|^{2s-n}\sum_{k=0}^\infty c_k(r_1(x,z))^{k+s}
\end{equation}
and
\begin{equation*}
g(x,z):=|z-x|^{2s-n}\int_{r_1(x,z)}^{r_0(x,z)} \frac{t^{s-1}}{(t+1)^{\frac{n}{2}}} \,dt.
\end{equation*}
Using~\eqref{GREEN}
and \eqref{mmno}, and dropping dimensional constants for the sake of shortness,
we can write
\begin{equation}
\label{splitgreen}
\mathcal{G}_s(x,z)=\mathcal{G}(x,z)+g(x,z).
\end{equation}
Now, we show that
\begin{equation}\label{VKLACL}
g(x,z)\leq
\begin{cases} C\chi(r,z)\,|z-x|^{2s-n}&\quad\text{if}\quad n>2s, \\ 
C\chi(r,z)\,\log r_0(x,z)&\quad\text{if}\quad n=2s, \\
C\chi(r,z)\,|z-x|^{2s-n}(r_0(x,z))^{s-\frac{n}{2}}&\quad\text{if}\quad n<2s, \end{cases}
\end{equation}
where~$\chi(r,z)=1$ if~$r_0(x,z)> \frac{1}{2}$
and~$\chi(r,z)=0$ if~$r_0(x,z)\leq \frac{1}{2}$.
To check this,
we notice that if~$r_0(x,z)\leq \frac{1}{2}$ we have that~$r_1(x,z)=r_0(x,z)$,
due to~\eqref{min}, and therefore~$g(x,z)=0$. 

On the other hand, if $r_0(x,z)>\frac{1}{2}$, 
we deduce from~\eqref{min} that~$r_1(x,z)=\frac12$, and consequently
\begin{equation*}
g(x,z)\leq |z-x|^{2s-n}\int_{1/2}^{r_0(x,z)} t^{s-\frac{n}{2}-1} dt\leq
\begin{cases} C|z-x|^{2s-n}&\quad\text{if}\quad n>2s, \\ 
C\log r_0(x,z)&\quad\text{if}\quad n=2s, \\
C|z-x|^{2s-n}(r_0(x,z))^{s-\frac{n}{2}}&\quad\text{if}\quad n<2s, \end{cases}
\end{equation*}
for some constant $C>0$. This completes the proof of~\eqref{VKLACL}.

Now, we exploit the bound in~\eqref{VKLACL} when $x=e+\epsilon\omega$. 
For this, we notice that 
if~$r_0(e+\epsilon\omega,z)>\frac12$,
recalling \eqref{kfg}, we find that
\begin{equation}\label{75g8676769}
|z-e-\epsilon\omega|^2\leq 6\epsilon<9\epsilon,
\end{equation}
and therefore $z\in B_{3\sqrt{\epsilon}}(e+\epsilon\omega)$.

Hence, using~\eqref{VKLACL},
\begin{equation}
\label{estimates1}
\begin{split}
&\left|\int_{B_1} f(z)g(e+\epsilon\omega,z) dz\right|\leq 
\int_{B_{3\sqrt{\epsilon}}(e+\epsilon\omega)}
|f(z)||g(e+\epsilon\omega,z)| dz \\ &
\leq\begin{cases} C\displaystyle\int_{B_{3\sqrt{\epsilon}}(e+\epsilon\omega)}
|f(z)||z-e-\epsilon\omega|^{2s-n} dz&\quad\text{if}\quad n>2s,\\
C\displaystyle\int_{B_{3\sqrt{\epsilon}}(e+\epsilon\omega)}|f(z)|\log r_0(
e+\epsilon\omega,z) dz&\quad\text{if}\quad n=2s, \\
C\displaystyle\int_{B_{3\sqrt{\epsilon}}(e+\epsilon\omega)}
|f(z)||z-e-\epsilon\omega|^{2s-n}(r_0(e+\epsilon\omega,z))^{
s-\frac{n}{2}} dz &\quad\text{if}\quad n<2s .\end{cases}
\end{split}
\end{equation}
Now, if $z\in B_{3\sqrt{\epsilon}}(e+\epsilon\omega)$,
then \begin{equation}\label{z0okdscxi}
|z-e|\leq |z-e-\epsilon\omega|+|\epsilon\omega|
\leq 3\sqrt{\epsilon}+\epsilon<4\sqrt{\epsilon}.\end{equation}
Furthermore,
for a given~$r\in(0,1)$, we have that~$
B_{3\sqrt{\epsilon}}(e+\epsilon\omega)\subseteq{\mathbb{R}}^n\setminus
B_r$, provided that~$\epsilon$ is sufficiently small.

Hence, if~$z\in B_{3\sqrt{\epsilon}}(e+\epsilon\omega)$,
we can exploit the regularity of~$f$ and deduce that
$$ |f(z)|=|f(z)-f(e)|\leq C|z-e|^\alpha.$$
This and~\eqref{z0okdscxi} lead to
\begin{equation}
\label{su}
|f(z)|\leq C\epsilon^{\frac{\alpha}{2}},
\end{equation}
for every~$z\in B_{3\sqrt{\epsilon}}(e+\epsilon\omega)$.

Thanks to \eqref{kfg}, \eqref{estimates1} and \eqref{su}, we have that
\begin{equation*}
\begin{split}
\left|\int_{B_1} f(z)g(e+\epsilon\omega,z) dz\right|&\leq\begin{cases} C\epsilon^{\frac{\alpha}{2}}
\displaystyle\int_{B_{3\sqrt{\epsilon}}(e+\epsilon\omega)}|z-e-\epsilon\omega|^{2s-n} dz&\quad\text{if}\quad n>2s,\\ C\epsilon^{\frac{\alpha}{2}}
\displaystyle\int_{B_{3\sqrt{\epsilon}}(e+\epsilon\omega)}\log \frac{3\epsilon}{|z-e-\epsilon\omega|^2} dz&\quad\text{if}\quad n=2s ,\\ C\epsilon^{\frac{\alpha}{2}+s-\frac{n}{2}}\displaystyle\int_{B_{3\sqrt{\epsilon}}(e+\epsilon\omega)} dz &\quad\text{if}\quad n<2s \end{cases} \\
&\leq C\epsilon^{\frac{\alpha}{2}+s},
\end{split}
\end{equation*}
up to renaming~$C$.

This and \eqref{splitgreen} give that
\begin{equation}
\label{alalal}
\int_{B_1} f(z)\mathcal{G}_s(e+\epsilon\omega,z) dz=\int_{B_1} f(z)\mathcal{G}(e+\epsilon\omega,z) dz+o(\epsilon^s).
\end{equation}
Now, we consider the series in~\eqref{GNAKDDEF},
and we split the contribution coming from the index $k=0$ from the ones coming from the indices $k>0$, namely
we write
\begin{equation}\label{IN16p}
\begin{split}
&\mathcal{G}(x,z)=\mathcal{G}_0(x,z)+\mathcal{G}_1(x,z), \\
\quad\text{with}\quad &\mathcal{G}_0(x,z):=\frac{|z-x|^{2s-n}}{s}(r_1(x,z))^s \\
\quad\text{and}\quad &\mathcal{G}_1(x,z):=|z-x|^{2s-n}\sum_{k=1}^{+\infty} c_k\,(r_1(x,z))^{k+s}.
\end{split}
\end{equation}
Firstly, we consider the contribution given by the term $\mathcal{G}_1$. Thanks to \eqref{min} and \eqref{su}, we have that
\begin{equation}
\label{estimates5}
\begin{split}
&\left|\int_{B_1\cap B_{3\sqrt{\epsilon}}(e+\epsilon\omega)}f(z)\mathcal{G}_1(e+\epsilon\omega,z) dz\right|\leq \int_{B_{3\sqrt{\epsilon}}(e+\epsilon\omega)} |f(z)|\mathcal{G}_1(e+\epsilon\omega,z) dz \\
&\quad\quad\leq C\epsilon^{\frac{\alpha}{2}}\int_{B_{3\sqrt{\epsilon}}(e+\epsilon\omega)}|z-e-\epsilon\omega|^{2s-n}\sum_{k=1}^{+\infty} |c_k|\,(r_1(e+\epsilon\omega,z))^{k+s} dz \\
&\quad\quad\leq C\epsilon^{\frac{\alpha}{2}}\int_{B_{3\sqrt{\epsilon}}(e+\epsilon\omega)}|z-e-\epsilon\omega|^{2s-n}\sum_{k=1}^{+\infty} |c_k|\,\left(\frac{1}{2}\right)^{k+s} dz \\
&\quad\quad\leq C\epsilon^{\frac{\alpha}{2}}\int_{B_{3\sqrt{\epsilon}}(e+\epsilon\omega)}|z-e-\epsilon\omega|^{2s-n} dz \\
&\quad\quad\leq C\epsilon^{\frac{\alpha}{2}+s},
\end{split}
\end{equation}
up to renaming the constant $C$ step by step. 

On the other hand, for every~$z\in{\mathbb{R}}^n$,
\[|z|=|e+\epsilon\omega+z-e-\epsilon\omega|
\geq|e+\epsilon\omega|-|z-e-\epsilon\omega|
\geq 1-\epsilon-|z-e-\epsilon\omega|.\] 
Therefore,
for every~$z\in B_1\setminus\left(B_r\cup B_{3\sqrt{\epsilon}}
(e+\epsilon\omega)\right)$, we can take~$e_*:=\frac{z}{|z|}$ and obtain that
\begin{equation}
\label{estimate2}
\begin{split}&
|f(z)|=|f(z)-f(e_*)|\leq
C|z-e_*|^\alpha= C(1-|z|)^\alpha
\\&\qquad\leq C(\epsilon+|z-e-\epsilon\omega|)^\alpha
\leq C|z-e-\epsilon\omega|^\alpha,
\end{split}\end{equation}
up to renaming~$C>0$.

Also, using \eqref{kfg}, we see that, for any $k>0$,
\begin{equation}
\label{estimates3}
\begin{split}
%% &(r_1(e+\epsilon\omega,z))^{k+s}=(r_1(e+\epsilon\omega,z))^{s+\frac{\alpha}{4}}(r_1(e+\epsilon\omega,z))^{k-\frac{\alpha}{4}} \\
%% {\mbox{and }}\quad\quad
&(r_0(e+\epsilon\omega,z))^{s+\frac{\alpha}{4}}\left(\frac{1}{2}\right)^{k-\frac{\alpha}{4}}\leq\frac{C\epsilon^{s+\frac{\alpha}{4}}}{2^k|z-e-\epsilon\omega|^{2s+\frac{\alpha}{2}}}.
\end{split}
\end{equation}
This, \eqref{min} and \eqref{estimate2} give that if $z\in B_1\setminus\left(B_r\cup B_{3\sqrt{\epsilon}}(e+\epsilon\omega)\right)$, then
\begin{equation*}
\begin{split}
|f(z)\mathcal{G}_1(e+\epsilon\omega,z)| \,&
\leq C|z-e-\epsilon\omega|^{\alpha+2s-n}
\sum_{k=1}^{+\infty}|c_k|\,(r_1(e+\epsilon\omega,z))^{k+s} \\
&=C|z-e-\epsilon\omega|^{\alpha+2s-n}
\sum_{k=1}^{+\infty}|c_k|\,(r_1(e+\epsilon\omega,z))^{s+\frac\alpha4}
(r_1(e+\epsilon\omega,z))^{k-\frac\alpha4}\\
&\le C|z-e-\epsilon\omega|^{\alpha+2s-n}
\sum_{k=1}^{+\infty}|c_k|\,(r_0(e+\epsilon\omega,z))^{s+\frac\alpha4}
\left(\frac12\right)^{k-\frac\alpha4}\\
&\leq C\epsilon^{s+\frac{\alpha}{4}}|z-e-\epsilon\omega|^{\frac{\alpha}{2}-n}\sum_{k=1}^{+\infty}\frac{|c_k|}{2^k},
\end{split}
\end{equation*}
where the latter series is absolutely convergent thanks to \eqref{bound}. 

This implies that, if we
set $E:=B_1\setminus\left(B_r\cup B_{3\sqrt{\epsilon}}
(e+\epsilon\omega)\right)$, it holds that
\begin{equation}
\label{estimates4}
\begin{split}
&\left|\int_E f(z)\mathcal{G}_1(e+\epsilon\omega,z) dz\right|\leq C\epsilon^{s+\frac{\alpha}{4}}\int_E |z-e-\epsilon\omega|^{\frac{\alpha}{2}-n} dz \\
\quad\quad &\qquad\qquad\leq C\epsilon^{s+\frac{\alpha}{4}}\int_{B_1} |z-e-\epsilon\omega|^{\frac{\alpha}{2}-n} dz\leq C\epsilon^{s+\frac{\alpha}{4}}\int_{B_3} |z|^{\frac{\alpha}{2}-n} dz \leq C\epsilon^{s+\frac{\alpha}{4}}.
\end{split}
\end{equation}
Moreover, if $z\in B_r$, we have that
\begin{equation*}
|e+\epsilon\omega-z|\geq 1-\epsilon-r,
\end{equation*}
and therefore, recalling~\eqref{estimates3},
\begin{equation*}
\begin{split}
\sup_{z\in B_r} |\mathcal{G}_1(e+\epsilon\omega,z)|\,&\leq |z-e-\epsilon\omega|^{2s-n}\sum_{k=1}^{+\infty}
|c_k|\,\big(r_1(e+\epsilon\omega,z)\big)^{s+\frac\alpha4}
\big(r_1(e+\epsilon\omega,z)\big)^{k-\frac\alpha4}\\&\le
|z-e-\epsilon\omega|^{2s-n}\sum_{k=1}^{+\infty}
|c_k|\,\big(r_0(e+\epsilon\omega,z)\big)^{s+\frac\alpha4}
\left(\frac12\right)^{k-\frac\alpha4}
\\ &\le C\,
|z-e-\epsilon\omega|^{-n-\frac\alpha2}\sum_{k=1}^{+\infty}
\frac{|c_k|}{2^k}
\\ &\le
C(1-\epsilon-r)^{-n-\frac{\alpha}{2}}\,
\epsilon^{s+\frac{\alpha}{4}},
\end{split}\end{equation*}
up to renaming~$C$.

As a consequence, we find that
\begin{equation}
\label{estimates6}
\begin{split}
\left|\int_{B_r} f(z)\mathcal{G}_1(e+\epsilon\omega,z) dz\right|
&\leq 
\sup_{z\in B_r} |\mathcal{G}_1
(e+\epsilon\omega,z)|\,
\left\|f\right\|_{L^1(B_r)}
\\ &\leq \left\|f\right\|_{L^1(B_r)}(1-\epsilon-r)^{-n-\frac{\alpha}{2}}\epsilon^{s+\frac{\alpha}{4}}
\\
&\leq \left\|f\right\|_{L^1(B_r)}2^{n+\frac{\alpha}{2}}(1-r)^{-n-\frac{\alpha}{2}}\epsilon^{s+\frac{\alpha}{4}}
\\
&=C\epsilon^{s+\frac{\alpha}{4}},
\end{split}
\end{equation}
as long as $\epsilon$ is suitably
small with respect to $r$, and $C$ is a positive constant
which depends on $\|f\|_{L^1(B_r)}$, $r$, $n$ and~$\alpha$.

Then, by \eqref{estimates5}, \eqref{estimates4} and \eqref{estimates6} we conclude that
\begin{equation}
\int_{B_1} f(z)\mathcal{G}_1(e+\epsilon\omega,z) dz=o(\epsilon^s).
\end{equation}
Inserting this information into \eqref{alalal}, 
and recalling~\eqref{IN16p},
we obtain
\begin{equation}
\label{ololol}
\int_{B_1} f(z)\mathcal{G}_s(e+\epsilon\omega,z) dz=\int_{B_1} f(z)\mathcal{G}_0(e+\epsilon\omega,z) dz+o(\epsilon^s).
\end{equation}
Now, we define \[\mathcal{D}_1:=\left\{z\in B_1\quad\text{s.t.}\quad r_0(e+\epsilon\omega,z)>1/2\right\}\] and \[\mathcal{D}_2:=\left\{z\in B_1\quad\text{s.t.}\quad r_0(e+\epsilon\omega,z)\leq 1/2\right\}.\]
If $z\in\mathcal{D}_1$, then $z\in B_1\setminus B_r$, thanks to~\eqref{75g8676769},
and hence we can use \eqref{estimates1} and \eqref{su} and write 
\[|f(z)\mathcal{G}_0(e+\epsilon\omega,z)|\leq 
C\epsilon^{\frac{\alpha}{2}}|z-e-\epsilon\omega|^{2s-n}.\] 
Then, recalling again \eqref{estimates1},
\begin{equation}
\label{D1}
\left|\int_{\mathcal{D}_1} f(z)\mathcal{G}_1(e+\epsilon\omega,z) dz\right|\leq C\epsilon^{\frac{\alpha}{2}}\int_{B_{3\sqrt{\epsilon}}(e+\epsilon\omega)} |z-e-\epsilon\omega|^{2s-n} dz=C\epsilon^{\frac{\alpha}{2}+s},
\end{equation}
up to renaming the constant $C>0$. This information and \eqref{ololol} give that
\begin{equation*}
\int_{B_1} f(z)\mathcal{G}_s(e+\epsilon\omega,z) dz=
\int_{\mathcal{D}_2} f(z)\mathcal{G}_0(e+\epsilon\omega,z) dz+o(\epsilon^s).
\end{equation*}
Now, by \eqref{kfg} and \eqref{min}, if $z\in\mathcal{D}_2$,
\begin{equation*}
\mathcal{G}_0(e+\epsilon\omega,z)=\frac{|z-e-\epsilon\omega|^{2s-n}}{s}(r_0(e+\epsilon\omega))^s=\frac{\epsilon^s(-\epsilon-2e\cdot\omega)^s(1-|z|^2)^s}{s|z-e-\epsilon\omega|^n}.
\end{equation*}
Hence, we have
\begin{equation}
\label{sesa}
\begin{split}
&\lim_{\epsilon\searrow 0}\epsilon^{-s}\int_{B_1} f(z)\mathcal{G}_s(e+\epsilon\omega,z) dz \\
=\;&\lim_{\epsilon\searrow 0}\epsilon^{-s}\int_{\mathcal{D}_2} f(z)\mathcal{G}_0(e+\epsilon\omega,z) dz
\\
=\;&\lim_{\epsilon\searrow 0} \int_{\left\{2\epsilon(-\epsilon-2e\cdot\omega)(1-|z|^2)\leq|z-e-\epsilon\omega|^2\right\}} f(z)\frac{(-\epsilon-2e\cdot\omega)^s(1-|z|^2)^s}{s|z-e-\epsilon\omega|^n} dz.
\end{split}
\end{equation}
Now we set
\begin{equation}\label{H7uJA78JsadA}
F_\epsilon(z):=\begin{cases}f(z)
\displaystyle\frac{(-\epsilon-2e\cdot\omega)^s(1-|z|^2)^s}{
s|z-e-\epsilon\omega|^n}&\quad\text{if}\quad 2\epsilon
(-\epsilon-2e\cdot\omega)(1-|z|^2)\leq|z-e-\epsilon\omega|^2, \\
0&\quad\text{otherwise}, \end{cases}
\end{equation}
and we prove that for any $\eta>0$ there exists $\delta>0$ independent
of~$\epsilon$ such that, for any $E\subset\mathbb{R}^n$ with $|E|\leq\delta$, we have
\begin{equation}
\label{eta}
\int_{B_1\cap E} |F_\epsilon(z)| dz\leq\eta.
\end{equation}
To this aim, given~$\eta$ and~$E$ as above,
we define
\begin{equation} \label{RGANrho}
\rho:= \min\left\{
\epsilon (-\epsilon-2e\cdot\omega),\,
\sqrt{{ 2\epsilon (-\epsilon-2e\cdot\omega)}(1-r)},\,
\left(
\frac{
2^{s+\alpha} s^2\,\epsilon^{s+\alpha}\,(-\epsilon-2e\cdot\omega)^\alpha
\eta}{3^{2s}\,\|f\|_{C^\alpha(B_1\setminus B_r)}\,|\partial B_1|}
\right)^{\frac1{2\alpha}}
\right\}.\end{equation}
We stress that the above definition is well-posed, thanks to~\eqref{RGHHCicj}.
In addition, using the integrability of~$f$, we take~$\delta>0$
such that if~$A\subseteq B_1$ and~$|A|\le\delta$ then
\begin{equation}\label{PPKAjnaOP} \int_{A} |f(x)|\,dx\le \frac{s\rho^n\eta}{2\cdot 3^s}.
\end{equation}
We set 
\begin{equation}\label{9ikjendE} E_1:=E\cap B_{\rho}(e+\epsilon\omega)\qquad{\mbox{and}}\qquad
E_2:=E\setminus B_{\rho}(e+\epsilon\omega).\end{equation}
{F}rom~\eqref{H7uJA78JsadA}, we see that
$$ |F_\epsilon(z)|\le 
\frac{|f(z)|\,\chi_\star(z)}{2^s s\,\epsilon^s
|z-e-\epsilon\omega|^{n-2s}},$$
where
$$ \chi_\star(z):=\begin{cases}
1 & \quad\text{if}\quad 2\epsilon
(-\epsilon-2e\cdot\omega)(1-|z|^2)\leq|z-e-\epsilon\omega|^2, \\
0&\quad\text{otherwise},
\end{cases}$$
and therefore
\begin{equation}\label{THAnaoa}
\int_{B_1\cap E_1} |F_\epsilon(z)|\,dz
\le \int_{B_1\cap E_1}\frac{|f(z)|\,\chi_\star(z)}{2^s s\,\epsilon^s
|z-e-\epsilon\omega|^{n-2s}}\,dz.
\end{equation}
Now, for every~$z\in B_1\cap E_1\subseteq B_{\rho}(e+\epsilon\omega)$ for which~$\chi_\star(z)\ne0$,
we have that
$$ 2\epsilon
(-\epsilon-2e\cdot\omega)(1-|z|^2)\le|z-e-\epsilon\omega|^2\le\rho^2,$$
and hence
$$ |z|\ge \sqrt{1-\frac{\rho^2}{ 2\epsilon (-\epsilon-2e\cdot\omega)}}
\ge 1-\frac{\rho^2}{ 2\epsilon (-\epsilon-2e\cdot\omega)},$$
which in turn gives that~$|z|\ge r$, recall~\eqref{RGANrho}.

{F}rom this and~\eqref{THAnaoa} we deduce that
\begin{equation}\label{9ikjendE2}
\begin{split}&
\int_{B_1\cap E_1} |F_\epsilon(z)|\,dz
\le \int_{1-\frac{\rho^2}{ 2\epsilon (-\epsilon-2e\cdot\omega)}\le|z|<1
}\frac{\|f\|_{C^\alpha(B_1\setminus B_r)}\,(1-|z|)^\alpha}{2^s s\,\epsilon^s
|z-e-\epsilon\omega|^{n-2s}}\,dz\\&\qquad
\le \frac{\|f\|_{C^\alpha(B_1\setminus B_r)}}{2^s s\,\epsilon^s}\,
\left(\frac{\rho^2}{ 2\epsilon (-\epsilon-2e\cdot\omega)}\right)^\alpha
\int_{1-\frac{\rho^2}{ 2\epsilon (-\epsilon-2e\cdot\omega)}\le|z|<1
}\frac{dz}{|z-e-\epsilon\omega|^{n-2s}}\\&\qquad
\le \frac{\|f\|_{C^\alpha(B_1\setminus B_r)}}{2^s s\,\epsilon^s}\,
\left(\frac{\rho^2}{ 2\epsilon (-\epsilon-2e\cdot\omega)}\right)^\alpha
\int_{B_3}\frac{dx}{|x|^{n-2s}}
\\&\qquad=\frac{3^{2s}\,\|f\|_{C^\alpha(B_1\setminus B_r)}\,
|\partial B_1|}{
2^{s+\alpha+1} s^2\,\epsilon^{s+\alpha}\,(-\epsilon-2e\cdot\omega)^\alpha }\,
\;\rho^{2\alpha}\\&\qquad\le
\frac{\eta}{2},
\end{split}\end{equation}
where~\eqref{RGANrho} has been exploited in the last inequality.

We also point out that, by~\eqref{H7uJA78JsadA},
\eqref{PPKAjnaOP} and~\eqref{9ikjendE},
\begin{eqnarray*}
\int_{B_1\cap E_2}|F_\epsilon(z)|\,dz
&\le&\int_{(B_1\setminus
B_{\rho}(e+\epsilon\omega))\cap E}
|f(z)|\,\frac{(-\epsilon-2e\cdot\omega)^s(1-|z|^2)^s}{
s|z-e-\epsilon\omega|^n}\,dz\\
&\le&\frac{3^s}{s\rho^n}\int_{B_1\cap E}|f(z)|\,dz\\&\le&\frac{\eta}{2}.
\end{eqnarray*}
This, \eqref{9ikjendE} and~\eqref{9ikjendE2} give~\eqref{eta},
as desired.

Notice also that the sequence $F_\epsilon(z)$ converges pointwise
to the function \[F(z):=f(z)\frac{(-2e\cdot\omega)^s(1-|z|^2)^s}{s|z-e|^n}.\]
Hence \eqref{sesa}, \eqref{eta} and the Vitali Convergence Theorem allow us to conclude that
\begin{equation}\label{lkl5678kl}
\begin{split}
\lim_{\epsilon\searrow 0}\int_{B_1} f(z)\mathcal{G}_s(e+\epsilon\omega,z) dz&=\lim_{\epsilon\searrow 0}\int_{B_1} F_\epsilon(z) dz \\
&=\int_{B_1} f(z)\frac{(-2e\cdot\omega)^s(1-|z|^2)^s}{s|z-e|^n} dz,
\end{split}
\end{equation}
which establishes the claim of Lemma \ref{lemsix}
(notice that the finiteness of the latter quantity in~\eqref{lkl5678kl}
follows from~\eqref{lklkl}).
\end{proof}
\end{lemma}

With this preliminary work,
we can now establish the boundary behaviour of solutions which is needed
in our setting. As a matter of fact, from Lemma~\ref{lemsix} we immediately deduce that:

\begin{corollary}
\label{propsev}
Let $e$, $\omega\in\partial B_1$, $\epsilon_0>0$
and~$r\in(0,1)$. 

Assume that $e+\epsilon\omega\in B_1$, for any $\epsilon\in(0,\epsilon_0]$. Let $f\in C^\alpha(\mathbb{R}^n\setminus B_r)\cap L^2(\mathbb{R}^n)$ for some $\alpha\in(0,1)$, with $f=0$ outside $B_1$. 

Let $u$ be as in~\eqref{0olwsKA}.
Then,
\begin{equation*}
\lim_{\epsilon\searrow 0}\epsilon^{-s}u(e+\epsilon\omega)=k(n,s)(-2e\cdot\omega)^s\int_{B_1} f(z)\frac{(1-|z|^2)^s}{s|z-e|^n} dz,
\end{equation*}
where $k(n,s)$ denotes a positive normalizing constant.
\end{corollary}

Now we apply the previous results to detect the
boundary growth of a suitable
first eigenfunction. For our purposes, the statement that we need
is the following:

\begin{corollary}\label{8iJJAUMPAAAxc}
There exists a nontrivial solution $\phi_*$ of \eqref{dirfun}
which belongs to~$H^s_0(B_1)\cap C^{\alpha}
({\mathbb{R}}^n\setminus B_{1/2})$,
for some~$\alpha\in(0,1)$, and such that, for every~$e\in\partial B_1$,
\begin{equation}\label{CHaert}
\lim_{\epsilon\searrow 0}\epsilon^{-s}\phi_*(e+\epsilon\omega)=k_*\,
(-e\cdot\omega)^s_+,
\end{equation}
for a suitable constant~$k_*>0$.

Furthermore,
for every $R\in(r,1)$, there exists~$C_R>0$ such that
\begin{equation}\label{COSI}
\sup_{x\in B_1\setminus B_R}
d^{-s}(x)\,|\phi_*(x)|\le C_R.\end{equation}

\begin{proof}
Let~$\alpha\in(0,1)$ and~$\phi\in H^s_0(B_1)\cap C^{\alpha}({\mathbb{R}}^n\setminus B_{1/2})$
be the nonnegative
and nontrivial solution of \eqref{dirfun}, as given in Corollary~\ref{QUESTCP}.

In the spirit of~\eqref{0olwsKA},
we define
$$
\phi_*(x):=
\begin{cases}
\displaystyle\lambda_1\int_{B_1} \mathcal{G}_s\left(x,y\right)\,\phi(y)\,dy & {\mbox{ if }}x\in B_1,\\
0&{\mbox{ if }}x\in{\mathbb{R}}^n\setminus B_1.
\end{cases}$$
We stress that we can use Proposition~\ref{LEJOS}
in this context, with~$f:=\lambda_1\phi$,
since condition~\eqref{CHlaIA} is satisfied in this case.

Then, from~\eqref{VIC2} and~\eqref{VIC4}, we know that~$\phi_*\in H^s_0(B_1)$
and, from~\eqref{VIC3},
$$ (-\Delta)^s \phi_*=\lambda_1\,\phi{\mbox{ in }}B_1.$$
In particular, we have that~$(-\Delta)^s (\phi-\phi_*)=0$ in~$B_1$,
and~$\phi-\phi_*\in H^s_0(B_1)$, which give that~$\phi-\phi_*$ vanishes identically.
Hence, we can write that~$\phi=\phi_*$, and thus~$\phi_*$
is a solution of~\eqref{dirfun}.

Now, we check~\eqref{CHaert}.
For this, we distinguish two cases.
If~$e\cdot\omega\ge 0$, we have that
$$ |e+\epsilon\omega|^2 =1+2\epsilon e\cdot\omega+\epsilon^2>1,$$
for all~$\epsilon>0$. Then, in this case~$e+\epsilon\omega\in
{\mathbb{R}}^n\setminus B_1$, and therefore~$\phi_*(e+\epsilon\omega)=0$.
This gives that, in this case,
\begin{equation}\label{GIANDIACNKS}
\lim_{\epsilon\searrow 0}\epsilon^{-s}\phi_*(e+\epsilon\omega)=0.\end{equation}
If instead~$e\cdot\omega<0$, we see that
$$ |e+\epsilon\omega|^2 =1+2\epsilon e\cdot\omega+\epsilon^2<1,$$
for all~$\epsilon>0$ sufficiently small. Hence, we can exploit
Corollary~\ref{propsev} and find that
\begin{equation}\label{GIANDIACNKS2}
\lim_{\epsilon\searrow 0}\epsilon^{-s}\phi_*(e+\epsilon\omega)=\lambda_1\,
k(n,s)(-2e\cdot\omega)^s\int_{B_1} \phi(z)\frac{(1-|z|^2)^s}{s|z-e|^n} \,dz,\end{equation}
with~$k(n,s)>0$. Then, we define
$$ k_*:=2^s\,k(n,s)\int_{B_1} \phi(z)\frac{(1-|z|^2)^s}{s|z-e|^n} \,dz.$$
We observe that~$k_*$ is positive by construction,
with~$k(n,s)>0$. Also, in light of
Lemma~\ref{lklkl}, we know that $k_*$ is finite. 
Hence, from~\eqref{GIANDIACNKS} and~\eqref{GIANDIACNKS2}
we obtain~\eqref{CHaert}, as desired.

It only remains to check~\eqref{COSI}.
For this, we use~\eqref{VIC3}, and we see that,
for every~$R\in(r,1)$, 
$$ \sup_{x\in B_1\setminus B_R}
d^{-s}(x)\,|\phi_*(x)|\le C_R\,\lambda_1\,\big(\|\phi\|_{L^1(B_1)}+
\|\phi\|_{L^\infty(B_1\setminus B_r)}\big),
$$
and this gives~\eqref{COSI} up to renaming~$C_R$.
\end{proof}
\end{corollary}

Now, we can complete the proof of Proposition~\ref{sharbou}, by arguing as follows.

\begin{proof}[Proof of Proposition~\ref{sharbou}]
Let $\psi$ be a test function in $C^\infty_0(\mathbb{R}^n)$.
Let also~$R:=\frac{r+1}{2}\in(r,1)$ and
$$ g_\epsilon(X):=
\epsilon^{-s}\phi_*(e+\epsilon X)\partial^{\beta}\psi(X).$$
We claim that
\begin{equation}\label{7UHSNs9oKN}
\sup_{{X\in{\mathbb{R}}^n}}|g_\epsilon(X)|\le C,\end{equation}
for some~$C>0$ independent of~$\epsilon$.
To prove this, we distinguish three cases.
If~$e+\epsilon X\in{\mathbb{R}}^n\setminus B_1$,
we have that~$\phi_*(e+\epsilon X)=0$ and thus~$g_\epsilon(X)=0$.
If instead~$e+\epsilon X\in B_R$,
we observe that
$$ R>|e+\epsilon X|\ge 1-\epsilon|X|,$$
and therefore~$|X|\ge \frac{1-R}{\epsilon}$. In particular,
in this case~$X$ falls outside the support of~$\psi$, as long as~$\epsilon>0$
is sufficiently small, and consequently~$\partial^{\beta}\psi(X)=0$
and~$g_\epsilon(X)=0$.

Hence, to complete the proof of~\eqref{7UHSNs9oKN},
we are only left with the case in which~$
e+\epsilon X\in B_1\setminus B_R$. In this situation,
we make use of~\eqref{COSI} and we find that
\begin{eqnarray*}
&& |\phi_*(e+\epsilon X)|\le C\,d^{s}(e+\epsilon X)=
C\,(1-|e+\epsilon X|)^s\\&&\qquad
\le C\,(1-|e+\epsilon X|)^s(1+|e+\epsilon X|)^s=
C\,(1-|e+\epsilon X|^2)^s\\&&\qquad=
C\,\epsilon^s(-2e\cdot X-\epsilon|X|^2)^s\le C\epsilon^s,
\end{eqnarray*}
for some~$C>0$ possibly varying from line to line,
and this completes the proof of~\eqref{7UHSNs9oKN}.

Now, from~\eqref{7UHSNs9oKN} and the
Dominated Convergence Theorem, we obtain that
\begin{equation}\label{eq567a8s81n} \lim_{\epsilon\searrow0}\int_{\mathbb{R}^n}
\epsilon^{-s}\phi_*(e+\epsilon X)\partial^{\beta}\psi(X) dX
=\int_{\mathbb{R}^n} \lim_{\epsilon\searrow0}
\epsilon^{-s}\phi_*(e+\epsilon X)\partial^{\beta}\psi(X) dX.\end{equation}
On the other hand, by Corollary~\ref{8iJJAUMPAAAxc},
used here with~$\omega:=\frac{X}{|X|}$, we know that
\begin{eqnarray*}&& \lim_{\epsilon\searrow0}
\epsilon^{-s}\phi_*(e+\epsilon X)
=\lim_{\epsilon\searrow0}
\epsilon^{-s}\phi_*(e+\epsilon |X|\omega)=|X|^s
\lim_{\epsilon\searrow 0}\epsilon^{-s}\phi_*(e+\epsilon\omega)\\&&\qquad=k_*\,|X|^s\,
(-e\cdot\omega)^s_+=k_*\,(-e\cdot X)^s_+.
\end{eqnarray*}
Substituting this into~\eqref{eq567a8s81n}, we thus find that
$$ \lim_{\epsilon\searrow0}\int_{\mathbb{R}^n}
\epsilon^{-s}\phi_*(e+\epsilon X)\partial^{\beta}\psi(X) dX
=k_*\,\int_{\mathbb{R}^n} (-e\cdot X)^s_+\partial^{\beta}\psi(X) dX.$$
As a consequence, integrating by parts twice,
\begin{equation*}
\begin{split}
&\lim_{\epsilon\searrow 0}\epsilon^{|\beta|-s}\int_{\mathbb{R}^n}
\partial^\beta\phi_*(e+\epsilon X)\psi(X) dX=
\lim_{\epsilon\searrow 0}\int_{\mathbb{R}^n}\partial^\beta
\Big(\epsilon^{-s}\phi_*(e+\epsilon X)\Big)\psi(X) dX \\
&\qquad=(-1)^{|\beta|}\lim_{\epsilon\searrow 0}\int_{\mathbb{R}^n}
\epsilon^{-s}\phi_*(e+\epsilon X)\partial^{\beta}\psi(X) dX \\
&\qquad=(-1)^{|\beta|}\,k_*\,\int_{\mathbb{R}^n} (-e\cdot X)^s_+\partial^{\beta}\psi(X) dX\\
&\qquad=k_*\,\int_{\mathbb{R}^n} \partial^{\beta}(-e\cdot X)^s_+\psi(X) dX
\\
&\qquad=(-1)^{|\beta|}\,k_*\, s(s-1)\ldots(s-|\beta|+1)e_1^{\beta_1}\ldots e_n^{\beta_n}\int_{\mathbb{R}^n}(-e\cdot X)^{s-|\beta|}_+\psi(X) dX.
\end{split}
\end{equation*}
Since the test function $\psi$ is arbitrary, the claim in
Proposition~\ref{sharbou} is proved.
\end{proof}

\section{Boundary behaviour of~$s$-harmonic functions}
\label{s:hwb}

In this section we analyze the asymptotic behaviour of $s$-harmonic
functions, with a ``spherical bump function'' as exterior Dirichlet datum.

The result needed for our purpose is the following:

\begin{lemma}
\label{hbump} 
Let $s>0$. Let~$m\in\mathbb{N}_0$
and~$\sigma\in(0,1)$ such that~$s=m+\sigma$.

Then, there exists
\begin{equation}\label{0oHKNSSH013oe2urjhfe}
{\mbox{$\psi\in H^s(\mathbb{R}^n)\cap C^s_0(\mathbb{R}^n)$
such that $
(-\Delta)^s \psi=0$ in~$B_1$,}}\end{equation} and, for
every $x\in\partial B_{1-\epsilon}$,
\begin{equation}\label{0oHKNSSH013oe2urjhfe:2}
\psi(x)=k\,\epsilon^s+o(\epsilon^s),\end{equation}
as $\epsilon\searrow 0$, for some $k>0$.
\begin{proof}

Let~$\overline{\psi}\in C^\infty(\mathbb{R},\,[0,1])$ such that $\overline{\psi}=0$ in $\mathbb{R}\setminus(2,3)$ and $\overline{\psi}>0$ in $(2,3)$. Let $\psi_0(x):=(-1)^m\overline{\psi}(|x|)$.
We recall the Poisson kernel 
$$ 
\Gamma_s(x,y):=(-1)^m\frac{\gamma_{n,\sigma}}{
|x-y|^n}\frac{(1-|x|^2)^s_+}{(|y|^2-1)^s},$$
for $x\in\mathbb{R}^n$, $y\in\mathbb{R}^n\setminus\overline{B}_1$, and
a suitable normalization constant~$\gamma_{n,\sigma}>0$ (see formulas~(1.10) and~(1.30)
in~\cite{ABX}).
We define
$$ \psi(x):=
\displaystyle\int_{{\mathbb{R}}^n\setminus B_1} \Gamma_s(x,y)\,\psi_0(y)\,dy+\psi_0(x).$$
Notice that~$\psi_0=0$ in~$B_{3/2}$ and therefore we can
exploit Theorem
in~\cite{ABX} and obtain that~\eqref{0oHKNSSH013oe2urjhfe}
is satisfied
(notice also that~$\psi=\psi_0$ outside~$B_1$, hence~$\psi$ is compactly supported).

Furthermore, to prove~\eqref{0oHKNSSH013oe2urjhfe:2}
we borrow some ideas from Lemma 2.2 in~\cite{MR3626547}
and we see that, for any $x\in \partial B_{1-\epsilon}$,
\begin{equation*}
\begin{split}
\psi(x)
&=c(-1)^m\int_{\mathbb{R}^n\setminus B_1} \frac{\psi_0(y)(1-|x|^2)^s}{(|y|^2-1)^s|x-y|^n} dy+\psi_0(x) \\
&=c(-1)^m\int_{\mathbb{R}^n\setminus B_1} \frac{\psi_0(y)(1-|x|^2)^s}{(|y|^2-1)^s|x-y|^n} dy \\
&=c\,(1-|x|^2)^s\int_2^3\left[\int_{\mathbb{S}^{n-1}} \frac{\rho^{n-1}\overline{\psi}(\rho)}{(\rho^2-1)^s|x-\rho\omega|^n} d\omega\right] d\rho
\\&=c\,(2\epsilon-\epsilon^2)^s\int_2^3\left[\int_{\mathbb{S}^{n-1}} \frac{\rho^{n-1}\overline{\psi}(\rho)}{(\rho^2-1)^s|(1-\epsilon)e_1-\rho\omega|^n} d\omega\right] d\rho \\
&=2^sc\,\epsilon^s\int_2^3\left[\int_{\mathbb{S}^{n-1}} \frac{\rho^{n-1}\overline{\psi}(\rho)}{(\rho^2-1)^s|e_1-\rho\omega|^n} d\omega\right] d\rho+o(\epsilon^s) \\
&=c\epsilon^s+o(\epsilon^s),
\end{split}
\end{equation*}
where~$c>0$ is a constant possibly varying from line to line, and this establishes~\eqref{0oHKNSSH013oe2urjhfe:2}.
\end{proof}
\end{lemma}

\begin{remark}\label{RUCAPSJD} {\rm
As in Proposition~\ref{sharbou}, one can extend~\eqref{0oHKNSSH013oe2urjhfe:2}
to higher derivatives (in the distributional sense), obtaining, for any~$e\in\partial B_1$
and~$\beta\in\mathbb{N}^n$
$$ \lim_{\epsilon\searrow0} \epsilon^{|\beta|-s}\partial^\beta\psi(e+\epsilon X)=k_\beta\,
e_1^{\beta_1}\dots e_n^{\beta_n}(-e\cdot X)_+^{s-|\beta|}
,$$
for some~$\kappa_\beta\ne0$.}\end{remark}

Using Lemma \ref{hbump}, in the spirit of \cite{MR3626547}, we
can construct a sequence of $s$-harmonic functions
approaching~$(x\cdot e)^s_+$ for a fixed unit vector $e$,
by using a blow-up argument. Namely, we prove the following:

\begin{corollary}
\label{lapiog}
Let $e\in\partial B_1$. There exists a sequence $v_{e,j}\in H^s(\mathbb{R}^n)\cap C^s(\mathbb{R}^n)$ such that $(-\Delta)^s v_{e,j}=0$ in $B_1(e)$, $v_{e,j}=0$ in $\mathbb{R}^n\setminus B_{4j}(e)$, and \[v_{e,j}\to\kappa(x\cdot e)^s_+\quad\mbox{in}\quad L^1(B_1(e)),\] as $j\to+\infty$, for some $\kappa>0$.
\begin{proof}
Let $\psi$ be as in Lemma \ref{hbump} and
define \[v_{e,j}(x):=j^s\psi\left(\frac{x}{j}-e\right).\]
The $s$-harmonicity and the property of being compactly supported follow
by the ones of $\psi$. We now prove the convergence.
To this aim, given $x\in B_1(e)$, we write $p_j:=\frac{x}{j}-e$ and $\epsilon_j:=1-|p_j|$. Recall that since $x\in B_1(e)$, then $|x-e|^2<1$, which implies that $|x|^2<2x\cdot e$ and $x\cdot e>0$ for any $x\in B_1(e)$. \\
As a consequence \[|p_j|^2=\left|\frac{x}{j}-e\right|^2=
\frac{|x|^2}{j^2}+1-2\frac{x}{j}\cdot e=1-\frac{2}{j}(x\cdot e)_+
+o\left(\frac{1}{j}\right)(x\cdot e)^2_+,\]
and so \[\epsilon_j=\frac{(1+o(1))}{j}(x\cdot e)_+.\]
Therefore, using \eqref{0oHKNSSH013oe2urjhfe:2},
\begin{equation*}
\begin{split}
v_{e,j}(x)&=j^s\psi(p_j) \\ &=j^s\kappa(\epsilon_j^s+o(\epsilon^s_j)) \\
&=j^s\left(\frac{\kappa}{j^s}(x\cdot e)_+^s+o\left(\frac{1}{j^s}\right)\right) \\
&=\kappa(x\cdot e)^s_+ +o(1).
\end{split}
\end{equation*}
Integrating over $B_1(e)$, we obtain the desired $L^1$-convergence.
\end{proof}
\end{corollary}

Now, we show that, as in the case $s\in (0,1)$ proved in
Theorem~3.1 of \cite{MR3626547}, we can find an $s$-harmonic function
with an arbitrarily large number of derivatives prescribed at some point.

\begin{proposition}
\label{maxhlapspan}
For any $\beta\in\mathbb{N}^n$, there exist~$p\in\mathbb{R}^n$, $R>r>0$, and~$v\in H^s(\mathbb{R}^n)\cap C^s(\mathbb{R}^n)$ such that
\begin{equation}
\label{csi}
\begin{cases}
(-\Delta)^s v=0&\quad\text{in}\quad B_r(p), \\
v=0&\quad\text{in}\quad\mathbb{R}^n\setminus B_R(p),
\end{cases}
\end{equation}
\begin{equation*}
D^\alpha v(p)=0\quad{\mbox{ for any }}\; \alpha\in\mathbb{N}^n\quad{\mbox{ with }}\;|\alpha|\leq|\beta|-1,\end{equation*}
\begin{equation*}
D^\alpha v(p)=0\quad{\mbox{ for any }}\; \alpha\in\mathbb{N}^n\quad{\mbox{ with }}\;|\alpha|=|\beta|\quad{\mbox{ and }}\;\alpha\neq\beta\end{equation*}  
and 
\begin{equation*}
D^\beta v(p)=1.\end{equation*}

\begin{proof}
Let $\mathcal{Z}$ be the set of all pairs~$(v,x)\in\left(H^s(\mathbb{R}^n)\cap C^s(\mathbb{R}^n)\right)\times B_r(p)$ that satisfy \eqref{csi} for some $R>r>0$ and $p\in\mathbb{R}^n$. 

To each pair $(v,x)\in\mathcal{Z}$ we associate the vector
$\left(D^\alpha v(x)\right)_{|\alpha|\leq|\beta|}\in\mathbb{R}^{K'}$, for some $K'=K'_{|\beta|}$ and 
consider~$\mathcal{V}$ to be the vector space spanned by this construction, namely
we set
$$\mathcal{V}:=\Big\{\left(D^\alpha v(x)\right)_{|\alpha|\leq|\beta|},
\quad{\mbox{ with }}\; (v,x)\in\mathcal{Z}
\Big\}.$$
We claim that
\begin{equation}\label{CLSPAZ}
\mathcal{V}=\mathbb{R}^{K'}.\end{equation}
To check this, we suppose by contradiction that~$\mathcal{V}$ lies in a 
proper subspace of~$\mathbb{R}^{K'}$. Then, $\mathcal{V}$ must lie in a
hyperplane, hence there exists 
\begin{equation}
\label{cnonull0}
c=(c_\alpha)_{|\alpha|\leq|\beta|}\in\mathbb{R}^{K'}\setminus\left\{0\right\} 
\end{equation}
which is orthogonal to any vector $\left(D^\alpha v(x)\right)_{|\alpha|\leq|\beta|}$
with~$(v,x)\in\mathcal{Z}$, that is
\begin{equation}
\label{prp18}
\sum_{|\alpha|\leq|\beta|} c_\alpha D^\alpha v(x)=0.
\end{equation}
We notice that the pair $(v_{e,j},x)$, with $v_j$ as in Corollary \ref{lapiog},
$e\in\partial B_1$
and $x\in B_1(e)$, belongs to~$\mathcal{Z}$. Consequently,
fixed~$\xi\in\mathbb{R}^n\setminus B_{1/2}$
and set~$e:=\frac{\xi}{|\xi|}$, we have that~\eqref{prp18} holds true when~$v:=v_{e,j}$ and $x\in B_1(e)$, namely
$$
\sum_{|\alpha|\leq|\beta|} c_\alpha D^\alpha v(x)=0.
$$
Let now~$\varphi\in C_0^\infty(B_1(e))$. 
Integrating by parts,
by Corollary \ref{lapiog} and the Dominated Convergence Theorem, 
we have that
\begin{eqnarray*}
&&0=\lim_{j\to+\infty}\int_{\mathbb{R}^n}\sum_{|\alpha|\leq|\beta|}
c_\alpha D^\alpha v_{e,j}(x)\varphi(x)\,dx
=\lim_{j\to+\infty}\int_{\mathbb{R}^n}\sum_{|\alpha|\leq|\beta|}(-1)^{|\alpha|}
c_\alpha v_{e,j}(x)D^\alpha\varphi(x)\,dx\\
&&\qquad=\kappa\int_{\mathbb{R}^n}\sum_{|\alpha|\leq|\beta|}(-1)^{|\alpha|}
c_\alpha(x\cdot e)^s_+D^\alpha\varphi(x)\,dx
=\kappa\int_{\mathbb{R}^n}\sum_{|\alpha|\leq|\beta|}c_\alpha
D^\alpha(x\cdot e)^s_+\varphi(x)\,dx.\end{eqnarray*}
This gives that, for every $x\in B_1(e)$,
\[\sum_{|\alpha|\leq|\beta|}c_\alpha D^\alpha(x\cdot e)^s_+=0.\] 
Moreover, for every $x\in B_1(e)$,
\[D^\alpha(x\cdot e)^s_+=s(s-1)\ldots(s-|\alpha|+1)
(x\cdot e)^{s-|\alpha|}_+e_1^{\alpha_1}\ldots e_n^{\alpha_n}.\]
In particular, for $x=\frac{e}{|\xi|}\in B_1(e)$,
\[D^\alpha(x\cdot e)^s_+\big|_{|_{x=e/{|\xi|}}}=s(s-1)
\ldots(s-|\alpha|+1)|\xi|^{-s}\xi_1^{\alpha_1}\ldots \xi_n^{\alpha_n}.\] And, using the usual multi-index notation, we write
\begin{equation}
\label{sampspal}
\sum_{|\alpha|\leq|\beta|}c_\alpha s(s-1)\ldots(s-|\alpha|+1)
\xi^\alpha=0,
\end{equation}
for any $\xi\in\mathbb{R}^n\setminus B_{1/2}$.
The identity \eqref{sampspal} describes
a polynomial in $\xi$ which vanishes for any~$\xi$ in an open subset of $\mathbb{R}^n$. As a result, the Identity Principle
for polynomials leads to
$$ c_\alpha s(s-1)\ldots(s-|\alpha|+1)=0,$$
for all~$|\alpha|\leq|\beta|$.

Consequently, since $s\in\mathbb{R}\setminus\mathbb{N}$, 
the product $s(s-1)\ldots(s-|\alpha|+1)$ never vanishes, 
and so the coefficients $c_\alpha$ are forced to be null for any $|\alpha|\leq|\beta|$. 
This is in contradiction with~\eqref{cnonull0}, and therefore the proof
of~\eqref{CLSPAZ} is complete.

{F}rom this, the desired claim in
Proposition~\ref{maxhlapspan} plainly follows.
\end{proof}
\end{proposition}

\section{A result which implies Theorem \ref{theone}}\label{s:fourthE}

We will use the notation
\begin{equation}\label{NEOAAJKin1a}
\Lambda_{-\infty}:=\Lambda_{(-\infty,\dots,-\infty)},\end{equation}
that is we exploit~\eqref{1.6BIS} with~$a_1:=\dots:=a_l:=-\infty$.
This section presents the following statement:

\begin{theorem}\label{theone2}
Suppose that 
\begin{equation*}\begin{split}&
{\mbox{either there exists~$i\in\{1,\dots,M\}$ such that~$\XB_i\ne0$
and~$s_i\not\in{\mathbb{N}}$,}}\\
&{\mbox{or there exists~$i\in\{1,\dots,l\}$ such that~$\XC_i\ne0$ and $\alpha_i\not\in{\mathbb{N}}$.}}\end{split}
\end{equation*}
Let $\ell\in\mathbb{N}$, $f:\mathbb{R}^N\rightarrow\mathbb{R}$,
with $f\in C^{\ell}\big(\overline{B_1^N}\big)$. Fixed $\epsilon>0$,
there exist
\begin{equation*}\begin{split}&
u=u_\epsilon\in C^\infty\left(B_1^N\right)\cap C\left(\mathbb{R}^N\right),\\
&a=(a_1,\dots,a_l)=(a_{1,\epsilon},\dots,a_{l,\epsilon})
\in(-\infty,0)^l,\\ {\mbox{and }}\quad&
R=R_\epsilon>1\end{split}\end{equation*} such that:
\begin{itemize}
\item for every~$h\in\{1,\dots,l\}$ and~$(x,y,t_1,\dots,t_{h-1},t_{h+1},\dots,t_l)$
\begin{equation}\label{SPAZIO}
{\mbox{the map ${\mathbb{R}}\ni t_h\mapsto u(x,y,t)$
belongs to~$C^{k_h,\alpha_h}_{-\infty}$,}}
\end{equation}
in the notation of formula~(1.4) of~\cite{CDV18},
\item it holds that
\begin{equation}\label{MAIN EQ:2}\left\{\begin{matrix}
\Lambda_{-\infty} u=0 &\mbox{ in }\;B_1^{N-l}\times(-1,+\infty)^l, \\
u(x,y,t)=0&\mbox{ if }\;|(x,y)|\ge R,
\end{matrix}\right.\end{equation}
\begin{equation}\label{ESTENSIONE}
\partial^{k_h}_{t_h} u(x,y,t)=0\qquad{\mbox{if }}t_h\in(-\infty,a_h),\qquad{\mbox{for all }}h\in\{1,\dots,l\},
\end{equation}
and
\begin{equation}\label{IAzofm:2}
\left\|u-f\right\|_{C^{\ell}(B_1^N)}<\epsilon.
\end{equation}\end{itemize}
\end{theorem}

The proof of Theorem~\ref{theone2} will basically occupy the
rest of this paper, and this will lead us to the completion of the
proof of Theorem~\ref{theone}. Indeed, we have that:

\begin{lemma}\label{GRAT}
If the statement of Theorem~\ref{theone2} holds true,
then the statement in Theorem~\ref{theone} holds true.
\end{lemma}

\begin{proof} Assume that the claims
in Theorem~\ref{theone2} are satisfied. Then, by~\eqref{SPAZIO} and~\eqref{ESTENSIONE},
we are in the position of exploting Lemma~A.1 in~\cite{CDV18}
and conclude that, in~$B_1^{N-l}\times(-1,+\infty)^l$,
$$ D^{\alpha_h}_{t_h ,a_h} u=D^{\alpha_h}_{t_h,-\infty} u,$$
for every~$h\in\{1,\dots,l\}$. This and~\eqref{MAIN EQ:2}
give that
\begin{equation} \label{33ujNAKS}
\Lambda_{a}u=\Lambda_{-\infty} u=0 \qquad\mbox{ in }\;B_1^{N-l}
\times(-1,+\infty)^l.\end{equation}
We also define
$$\underline{a}:=\min_{h\in\{1,\dots,l\}} a_h$$
and take~$\tau\in C^\infty_0 ([-\underline{a}-2,3])$ with~$\tau=1$
in~$[-\underline{a}-1,1]$. Let
\begin{equation}\label{UJNsdA} U(x,y,t):=u(x,y,t)\,\tau(t_1)\dots\tau(t_l).\end{equation}
Our goal is to prove that~$U$ satisfies the theses of Theorem~\ref{theone}.
To this end, we observe that~$u=U$ in~$B^N_1$, therefore~\eqref{IAzofm}
for~$U$
plainly follows from~\eqref{IAzofm:2}.

In addition, from~\eqref{defcap}, we see that~$
D^{\alpha_h}_{t_h,a_h}$ at a point~$t_h\in(-1,1)$
only depends on the values of the function
between~$a_h$ and~$1$. Since the cutoffs in~\eqref{UJNsdA} do not
alter these values, we see that~$D^{\alpha_h}_{t_h,a_h}U=D^{\alpha_h}_{t_h,a_h}u$
in~$B_1^N$, and accordingly~$\Lambda_a U=\Lambda_a u$ in~$B_1^N$.
This and~\eqref{33ujNAKS} say that
\begin{equation}\label{9OAJA}
\Lambda_a U=0\qquad{\mbox{in }}B_1^N.\end{equation}
Also, since~$u$ in Theorem~\ref{theone2} is compactly supported
in the variable~$(x,y)$, we see from~\eqref{UJNsdA} that~$U$
is compactly supported in the variables~$(x,y,t)$.
This and~\eqref{9OAJA} give that~\eqref{MAIN EQ} is satisfied by~$U$
(up to renaming~$R$).
\end{proof} 

\section{A pivotal span result towards the proof of Theorem \ref{theone2}}\label{s:fourth0}

In what follows, we let~$\Lambda_{-\infty}$
be as in~\eqref{NEOAAJKin1a}, we recall the setting in~\eqref{1.0},
and we
use the following multi-indices notations:
\begin{equation}\label{mulPM}
\begin{split}
& \iota=\left(i,I,\mathfrak{I}\right)=\left(i_1,\ldots,i_n,I_1,\ldots,I_M,\mathfrak{I}_1,
\ldots,\mathfrak{I}_l\right)\in\mathbb{N}^N\\
{\mbox{and }} &
\partial^\iota w=\partial^{i_1}_{x_1}\ldots\partial^{i_n}_{x_n}
\partial^{I_1}_{y_1}\ldots\partial^{I_M}_{y_M}\partial^{\mathfrak{I}_1}_{t_1}
\ldots\partial^{\mathfrak{I}_l}_{t_l}w.
\end{split}\end{equation}
Inspired by Lemma 5 of~\cite{DSV1},
we consider the span of the derivatives of functions in~$
\ker\Lambda_{-\infty}$, with derivatives up to a fixed order $K\in{\mathbb{N}}$.
We want to prove that the derivatives of such functions span
a maximal vectorial space. 

For this, we denote by $\partial^K w(0)$
the vector with entries given,
in some prescribed order,
by~$
\partial^\iota w(0)$ with $\left|\iota\right|\leq K$.

We notice that
\begin{equation}\label{8iokjKK}
{\mbox{$\partial^K w(0)\in\mathbb{R}^{K'}$ for some $K'\in{\mathbb{N}}$,}}
\end{equation}
with~$K'$ depending on~$K$.

Now, we adopt the notation in formula~(1.4) of~\cite{CDV18},
and
we denote by~$
\mathcal{A}$ \label{CALSASS}
the set of all functions~$w=w(x,y,t)$
such that for all~$h\in\{1,\ldots, l\}$ and all~$
(x,y,t_1,\ldots,t_{h-1},t_{h+1},\ldots,t_l)\in\mathbb{R}^{N-1}$,
the map~${\mathbb{R}}\ni t_h\mapsto w(x,y,t)$ belongs to~$
C^{\infty}((a_h,+\infty))\cap C^{k_h,\alpha_h}_{-\infty}$,
and~\eqref{ESTENSIONE} holds true for some~$a_h\in (-2,0)$.

We also set
\begin{equation*}
\begin{split}
\mathcal{H}:=\Big\{w\in C(\mathbb{R}^N)
\cap C_0(\mathbb{R}^{N-l})\cap C^\infty(\mathcal{N})\cap\mathcal{A},
\text{ for some neighborhood 
$\mathcal{N} $
of the origin, } \\
 \text{ such that }
\Lambda_{-\infty} w=0 \text{ in }\mathcal{N}\Big\}
\end{split}
\end{equation*}
and, for any $w\in\mathcal{H}$, let $\mathcal{V}_K$ be the vector space spanned by the vector $\partial^K w(0)$. 

By \eqref{8iokjKK}, we know that~$\mathcal{V}_K\subseteq\mathbb{R}^{K'}$.
In fact, we show that equality holds in this inclusion, as
stated in the following\footnote{Notice that results
analogous to Lemma~\ref{lemcin}
cannot hold for solutions of local operators: for instance,
pure second derivatives of harmonic functions have to satisfy
a linear equation, so they are forced to lie in a proper subspace.
In this sense, results such as Lemma~\ref{lemcin} here reveal a truly nonlocal
phenomenon.}
result:
\begin{lemma}
\label{lemcin}
It holds that $\mathcal{V}_K=\mathbb{R}^{K'}$.
\end{lemma}

The proof of Lemma~\ref{lemcin} is 
by contradiction. Namely, if $\mathcal{V}_K$ does not exhaust the whole of $\mathbb{R}^{K'}$ there exists 
\begin{equation}
\label{tetaort}
\theta\in\partial B_1^{K'}
\end{equation}
such that
\begin{equation}
\label{aza}
\mathcal{V}_K\subseteq\left\{\zeta\in\mathbb{R}^{K'}
\text{ s.t. } \theta\cdot\zeta=0\right\}.
\end{equation}
In coordinates, recalling~\eqref{mulPM},
we write~$\theta$
as~$\theta_\iota=\theta_{i,I,\mathfrak{I}}$,
with~$i\in\mathbb{N}^{p_1+\dots+p_n}$,
$I\in\mathbb{N}^{m_1+\dots+m_M}$
and~$\mathfrak{I}\in\mathbb{N}^l$.
We consider
\begin{equation}\label{IBARRA}
\begin{split}&
{\mbox{a multi-index $\overline{I}\in\mathbb{N}^{m_1+\dots+m_M}$ 
such that it maximizes~$|I|$}}\\
&{\mbox{among all the multi-indexes~$(i,I,\mathfrak{I})$
for which~$\left|i\right|+\left|I\right|+|\mathfrak{I}|\leq K$}}\\
&{\mbox{and~$\theta_{i,I,\mathfrak{I}}\ne0$
for some~$(i,\mathfrak{I})$.}}
\end{split}\end{equation}
Some comments on the setting\label{FFOAK}
in~\eqref{IBARRA}. We stress that, by~\eqref{tetaort},
the set~$\mathcal{S}$
of indexes~$I$ for which there exist indexes~$
(i,\mathfrak{I})$ such that~$|i|+|I|+|\mathfrak{I}|\le K$
and~$\theta_{i,I,\mathfrak{I}}\ne0$ is not empty.
Therefore, since~${\mathcal{S}}$ is a finite set,
we can take
$$ S:=\sup_{I\in {\mathcal{S}}} |I|=\max_{I\in {\mathcal{S}}}
|I|\in{\mathbb{N}}\cap [0,K].$$
Hence, we consider a multi-index $\overline{I}$ for
which~$|\overline I|=S$ to obtain the
setting in~\eqref{IBARRA}. By construction, we have that
\begin{itemize}
\item $|i|+|\overline I|+|\mathfrak{I}|\le K$,
\item  if~$|I|>|\overline I|$, then $\theta_{i,I,\mathfrak{I}}=0$,
\item
and there exist multi-indexes~$i$ and~$\mathfrak{I}$
such that~$\theta_{i,\overline I,\mathfrak{I}}\ne0$.\end{itemize}

As a variation of the setting in~\eqref{IBARRA},
we can also consider
\begin{equation}\label{IBARRA2}
\begin{split}&
{\mbox{a multi-index $\overline{\mathfrak{I}}\in\mathbb{N}^{l}$ 
such that it maximizes~$|\mathfrak{I}|$}}\\
&{\mbox{among all the multi-indexes~$(i,I,\mathfrak{I})$
for which~$\left|i\right|+\left|I\right|+|\mathfrak{I}|\leq K$}}\\
&{\mbox{and~$\theta_{i,I,\mathfrak{I}}\ne0$
for some~$(i,I)$.}}
\end{split}\end{equation}
In the setting of~\eqref{IBARRA} and~\eqref{IBARRA2},
we claim that there exists an open set
of~$\mathbb{R}^{p_1+\ldots+p_n}\times\mathbb{R}^{m_1+\ldots+m_M}\times\mathbb{R}^{l}$
such that for every~$(\XX,\XY,\XT)$ in such open set we have that
\begin{equation}
\label{ipop}\begin{split}
{\mbox{either }}\qquad&
0=\sum_{{|i|+|I|+|\mathfrak{I}|\le K}\atop{|I| = |\overline{I}|}}
C_{i,I,\mathfrak{I}}\;\theta_{i,I,\mathfrak{I}}\;
\XX^i \XY^{{I}}\XT^{\mathfrak{I}},\qquad{\mbox{ with }}\qquad
C_{i,I,\mathfrak{I}}\ne0,\\
{\mbox{or }}\qquad&
0=\sum_{{|i|+|I|+|\mathfrak{I}|\le K}\atop{|\mathfrak{I}| = |\overline{\mathfrak{I}}|}}
C_{i,I,\mathfrak{I}}\;\theta_{i,I,\mathfrak{I}}\;
\XX^i \XY^{{I}}\XT^{\mathfrak{I}},\qquad{\mbox{ with }}\qquad
C_{i,I,\mathfrak{I}}\ne0.\end{split}
\end{equation}
In our framework, the claim in~\eqref{ipop} will be pivotal
towards the completion of the proof of Lemma~\ref{lemcin}.
Indeed, let us suppose for the moment that~\eqref{ipop}
is established and let us complete the proof of Lemma~\ref{lemcin}
by arguing as follows.

Formula \eqref{ipop} says that $\theta\cdot\partial^K w(0)$
is a polynomial which vanishes for any triple $(\XX,\XY,\XT)$
in an open subset of $\mathbb{R}^{p_1+\ldots+p_n}\times\mathbb{R}^{m_1+\ldots+m_M}\times\mathbb{R}^{l}$.
Hence, using the identity principle of polynomials, we have that each $C_{i,I,\mathfrak{I}}\;\theta_{i,I,\mathfrak{I}}$ is equal to zero
whenever~$|i|+|I|+|\mathfrak{I}|\le K$
and either~$|I|=|\overline I|$ (if the first identity in~\eqref{ipop}
holds true) or~$|\mathfrak{I}|=|\overline{\mathfrak{I}}|$
(if the second identity in~\eqref{ipop}
holds true). Then, since~$C_{i,I,\mathfrak{I}}\neq 0$,
we conclude that each $\theta_{i,I,\mathfrak{I}}$ is zero
as long as either~$|I|=|\overline{I}|$ (in the first case)
or~$|\mathfrak{I}|=|\overline{\mathfrak{I}}|$
(in the second case), but this contradicts either
the definition of $\overline{I}$
in~\eqref{IBARRA} (in the first case)
or the definition of~$\overline{\mathfrak{I}}$
in~\eqref{IBARRA2} (in the second case). This would therefore complete the proof
of Lemma~\ref{lemcin}.
\medskip

In view of the discussion above,
it remains to prove~\eqref{ipop}.
To this end, we distinguish the following four
cases:
\begin{enumerate}

\item\label{itm:case1} there exist $i\in\{1,\dots,n\}$ and
$j\in\{1,\dots,M\}$ such that~$\XA_i\ne0$ and~$\XB_j\ne0$,
\item\label{itm:case2} there exist $i\in\{1,\dots,n\}$ and
$h\in\{1,\dots,l\}$ such that~$\XA_i\ne0$ and~$\XC_h\ne0$,
\item\label{itm:case3} we have that~$\XA_1=\dots=\XA_n=0$,
and there exists~$j\in\{1,\dots,M\}$ such that~$\XB_j\ne0$,
\item\label{itm:case4} we have that~$\XA_1=\dots=\XA_n=0$,
and there exists~$h\in\{1,\dots,l\}$ such that~$\XC_h\ne0$.

\end{enumerate}

Notice that cases~\ref{itm:case1} and~\ref{itm:case3}
deal with the case in which space fractional diffusion is present
(and in case~\ref{itm:case1} one also has classical
derivatives, while in case~\ref{itm:case3}
the classical derivatives are absent).

Similarly, cases~\ref{itm:case2} and~\ref{itm:case4}
deal with the case in which time fractional diffusion is present
(and in case~\ref{itm:case2} one also has classical
derivatives, while in case~\ref{itm:case4}
the classical derivatives are absent).

Of course, the case in which both space and time
fractional diffusion occur is already comprised by the
previous cases (namely, it is comprised in
both cases~\ref{itm:case1} and~\ref{itm:case2}
if classical derivatives are also present,
and in both cases~\ref{itm:case3}
and~\ref{itm:case4} if classical derivatives are absent).

\begin{proof}[Proof of \eqref{ipop}, case \ref{itm:case1}]
For any $j\in\left\{1,\ldots,M\right\}$ we denote by $\tilde{\phi}_{\star,j}$
the first eigenfunction for $(-\Delta)^{s_j}_{y_j}$
vanishing outside $B_1^{m_j}$ given in Corollary \ref{QUESTCP}. 
We normalize it such that $ \|\tilde{\phi}_{\star,j}\|_{L^2(\mathbb{R}^{m_j})}=1$,
and we write $\lambda_{\star,j}\in(0,+\infty)$ to indicate the corresponding first eigenvalue
(which now depends on~$s_j$), namely we write
\begin{equation}
\label{lambdastarj}
\begin{cases}
(-\Delta)^{s_j}_{y_j}\tilde{\phi}_{\star,j}=\lambda_{\star,j}\tilde{\phi}_{\star,j}&\quad\text{in}\,B_1^{m_j} ,\\
\tilde{\phi}_{\star,j}=0&\quad\text{in}\,\mathbb{R}^{m_j}\setminus\overline{B}_1^{m_j}.
\end{cases}
\end{equation}
Up to reordering the variables and/or
taking the operators to the other side of the equation, 
given the assumptions of case~\ref{itm:case1},
we can suppose that 
\begin{equation}\label{AGZ}
{\mbox{$\XA_1\ne0$}}\end{equation}
and
\begin{equation}\label{MAGGZ}
{\mbox{$\XB_M>0$}}.\end{equation} 
In view of~\eqref{AGZ}, we can define
\begin{equation}\label{MAfghjkGGZ} R:=\left( 
\frac{ 1 }{|\XA_1|}\displaystyle\left(\sum_{j=1}^{M-1}{|\XB_j|\lambda_{\star,j}}
+\sum_{h=1}^l|\XC_h|\right)\right)^{1/|r_{1}|}.\end{equation}
Now, we fix two sets of free parameters 
\begin{equation}\label{FREExi}
\XX_1\in(R+1,R+2)^{p_1},\ldots,\XX_n\in(R+1,R+2)^{p_n},\end{equation}
and
\begin{equation}\label{FREEmustar}
\XT_{\star,1}\in\left(\frac12,1\right),\dots,\XT_{\star,l}\in\left(\frac12,1\right).\end{equation}
We also set 
\begin{equation}\label{1.6md}
{\mbox{$\lambda_j:=\lambda_{\star,j}$ for $j\in\left\{1,\ldots,M-1\right\}$, }}\end{equation}
where $\lambda_{\star,j}$ is defined as in \eqref{lambdastarj}, and
\begin{equation}
\label{alp}
\lambda_M\,:=\,\frac{1}{\XB_M}\left(
\sum_{j=1}^n {\left|\XA_j\right|\XX_j^{r_j}}-\sum_{j=1}^{M-1}
{\XB_j\lambda_j}-\sum_{h=1}^l\XC_h\XT_{\star,h}\right).\end{equation}
Notice that this definition is well-posed, thanks to~\eqref{MAGGZ}.
In addition, from~\eqref{FREExi}, we can write~$\XX_{j}=(
\XX_{j1},\dots,\XX_{jp_j})$, and
we know that~$\XX_{j\ell}>R+1$
for any~$j\in\{1,\dots,n\}$ and any~$\ell\in\{1,\dots,p_j\}$.
Therefore,
\begin{equation}\label{1.15bis} \XX_j^{r_j}= \XX_{j1}^{r_{j1}}\dots\XX_{jp_j}^{r_{jp_j}}\ge0.\end{equation}
{F}rom this, \eqref{MAfghjkGGZ} and~\eqref{FREEmustar}, we deduce that
\begin{eqnarray*}&& \sum_{j=1}^n {\left|\XA_j\right|\XX_j^{r_j}}
\ge \left|\XA_1\right|\XX_1^{r_1}\ge
\left|\XA_1\right| (R+1)^{|r_1|}>
\left|\XA_1\right| R^{|r_1|}\\&&\qquad
=\sum_{j=1}^{M-1}
{|\XB_j|\lambda_j}+\sum_{h=1}^l |\XC_h|\geq\sum_{j=1}^{M-1}
{\XB_j\lambda_j}+\sum_{h=1}^l \XC_h\XT_{\star,h},\end{eqnarray*}
and consequently, by~\eqref{alp},
\begin{equation}
\label{alp-0}
\lambda_M>0.\end{equation}
We also set
\begin{equation}\label{OMEj}
\omega_j:=\begin{cases}1&\quad\text{if }\,j=1,\dots,M-1 ,\\
\displaystyle\frac{\lambda_{\star,M}^{1/2s_M}}{
\lambda_M^{1/2s_M}}&\quad\text{if }\,j=M.\end{cases}
\end{equation}
Notice that this definition is well-posed, thanks to~\eqref{alp-0}.
In addition, 
by~\eqref{lambdastarj}, we have that,
for any $j\in\{1,\dots,M\}$, the functions
\begin{equation}
\label{autofun1}
\phi_j\left(y_j\right):=\tilde{\phi}_{\star,j}\left(\frac{y_j}{\omega_j}\right)
%\label{autofun2}
%{\mbox{and }}\quad \psi_h\left(t_h\right)&:=&\tilde{\psi}_{\star,h}\left(\frac{t_h}{\varrho_h}\right)
\end{equation}
are eigenfunctions of $(-\Delta)^{s_j}_{y_j}$ in $B_{\omega_j}^{m_j}$ 
with external homogenous Dirichlet boundary condition,
and eigenvalues $\lambda_j$:
namely, we can rewrite~\eqref{lambdastarj} as
\begin{equation}\label{REGSWYS-A}
\begin{cases}
(-\Delta)^{s_j}_{y_j} {\phi}_{j}=\lambda_{j} {\phi}_{j}&\quad\text{in}\,B_{\omega_j}^{m_j} ,\\
{\phi}_{j}=0&\quad\text{in}\,\mathbb{R}^{m_j}\setminus\overline{B}_{\omega_j}^{m_j}.
\end{cases}
\end{equation}
Now, we define
\begin{equation}
\label{chosofpsistar}
\psi_{\star,h}(t_h):=E_{\alpha_h,1}(t_h^{\alpha_h}),
\end{equation}
where~$E_{\alpha_h,1}$ denotes the Mittag-Leffler function
with parameters $\alpha:=\alpha_h$ and $\beta:=1$ as defined 
in \eqref{Mittag}.

Moreover, 
we consider~$a_h\in(-2,0)$, for every~$h=1,\dots,l$,
to be chosen appropriately in what follows
(the precise choice will be performed in~\eqref{pata7UJ:AKK}),
and, recalling~\eqref{FREEmustar},
we let 
\begin{equation}\label{TGAdef}
\XT_h:=\XT_{\star,h}^{1/{\alpha_h}},\end{equation} and we define
\begin{equation}
\label{autofun2}
\psi_h(t_h):=\psi_{\star,h}\big(\XT_h (t_h-a_h)\big)=
E_{\alpha_h,1}\big(\XT_{\star,h} (t_h-a_h)^{\alpha_h}\big).
\end{equation}
We point out that, thanks to Lemma~\ref{MittagLEMMA}, the function in~\eqref{autofun2}, solves
\begin{equation}
\label{jhjadwlgh}
\begin{cases}
D^{\alpha_h}_{t_h,a_h}\psi_h(t_h)=\XT_{\star,h}\psi_h(t_h)&\quad\text{in }\,(a_h,+\infty), \\
\psi_h(a_h)=1, \\
\partial^m_{t_h}\psi_h(a_h)=0&\quad\text{for every }\,m\in\{1,\dots,[\alpha_h] \}.
\end{cases}
\end{equation}
Moreover, for any $h\in\{1,\ldots, l\}$, we define
\begin{equation}
\label{starest}
\psi^{\star}_h(t_h):=\begin{cases}
\psi_h(t_h)\qquad\text{ if }\,t_h\in[a_h,+\infty) \\
1\qquad\qquad\text{ if }\,t_h\in(-\infty,a_h).\end{cases}
\end{equation}
Thanks to \eqref{jhjadwlgh} and Lemma A.3 in \cite{CDV18} applied here with $b:=a_h$, $a:=-\infty$, $u:=\psi_h$, $u_\star:=\psi^{\star}_h$, we have that $\psi^{\star}_h\in C^{k_h,\alpha_h}_{-\infty}$, and
\begin{equation}\label{DOBACHA}
D^{\alpha_h}_{t_h,-\infty}\psi_h^\star(t_h)=D^{\alpha_h}_{t_h,a_h}\psi_h(t_h)
=\XT_{\star,h}\psi_h(t_h)
=\XT_{\star,h}\psi_h^\star(t_h)\,\text{ in every interval }\,I\Subset(a_h,+\infty).
\end{equation}
We observe that the setting in \eqref{starest} is compatible with the ones in \eqref{SPAZIO} and \eqref{ESTENSIONE} .

{F}rom~\eqref{Mittag} and~\eqref{autofun2}, we see that
\begin{equation*}
\psi_h(t_h)=
\sum_{j=0}^{+\infty} {\frac{
\XT_{\star,h}^j\, (t_h-a_h)^{\alpha_h j}}{\Gamma\left(\alpha_h j+1\right)}}.
\end{equation*}
Consequently, for every~${\mathfrak{I}_h}\in\mathbb{N}$,
we have that
\begin{equation}\label{STAvca}
\partial^{\mathfrak{I}_h}_{t_h}\psi_h(t_h)=
\sum_{j=0}^{+\infty} {\frac{
\XT_{\star,h}^j\, \alpha_h j(\alpha_h j-1)\dots(\alpha_h j-
\mathfrak{I}_h+1)
(t_h-a_h)^{\alpha_h j-\mathfrak{I}_h}}{\Gamma\left(\alpha_h j+1\right)}}
.\end{equation}
Now, we define, for any $i\in\left\{1,\ldots,n\right\}$,
\begin{equation*}
\overline{\XA}_i:=
\begin{cases}
\displaystyle\frac{\XA_i}{\left|\XA_i\right|}\quad \text{ if }\XA_i\neq 0 ,\\
1 \quad \text{ if }\XA_i=0.
\end{cases}
\end{equation*}
We notice that
\begin{equation}\label{NOZABA}
{\mbox{$\overline{\XA}_i\neq 0$ for all~$i\in\left\{1,\ldots,n\right\}$,}}\end{equation}
and
\begin{equation}\label{NOZABAnd}
{\XA}_i\overline{\XA}_i=|{\XA}_i|.\end{equation}
Now, for each~$i\in\{1,\dots,n\}$, we consider
the multi-index~$r_i=(r_{i1},\dots,r_{i p_i})\in\mathbb{N}^{p_i}$.
This multi-index acts on~$\mathbb{R}^{p_i}$,
whose variables are denoted by~$x_i=(x_{i1},\dots,x_{ip_i})\in\mathbb{R}^{p_i}$.
We let~$\overline{v}_{i1}$ be the solution of
the Cauchy problem
\begin{equation}\label{CAH1}
\begin{cases}
\partial^{r_{i1}}_{x_{i1}}\overline{v}_{i1}=-\overline{\XA}_i\overline{v}_{i1} \\
\partial^{\beta_1}_{x_{i1}}\overline{v}_{i1}\left(0\right)=1\quad
\text{ for every } \beta_1\leq r_{i1}-1.
\end{cases}
\end{equation}
We notice that
the solution of the Cauchy problem in~\eqref{CAH1}
exists at least in a neighborhood of the origin of the form
$[-\rho_{i1},\rho_{i1}]$ for a suitable $\rho_{i1}>0$.

Moreover, if~$p_i\ge2$,
for any $\ell\in\{2,\dots, p_i\}$, we consider
the solution of the following Cauchy problem:
\begin{equation}\label{CAH2}
\begin{cases}
\partial^{r_{i\ell}}_{x_{i\ell}}\overline{v}_{i\ell}=\overline{v}_{i\ell} \\
\partial^{\beta_\ell}_{x_{i\ell}}\overline{v}_{i\ell}\left(0\right)=1\quad
\text{ for every } \beta_\ell\leq r_{i\ell}-1.
\end{cases}
\end{equation}
As above, these solutions
are well-defined at least in a neighborhood of the origin of the form $[-\rho_{i\ell},\rho_{i\ell}]$,
for a suitable $\rho_{i\ell}>0$.

Then, we define
\[\overline{\rho}_i:=\min\{ \rho_{i1},\dots,\rho_{i p_i}\}=\min_{\ell\in\{1,\dots,p_i\}}\rho_{i\ell}.\] 
In this way, for every~$x_i=(x_{i1},\dots,x_{ip_i})\in B^{p_i}_{\overline{\rho}_i}$,
we set
\begin{equation}\label{vba}
\overline{v}_i(x_i):=\overline{v}_{i1}(x_{i1})\ldots
\overline{v}_{i{p_i}}(x_{ip_i}).\end{equation}
By~\eqref{CAH1} and~\eqref{CAH2}, we have that
\begin{equation}
\label{cau}
\begin{cases}
\partial^{r_i}_{x_i}\overline{v}_i=-\overline{\XA}_i\overline{v}_i \\
\\
\partial^{\beta}_{x_i}\overline{v}_i\left(0\right)=1\quad
\begin{matrix}
&\text{ for every $\beta=(\beta_1,\dots\beta_{p_i})\in{\mathbb{N}}^{p_i}$}\\
&\text{ such that~$\beta_{\ell}\leq r_{i\ell}-1$ for each~$\ell\in\{1,\dots,p_i\}$.}\end{matrix}
\end{cases}
\end{equation}
Now, we define \[\rho:=\min\{ \overline{\rho}_1,\dots\overline{\rho}_n\}
=\min_{i\in\{1,\dots,n\}}\overline{\rho}_i.\]
We take 
\begin{equation*}
\overline{\tau}\in C_0^\infty\left(B^{p_1+\ldots+p_n}_{\rho/(R+2)}\right),
\end{equation*}
with $\overline{\tau}=1$ in $B^{p_1+\ldots+p_n}_{\rho/(2(R+2))}$, and,
for every~$x=(x_1,\dots,x_n)\in{\mathbb{R}}^{p_1}\times\dots\times{\mathbb{R}}^{p_n}$, we set 
\begin{equation}
\label{otau1}\tau_1\left(x_1,\ldots,x_n\right):=\overline{\tau}\left(\XX_1 \otimes x_1,\ldots,\XX_n \otimes x_n\right).\end{equation}
We recall that the free parameters~$\XX_1,\dots,\XX_n$
have been introduced in~\eqref{FREExi}, and we have used here the notation
$$ \XX_i\otimes x_i = (\XX_{i1},\dots,\XX_{ip_i})\otimes(x_{i1},\dots,x_{ip_i}):= (\XX_{i1}x_{i1},\dots,\XX_{ip_i}x_{ip_i})\in{\mathbb{R}}^{p_i},$$
for every~$i\in\{1,\dots,n\}$.

We also set, for any~$i\in\{1,\dots,n\}$,
\begin{equation}
\label{vuggei}
v_i\left(x_i\right):=\overline{v}_i\left(\XX_i\otimes x_i\right).
\end{equation}
We point out that if~$x_i\in B^{p_i}_{\overline\rho_i/(R+2)}$ we have that
$$ |\XX_i\otimes x_i|^2=\sum_{\ell=1}^{p_i}(\XX_{i\ell}x_{i\ell})^2\le
(R+2)^2\sum_{\ell=1}^{p_i}x_{i\ell}^2<\overline\rho_i^2,$$
thanks to~\eqref{FREExi}, and therefore the setting in~\eqref{vuggei}
is well-defined for every~$x_i\in B^{p_i}_{\overline\rho_i/(R+2)}$.

Recalling~\eqref{cau} and~\eqref{vuggei}, we see that, for any~$i\in\{1,\dots,n\}$,
\begin{equation}\label{QYHA0ow1dk} \partial_{x_i}^{r_i}v_i(x_i)=
\XX_i^{r_i}\partial_{x_i}^{r_i}\overline{v}_i\left(\XX_i\otimes x_i\right)
=
-\overline\XA_i \XX_i^{r_i}\overline{v}_i\left(\XX_i\otimes x_i\right)
=
-\overline\XA_i \XX_i^{r_i} {v}_i( x_i).\end{equation}
%Let also~\begin{equation}\label{otau2}\tau_2\in C^\infty_0([-30,30]^{l})\end{equation}
%with~$\tau_2=1$ in~$[-20,20]^{l}$.
We take $e_1,\ldots,e_M$, with
\begin{equation}
\label{econ}
e_j\in\partial B_{\omega_j}^{m_j},
\end{equation}
and we introduce an additional set of free parameters
$Y_1,\ldots,Y_M$ with 
\begin{equation}
\label{eq:FREEy}
Y_j\in\mathbb{R}^{m_j}\qquad{\mbox{
and }}\qquad e_j\cdot Y_j<0. \end{equation}
We let~$\epsilon>0$, to be taken small
possibly depending on the  free parameters $e_j$, $Y_j$ and $\XT_h$, 
and we define
\begin{equation}
\label{svs}
\begin{split}
w\left(x,y,t\right): 
=&\tau_1\left(x\right)v_1\left(x_1\right)\cdot
\ldots \cdot
v_n\left(x_n\right)\phi_1\left(y_1+e_1+\epsilon Y_1\right)\cdot
\ldots\cdot
\phi_M\left(y_M+e_M+\epsilon Y_M\right) \\
&\times\psi^\star_1(t_1)\cdot\ldots\cdot\psi^\star_l(t_l),
\end{split}
\end{equation}
where the setting in~\eqref{autofun1}, %% \eqref{autofun2},
\eqref{starest}, \eqref{otau1} and~\eqref{vuggei} has been exploited.

%We remark that, for every~$h\in\{1,\dots,l\}$,
%we have that~$\tau_2 (t_1,\dots,t_{h-1},\eta,t_{h+1},\dots,t_l)=1$
%when~$\eta\in(a_h,t_h)$ and~$|t_1|,\dots,|t_l|\le18$, and consequently
%\begin{equation}\label{TH92KA}
%\begin{split}&
%\Gamma([\alpha_h]+1-\alpha_h)\,D^{\alpha_h}_{t_h,a_h}
%\Big(\psi_1(t_1)\dots\psi_l(t_l)\tau_2(t)\Big)\\
%=\,&
%\int_{a_h}^{t_h} \frac{\partial^{[\alpha_h]+1}_{\eta}
%\Big(
%\psi_1(t_1)\dots\psi_{h-1}(t_{h-1})
%\psi_h(\eta)\psi_{h+1}(t_{h+1})\dots\psi_l(t_l)\;\tau_2 (t_1,\dots,t_{h-1},\eta,t_{h+1},\dots,t_l)
%\Big)}{(t_h-\eta)^{\alpha_h-[\alpha_h]}}\,d\eta\\
%=\,&
%\int_{a_h}^{t_h} \frac{\partial^{[\alpha_h]+1}_{\eta}
%\big(\psi_1(t_1)\dots\psi_{h-1}(t_{h-1})
%\psi_h(\eta)\psi_{h+1}(t_{h+1})\dots\psi_l(t_l)
%\big)}{(t_h-\eta)^{\alpha_h-[\alpha_h]}}\,d\eta\\
%=\,&\Gamma([\alpha_h]+1-\alpha_h)\,D^{\alpha_h}_{t_h,a_h}
%\Big(\psi_1(t_1)\dots\psi_l(t_l)\Big)\\
%=\,&\Gamma([\alpha_h]+1-\alpha_h)\,
%\psi_1(t_1)\dots\psi_{h-1}(t_{h-1})\,
%D^{\alpha_h}_{t_h,a_h}\psi_h(t_h)\,\psi_{h+1}(t_{h+1})\dots
%\psi_l(t_l),
%\end{split}\end{equation}
%as long as~$|t|\le 18$.

We also notice that $w\in C\left(\mathbb{R}^N\right)\cap C_0\left(\mathbb{R}^{N-l}\right)\cap\mathcal{A}$. Moreover,
if
\begin{equation}\label{pata7UJ:AKK}
a=(a_1,\dots,a_l):=\left(-\frac{\epsilon}{\XT_1},\dots,-\frac{\epsilon}{\XT_l}\right)\in\mathbb{R}^l\end{equation}
and~$(x,y)$ is sufficiently close to the origin
and~$t\in(a_1,+\infty)\times\dots\times(a_l,+\infty)$, we have that
\begin{eqnarray*}&&
\Lambda_{-\infty} w\left(x,y,t\right)\\&=&
\left( \sum_{i=1}^n \XA_i \partial^{r_i}_{x_i}
+\sum_{j=1}^{M} \XB_j (-\Delta)^{s_j}_{y_j}+
\sum_{h=1}^{l} \XC_h D^{\alpha_h}_{t_h,-\infty}\right) w\left(x,y,t\right)\\
&=&
\sum_{i=1}^n \XA_i 
v_1\left(x_1\right)
\ldots  
v_{i-1}\left(x_{i-1}\right)
\partial^{r_i}_{x_i}v_{i}\left(x_{i}\right)v_{i+1}\left(x_{i+1}\right)
\ldots  v_n\left(x_n\right)\\
&&\qquad\times\phi_1\left(y_1+e_1+\epsilon Y_1\right)
\ldots\phi_M\left(y_M+e_M+\epsilon Y_M\right)\psi^\star_1\left(t_1\right)
\ldots\psi^\star_l\left(t_l\right)\\
&&+\sum_{j=1}^{M} \XB_j 
v_1\left(x_1\right)
\ldots v_n\left(x_n\right)\phi_1\left(y_1+e_1+\epsilon Y_1\right)
\ldots\phi_{j-1}\left(y_{j-1}+e_{j-1}+\epsilon Y_{j-1}\right)\\&&\qquad\times
(-\Delta)^{s_j}_{y_j}\phi_j\left(y_j+e_j+\epsilon Y_j\right)
\phi_{j+1}\left(y_{j+1}+e_{j+1}+\epsilon Y_{j+1}\right)
\ldots\phi_M\left(y_M+e_M+\epsilon Y_M\right)\\
&&\qquad\times\psi^\star_1\left(t_1\right)
\ldots\psi^\star_l\left(t_l\right)\\
&&+\sum_{h=1}^{l} \XC_h 
v_1\left(x_1\right)
\ldots v_n\left(x_n\right)\phi_1\left(y_1+e_1+\epsilon Y_1\right)\ldots\phi_M\left(y_M+e_M+\epsilon Y_M\right)\psi^\star_1\left(t_1\right)
\ldots\psi^\star_{h-1}\left(t_{h-1}\right)\\
&&\qquad\times D^{\alpha_h}_{t_h,-\infty}\psi^\star_h\left(t_h\right)
\psi^\star_{h+1}(t_{h+1})\ldots\psi^\star_l\left(t_l\right)\\
&=&
-\sum_{i=1}^n \XA_i \overline\XA_i\XX_i^{r_i}
v_1\left(x_1\right)
\ldots v_n\left(x_n\right)\phi_1\left(y_1+e_1+\epsilon Y_1\right)
\ldots\phi_M\left(y_M+e_M+\epsilon Y_M\right)\psi^\star_1(t_1)\ldots\psi^\star_l(t_l)\\
&&+\sum_{j=1}^{M} \XB_j \lambda_j
v_1\left(x_1\right)
\ldots v_n\left(x_n\right)\phi_1\left(y_1+e_1+\epsilon Y_1\right)
\ldots\phi_M\left(y_M+e_M+\epsilon Y_M\right)\psi^\star_1(t_1)\ldots\psi^\star_l(t_l)
\\ 
&&+\sum_{h=1}^{l} \XC_h \XT_{\star,h}
v_1\left(x_1\right)
\ldots v_n\left(x_n\right)\phi_1\left(y_1+e_1+\epsilon Y_1\right)
\ldots\phi_M\left(y_M+e_M+\epsilon Y_M\right)\psi^\star_1(t_1)\ldots\psi^\star_l(t_l)
\\
&=&\left( -\sum_{i=1}^n \XA_i \overline\XA_i\XX_i^{r_i}
+\sum_{j=1}^{M} \XB_j \lambda_j+\sum_{h=1}^l \XC_h\XT_{\star,h}\right) w(x,y,t),
\end{eqnarray*}
thanks to \eqref{REGSWYS-A}, \eqref{DOBACHA} and~\eqref{QYHA0ow1dk}
.

Consequently, making use of~\eqref{1.6md}, \eqref{alp} and~\eqref{NOZABAnd},
if~$(x,y)$ lies near the origin and~$t\in(a_1,+\infty)\times\dots\times(a_l,+\infty)$,
we have that
\begin{eqnarray*}&&
\Lambda_{-\infty} w\left(x,y,t\right)=
\left( -\sum_{i=1}^n |\XA_i |\XX_i^{r_i}
+\sum_{j=1}^{M-1} \XB_j \lambda_{j}+\XB_M\lambda_M+\sum_{h=1}^l \XC_h\XT_{\star,h}\right) w(x,y,t)
\\&&\qquad=
\left( -\sum_{i=1}^n |\XA_i |\XX_i^{r_i}
+\sum_{j=1}^{M-1} \XB_j \lambda_{\star,j}+\XB_M\lambda_M+\sum_{h=1}^l \XC_h\XT_{\star,h}\right) w(x,y,t)=0.
\end{eqnarray*}
This says that~$w\in\mathcal{H}$. Thus, in light of~\eqref{aza} we have that
\begin{equation}
\label{eq:ort}
0=\theta\cdot\partial^K w\left(0\right)=
\sum_{\left|\iota\right|\leq K}
{\theta_{\iota}\partial^\iota w\left(0\right)}=\sum_{|i|+|I|+|\mathfrak{I}|\le K}
\theta_{i,I,\mathfrak{I}}\,\partial_{x}^{i}\partial_{y}^{I}\partial_{t}^{\mathfrak{I}}w\left(0\right) 
.\end{equation}
%Since~$w$ is constant in~$t$ near the origin,
%recalling the notation in~\eqref{mulPM}
%we have that~$\partial^\iota w(0)=
%\partial^{(i,I,\mathfrak{I})} w(0)=0$ if~$|\mathfrak{I}|>0$.
Now, we recall~\eqref{vba} and we claim that,
for any $j\in\{1,\dots,n\}$,
any~$\ell\in\{1,\dots, p_j\}$ and any~$
i_{j\ell}\in\mathbb{N}$, we have that
\begin{equation}
\label{notnull}
\partial^{i_{j\ell}}_{x_{j\ell}}\overline{v}_{j\ell}(0)\neq 0.
\end{equation}
We prove it by induction over~$i_{j\ell}$. Indeed, if 
$i_{j\ell}\in\left\{0,\ldots,r_{j\ell}-1\right\}$, then
the initial condition in \eqref{CAH1} (if~$\ell=1$)
or~\eqref{CAH2} (if~$\ell\ge2$) gives that~$
\partial^{i_{j\ell}}_{x_{i\ell}}\overline{v}_{i\ell}\left(0\right)=1$, and
so~\eqref{notnull}
is true in this case.

To perform the inductive step,
let us now
suppose that the claim in~\eqref{notnull}
still holds for all $i_{j\ell}\in\left\{0,\ldots,i_0\right\}$
for some~$i_0$ such that $i_0\geq r_{j\ell}-1$.
Then, using the equation in~\eqref{CAH1} (if~$\ell=1$)
or in~\eqref{CAH2} (if~$\ell\ge2$), we have that
\begin{equation}\label{8JAMANaoaksd}
\partial^{i_0+1}_{x_{j\ell}}\overline{v}_j=
\partial^{i_0+1-r_{j\ell}}_{x_{j\ell}}\partial^{r_{j\ell}}_{x_{j\ell}}
\overline{v}_j=-\tilde{a}_j\partial^{i_0+1-r_{j\ell}}_{x_{j\ell}}\overline{v}_j,
\end{equation}
with
$$\tilde{a}_j:=\begin{cases}
\overline{a}_j & {\mbox{ if }}\ell=1,\\
-1 & {\mbox{ if }}\ell\ge2.
\end{cases}$$
Notice that~$\tilde{a}_j\ne0$, in view of~\eqref{NOZABA},
and~$\partial^{i_0+1-r_{j\ell}}_{x_{j \ell}}\overline{v}_j\left(0\right)\neq 0$,
by the inductive assumption. These considerations and~\eqref{8JAMANaoaksd}
give that~$\partial^{i_0+1}_{x_{j\ell}}\overline{v}_j\left(0\right)\neq 0$,
and this proves~\eqref{notnull}. 

Now, using \eqref{vba} and~\eqref{notnull} we have that,
for any $j\in\{1,\dots,n\}$ and any~$
i_{j}\in\mathbb{N}^{p_j}$,
\begin{equation*}
\partial^{i_j}_{x_{j}}\overline{v}_{j}(0)\neq 0.
\end{equation*}
This, \eqref{FREExi} and the computation in~\eqref{QYHA0ow1dk} give that,
for any $j\in\{1,\dots,n\}$ and any~$
i_{j}\in\mathbb{N}^{p_j}$,
\begin{equation}
\label{eq:adoajfiap}
\partial^{i_j}_{x_{j}} {v}_{j}(0)=\XX_j^{i_j}\partial^{i_j}_{x_{j}}\overline{v}_{j}(0)\neq 0.
\end{equation}
We also notice that, in light of~\eqref{starest}, \eqref{svs} and~\eqref{eq:ort},
\begin{equation}\label{7UJHASndn2weirit}\begin{split}&
0=\sum_{|i|+|I|+|\mathfrak{I}|\le K}
\theta_{i,I,\mathfrak{I}}\,
\partial^{i_1}_{x_{1}} {v}_{1}(0)\ldots
\partial^{i_n}_{x_{n}} {v}_{n}(0)
\,\partial_{y_1}^{I_1}\phi_1\left(e_1+\epsilon Y_1\right)
\ldots\partial_{y_M}^{I_M}\phi_M\left(e_M+\epsilon Y_M\right)\\
&\qquad\qquad\qquad\times\partial^{\mathfrak{I}_1}_{t_1}
\psi_1(0)\ldots\partial^{\mathfrak{I}_l}\psi_l(0).\end{split}\end{equation}
Now, by~\eqref{autofun1}
and Proposition \ref{sharbou}
(applied to $s:=s_j$, $\beta:=I_j$, $e:=\frac{e_j}{\omega_j}\in\partial B_1^{m_j}$, due to~\eqref{econ}, and $X:=\frac{Y_j}{\omega_j}$),
we see that, for any $j=1,\ldots, M$,
\begin{equation}
\label{alppoo}
\begin{split}\omega_j^{|I_j|}
\lim_{\epsilon\searrow 0}\epsilon^{|I_j|-s_j}\partial_{y_j}^{I_j}
\phi_j\left( e_j+\epsilon Y_j \right)
\; =\;&
\lim_{\epsilon\searrow 0}\epsilon^{|I_j|-s_j}\partial_{y_j}^{I_j}
\tilde\phi_{\star,j}\left(\frac{e_j+\epsilon Y_j}{\omega_j}\right)
\\=\;& \kappa_j \frac{e_j^{I_j}}{\omega_j^{|I_j|}}\left(-\frac{e_j}{\omega_j}\cdot\frac{Y_j}{\omega_j}\right)_+^{s_j-|I_j|}
,\end{split}\end{equation}
with~$\kappa_j\ne0$, in the sense of distributions (in the coordinates~$Y_j$).

Moreover, using~\eqref{STAvca} and \eqref{pata7UJ:AKK},
it follows that
\begin{eqnarray*}
\partial^{\mathfrak{I}_h}_{t_h}\psi_h(0)&=&
\sum_{j=0}^{+\infty} {\frac{
\XT_{\star,h}^j\, \alpha_h j(\alpha_h j-1)\dots(\alpha_h j-
\mathfrak{I}_h+1)
(0-a_h)^{\alpha_h j-\mathfrak{I}_h}}{\Gamma\left(\alpha_h j+1\right)}}\\
&=& \sum_{j=0}^{+\infty} {\frac{
\XT_{\star,h}^j\, \alpha_h j(\alpha_h j-1)\dots(\alpha_h j-
\mathfrak{I}_h+1)
\,{\epsilon}^{\alpha_h j-\mathfrak{I}_h}}{\Gamma\left(\alpha_h j+1\right)\;
{\XT_h}^{\alpha_h j-\mathfrak{I}_h}}}\\
&=& \sum_{j=1}^{+\infty} {\frac{
\XT_{\star,h}^j\, \alpha_h j(\alpha_h j-1)\dots(\alpha_h j-
\mathfrak{I}_h+1)
\,{\epsilon}^{\alpha_h j-\mathfrak{I}_h}}{\Gamma\left(\alpha_h j+1\right)\;
{\XT_h}^{\alpha_h j-\mathfrak{I}_h}}}.
\end{eqnarray*}
Accordingly, recalling~\eqref{TGAdef}, we find that
\begin{equation}
\label{limmittl}\begin{split}&
\lim_{\epsilon\searrow 0}\epsilon^{\mathfrak{I}_h-\alpha_h}
\partial^{\mathfrak{I}_h}_{t_h}\psi_h(0)
=\lim_{\epsilon\searrow 0}
\sum_{j=1}^{+\infty} {\frac{
\XT_{\star,h}^j\, \alpha_h j(\alpha_h j-1)\dots(\alpha_h j-
\mathfrak{I}_h+1)
\,{\epsilon}^{\alpha_h (j-1)}}{\Gamma\left(\alpha_h j+1\right)\;
{\XT_h}^{\alpha_h j-\mathfrak{I}_h}}}
\\&\qquad=
{\frac{ \XT_{\star,h}\,\alpha_h (\alpha_h -1)\dots(\alpha_h -\mathfrak{I}_h+1)
}{\Gamma\left(\alpha_h +1\right)\;
\XT_h^{\alpha_h -\mathfrak{I}_h}}} =
{\frac{ \XT_h^{\mathfrak{I}_h}\,\alpha_h (\alpha_h -1)\dots(\alpha_h -\mathfrak{I}_h+1)
}{\Gamma\left(\alpha_h +1\right)}}.
\end{split}\end{equation}
Also, recalling~\eqref{IBARRA}, we can write~\eqref{7UJHASndn2weirit} as
\begin{equation}
\label{eq:ort:X}\begin{split}&
0=\sum_{{|i|+|I|+|\mathfrak{I}|\le K}\atop{|I|\le|\overline{I}|}}
\theta_{i,I,\mathfrak{I}}\,
\partial^{i_1}_{x_{1}} {v}_{1}(0)\ldots
\partial^{i_n}_{x_{n}} {v}_{n}(0)
\,\partial_{y_1}^{I_1}\phi_1\left(e_1+\epsilon Y_1\right)
\ldots\partial_{y_M}^{I_M}\phi_M\left(e_M+\epsilon Y_M\right)\\&\qquad\qquad\times
\partial^{\mathfrak{I}_1}_{t_1}\psi_{1}(0)\ldots
\partial^{\mathfrak{I}_l}_{t_l}\psi_{l}(0).\end{split}
\end{equation}
Moreover, we define
\begin{equation*}
\Xi:=\left|\overline{I}\right|-\sum_{j=1}^M {s_j}+|\mathfrak{I}|-\sum_{h=1}^l {\alpha_h}.
\end{equation*}
Then, we
multiply~\eqref{eq:ort:X} by $\epsilon^{\Xi}\in(0,+\infty)$, and we
send~$\epsilon$ to zero. In this way, we obtain from~\eqref{alppoo},
\eqref{limmittl}
and~\eqref{eq:ort:X} that
\begin{eqnarray*}
0&=&\lim_{\epsilon\searrow0}
\epsilon^{\Xi}
\sum_{{|i|+|I|+|\mathfrak{I}|\le K}\atop{|I|\le|\overline{I}|}}
\theta_{i,I,\mathfrak{I}}\,
\partial^{i_1}_{x_{1}} {v}_{1}(0)\ldots
\partial^{i_n}_{x_{n}} {v}_{n}(0)
\,\partial_{y_1}^{I_1}\phi_1\left(e_1+\epsilon Y_1\right)
\ldots\partial_{y_M}^{I_M}\phi_M\left(e_M+\epsilon Y_M\right)
\\ &&\qquad\times\partial^{\mathfrak{I}_1}_{t_1}\psi_1(0)\ldots\partial^{\mathfrak{I}_l}_{t_l}\psi_l(0)
\\ &=&\lim_{\epsilon\searrow0}
\sum_{{|i|+|I|+|\mathfrak{I}|\le K}\atop{|I|\le|\overline{I}|}}
\epsilon^{|\overline{I}|-|I|}
\theta_{i,I,\mathfrak{I}}\,
\partial^{i_1}_{x_{1}} {v}_{1}(0)\ldots
\partial^{i_n}_{x_{n}} {v}_{n}(0)\\
&&\qquad\times\epsilon^{|I_1|-s_1}\partial_{y_1}^{I_1}\phi_1\left(e_1+\epsilon Y_1\right)
\ldots\epsilon^{|I_M|-s_M}\partial_{y_M}^{I_M}\phi_M\left(e_M+\epsilon Y_M\right)
\\ &&\qquad\times\epsilon^{\mathfrak{I}_1-\alpha_1}\partial^{\mathfrak{I}_1}_{t_1}\psi_1(0)\ldots\epsilon^{\mathfrak{I}_l-\alpha_l}\partial^{\mathfrak{I}_l}_{t_l}\psi_l(0)
\\&=&
\sum_{{|i|+|I|+|\mathfrak{I}|\le K}\atop{|I| = |\overline{I}|}}
\tilde C_{i,I,\mathfrak{I}}\,\theta_{i,I,\mathfrak{I}}\,
\partial^{i_1}_{x_{1}} {v}_{1}(0)\ldots
\partial^{i_n}_{x_{n}} {v}_{n}(0)\\\
&&\qquad\times e_1^{I_1}\ldots e_M^{I_M}\,
\left(- e_1 \cdot Y_1 \right)_+^{s_1-|I_1|}
\ldots
\left(- e_M \cdot Y_M\right)_+^{s_M-|I_M|}\XT_1^{\mathfrak{I}_1}\ldots\XT_l^{\mathfrak{I}_l}
,\end{eqnarray*}
for a suitable~$\tilde C_{i,I,\mathfrak{I}}\ne0$
(strictly speaking, the above identity holds
in the sense of distribution with respect to the coordinates~$Y$
and~$\XT$, but since the left hand side vanishes,
we can consider it also a pointwise identity).

Hence, recalling~\eqref{eq:adoajfiap},
\begin{equation}\label{GANAfai}\begin{split}
0&\,=\,\sum_{{|i|+|I|+|\mathfrak{I}|\le K}\atop{|I| = |\overline{I}|}}
C_{i,I,\mathfrak{I}}\,\theta_{i_1,\dots,i_n,I_1,\dots,I_M,\mathfrak{I}_1,\dots,\mathfrak{I}_l}\;
\XX_1^{i_1}\ldots\XX_n^{i_n}\\&\qquad\qquad\times
e_1^{I_1}\ldots e_M^{I_M}\,
\left(- e_1 \cdot Y_1 \right)_+^{s_1-|I_1|}
\ldots
\left(- e_M \cdot Y_M\right)_+^{s_M-|I_M|}
\XT_1^{\mathfrak{I}_1}\ldots\XT_l^{\mathfrak{I}_l}\\
&\,=\,
\left(- e_1 \cdot Y_1 \right)_+^{s_1}
\ldots
\left(- e_M \cdot Y_M\right)_+^{s_M}\\&\qquad\qquad\times
\sum_{{|i|+|I|+|\mathfrak{I}|\le K}\atop{|I| = |\overline{I}|}}
C_{i,I,\mathfrak{I}}\,\theta_{i,I,\mathfrak{I}}\;
\XX^i\,e^{I}\,
\left(- e_1 \cdot Y_1 \right)_+^{-|I_1|}
\ldots
\left(- e_M \cdot Y_M\right)_+^{-|I_M|}
\XT^{\mathfrak{I}}
,\end{split}\end{equation}
for a suitable~$C_{i,I,\mathfrak{I}}\ne0$.

We observe that the equality in~\eqref{GANAfai}
is valid for any choice of the free parameters~$(\XX,Y,\XT)$
in an open subset of~$\mathbb{R}^{p_1+\dots+p_n}\times
\mathbb{R}^{m_1+\dots+m_M}\times\mathbb{R}^l$,
as prescribed in~\eqref{FREExi},~\eqref{FREEmustar}
and~\eqref{eq:FREEy}. 

Now, we take new free parameters, $\XY_1,\ldots,\XY_M$ with $\XY_j\in\mathbb{R}^{m_j}\setminus\{0\}$, and we
define
\begin{equation}\label{COMPATI}
e_j:=\frac{\omega_j\XY_j}{|\XY_j|}\quad {\mbox{ and }}\quad 
Y_j:=-\frac{\XY_j}{|\XY_j|^2}.\end{equation}
We stress that the setting in~\eqref{COMPATI} is compatible with that in~\eqref{eq:FREEy}, since
$$ e_j\cdot Y_j=-\frac{\omega_j\XY_j}{|\XY_j|}\cdot
\frac{\XY_j}{|\XY_j|^2}=-\frac{\omega_j}{|\XY_j|}<0,$$
thanks to~\eqref{OMEj}. We also notice that, for all~$j\in\{1,\dots,M\}$,
$$ e_j^{I_j}\left(- e_j \cdot Y_j \right)_+^{-|I_j|}=
\frac{\omega_j^{|I_j|}\XY_j^{I_j}}{|\XY_j|^{|I_j|}}\,
\frac{|\XY_j|^{|I_j|}}{\omega_j^{|I_j|}}=\XY_j^{I_j},
$$
and hence
$$ e^{I}\,
\left(- e_1 \cdot Y_1 \right)_+^{-|I_1|}
\ldots
\left(- e_M \cdot Y_M\right)_+^{-|I_M|}=
\XY^I.$$
Plugging this into formula \eqref{GANAfai},
we obtain the first identity in~\eqref{ipop},
as desired.
Hence, 
the proof of~\eqref{ipop} in case~\ref{itm:case1}
is complete.
\end{proof}

\begin{proof}[Proof of \eqref{ipop}, case \ref{itm:case2}]
Thanks to the assumptions given in case~\ref{itm:case2}, we can suppose
that formula~\eqref{AGZ} still holds, and also that
\begin{equation}\label{MAGGZC}
{\mbox{$\XC_l>0$}}.\end{equation} In addition,
for any $j\in\{1,\ldots,M\}$, we consider $\lambda_j$
and~$\phi_j$ as in \eqref{REGSWYS-A}.

Then, we define
\begin{equation}\label{MAfghjkGGZC} R:=\left( 
\frac{ 1 }{|\XA_1|}\displaystyle\left(\sum_{h=1}^{l-1}|\XC_h|+\sum_{j=1}^M|\XB_j|\lambda_j\right)\right)^{1/|r_{1}|}.\end{equation}
We notice that, in light of \eqref{AGZ}, the setting in~\eqref{MAfghjkGGZC} is well-defined. 

Now, we fix two sets of free parameters $\XX_1,\ldots,\XX_n$
as in~\eqref{FREExi}
and~$\,\XT_{\star,1},\dots,\XT_{\star,l}$ as in~\eqref{FREEmustar}, here taken with~$R$
as in~\eqref{MAfghjkGGZC}.
Moreover, we define
\begin{equation}
\label{alpc}
\lambda\,:=\,\frac{1}{\XC_l\,\XT_{\star,l}}\left(
\sum_{j=1}^n {\left|\XA_j\right|\XX_j^{r_j}}-\sum_{j=1}^M\XB_j\lambda_j-\sum_{h=1}^{l-1}\XC_h\XT_{\star,h}\right).\end{equation}
We notice that \eqref{alpc} is well-defined,
thanks to~\eqref{FREEmustar}
and~\eqref{MAGGZC}.
Furthermore, recalling~\eqref{FREExi}, \eqref{1.15bis} and~\eqref{MAfghjkGGZC},
we find that
\begin{equation*}
\begin{split}
\sum_{i=1}^n&|\XA_i|\XX_i^{r_i}\ge|\XA_1|\XX_1^{r_1}>
|\XA_1|(R+1)^{|r_1|}>|\XA_1|R^{|r_1|} \\
&=\sum_{h=1}^{l-1}|\XC_h|+\sum_{j=1}^M|\XB_j|\lambda_j
\ge\sum_{h=1}^{l-1}\XC_h\XT_{\star,h}+\sum_{j=1}^M\XB_j\lambda_j.
\end{split}
\end{equation*}
Consequently, by~\eqref{alpc},
\begin{equation}
\label{alp-0c}
\lambda>0.\end{equation}
Hence, we can define
\begin{equation}\label{OVlam}
\overline{\lambda}:=\lambda^{1/{\alpha_l}}.\end{equation}
Moreover, 
we consider~$a_h\in(-2,0)$, for every~$h\in\{ 1,\dots,l\}$,
to be chosen appropriately in what follows
(the exact choice will be performed in~\eqref{pata7UJ:AKKcc}),
and, using the notation in~\eqref{chosofpsistar}
and~\eqref{TGAdef}, we define
\begin{equation}
\label{autofun2c}
\psi_h(t_h):=\psi_{\star,h}\big(\XT_h (t_h-a_h)\big)=
E_{\alpha_h,1}\big(\XT_{\star,h} (t_h-a_h)^{\alpha_h}\big)\quad\text{if }\,h\in\{1,\dots,l-1\}
\end{equation}
and
\begin{equation}
\label{psielle}
\psi_{l}(t_l):=\psi_{\star,l}\big(\overline{\lambda}\,\XT_l (t_l-a_l)\big)=
E_{\alpha_l,1}\big(\lambda\,\XT_{\star,l} (t_l-a_l)^{\alpha_l}\big).
\end{equation}
We recall that, thanks to Lemma~\ref{MittagLEMMA}, the function in~\eqref{autofun2c}
solves~\eqref{jhjadwlgh} and satisfies~\eqref{STAvca} for any $h\in\{1,\dots,l-1\}$, while the function
in~\eqref{psielle} solves
\begin{equation}
\label{jhjadwlghplop}
\begin{cases}
D^{\alpha_l}_{t_l,a_l}\psi_l(t_l)=\lambda\,\XT_{\star,l}\psi_l(t_l)&\quad\text{in }\,(a_l,+\infty), \\
\psi_l(a_l)=1, \\
\partial^m_{t_l}\psi_l(a_l)=0&\quad\text{for every }\,m\in\{1,\dots,[\alpha_l] \}.
\end{cases}
\end{equation}
As in~\eqref{starest}, we extend the functions~$\psi_{h}$
constantly in~$(-\infty,a_h)$, calling~$\psi^\star_h$ this
extended function. In this way,
Lemma~A.3 in~\cite{CDV18} translates~\eqref{jhjadwlghplop} into
\begin{equation}
\label{jhjadwlghplop2}
D^{\alpha_h}_{t_h,-\infty}\psi_h^\star(t_h)=\XT_{\star,h}\psi_h(t_h)
=\XT_{\star,h}\psi_h^\star(t_h)\,\text{ in every interval }\,I\Subset(a_h,+\infty).
\end{equation}
Now, we let~$\epsilon>0$, to be taken small
possibly depending on the  free parameters,
and we exploit the functions defined in \eqref{otau1} and \eqref{vuggei},
provided that
one replaces the positive constant $R$ defined in \eqref{MAfghjkGGZ}
with the one in \eqref{MAfghjkGGZC}, when necessary. 

With this idea in mind, for any $j\in\{1,\ldots,M\}$, we let\footnote{Comparing~\eqref{freee2}
with \eqref{econ}, we observe that~\eqref{econ} reduces to~\eqref{freee2} with the choice~$\omega_j:=1$.}
\begin{equation}
\label{freee2}
e_j\in\partial B^{m_j}_1,
\end{equation}
and we define
\begin{equation}
\label{svsc}
\begin{split}
w\left(x,y,t\right): 
=&\tau_1\left(x\right)v_1\left(x_1\right)\cdot
\ldots \cdot
v_n\left(x_n\right)\phi_1\left(y_1+e_1+\epsilon Y_1\right)\cdot
\ldots\cdot
\phi_M\left(y_M+e_M+\epsilon Y_M\right) \\
&\times\psi^\star_1(t_1)\cdot\ldots\cdot\psi^\star_l(t_l),
\end{split}
\end{equation}
where the setting in~\eqref{REGSWYS-A}, %% \eqref{starest}, 
\eqref{otau1},
\eqref{vuggei}, \eqref{eq:FREEy}, \eqref{autofun2c}
and~\eqref{psielle} has been exploited. 

We also notice that $w\in C\left(\mathbb{R}^N\right)\cap C_0(\mathbb{R}^{N-l})\cap\mathcal{A}$. Moreover,
if
\begin{equation}\label{pata7UJ:AKKcc}
a=(a_1,\dots,a_l):=\left(-\frac{\epsilon}{\XT_1},\dots,-\frac{\epsilon}{\,\XT_l}\right)\in\mathbb{R}^l\end{equation}
and~$(x,y)$ is sufficiently close to the origin
and~$t\in(a_1,+\infty)\times\dots\times(a_l,+\infty)$, we have that
\begin{eqnarray*}&&
\Lambda_{-\infty} w\left(x,y,t\right)\\&=&
\left( \sum_{i=1}^n \XA_i \partial^{r_i}_{x_i}
+\sum_{j=1}^{M} \XB_j (-\Delta)^{s_j}_{y_j}+
\sum_{h=1}^{l} \XC_h D^{\alpha_h}_{t_h,-\infty}\right) w\left(x,y,t\right)\\
&=&
\sum_{i=1}^n \XA_i 
v_1\left(x_1\right)
\ldots  
v_{i-1}\left(x_{i-1}\right)
\partial^{r_i}_{x_i}v_{i}\left(x_{i}\right)v_{i+1}\left(x_{i+1}\right)
\ldots  v_n\left(x_n\right)\\
&&\qquad\times\phi_1\left(y_1+e_1+\epsilon Y_1\right)
\ldots\phi_M\left(y_M+e_M+\epsilon Y_M\right)\psi^\star_1\left(t_1\right)
\ldots\psi^\star_{l-1}\left(t_{l-1}\right)\psi^\star_l\left(t_l\right)\\
&&+\sum_{j=1}^{M} \XB_j 
v_1\left(x_1\right)
\ldots v_n\left(x_n\right)\phi_1\left(y_1+e_1+\epsilon Y_1\right)
\ldots\phi_{j-1}\left(y_{j-1}+e_{j-1}+\epsilon Y_{j-1}\right)\\&&\qquad\times
(-\Delta)^{s_j}_{y_j}\phi_j\left(y_j+e_j+\epsilon Y_j\right)
\phi_{j+1}\left(y_{j+1}+e_{j+1}+\epsilon Y_{j+1}\right)
\ldots\phi_M\left(y_M+e_M+\epsilon Y_M\right)\\
&&\qquad\times\psi^\star_1\left(t_1\right)
\ldots\psi^\star_{l-1}\left(t_{l-1}\right)\psi^\star_l\left(t_l\right)\\
&&+\sum_{h=1}^{l} \XC_h 
v_1\left(x_1\right)
\ldots v_n\left(x_n\right)\phi_1\left(y_1+e_1+\epsilon Y_1\right)\ldots\phi_M\left(y_M+e_M+\epsilon Y_M\right)\psi^\star_1\left(t_1\right)
\ldots\psi^\star_{h-1}\left(t_{h-1}\right)\\
&&\qquad\times D^{\alpha_h}_{t_h,-\infty}\psi^\star_h\left(t_h\right)
\psi^\star_{h+1}(t_{h+1})\ldots\psi^\star_{l-1}\left(t_{l-1}\right)\psi^\star_l\left(t_l\right)\\
&=&
-\sum_{i=1}^n \XA_i \overline\XA_i\XX_i^{r_i}
v_1\left(x_1\right)
\ldots v_n\left(x_n\right)\phi_1\left(y_1+e_1+\epsilon Y_1\right)
\ldots\phi_M\left(y_M+e_M+\epsilon Y_M\right)\\
&&\qquad\times\psi^\star_1(t_1)\ldots\psi^\star_{l-1}(t_{l-1})\psi^\star_l(t_l)\\
&&+\sum_{j=1}^{M} \XB_j \lambda_j
v_1\left(x_1\right)
\ldots v_n\left(x_n\right)\phi_1\left(y_1+e_1+\epsilon Y_1\right)
\ldots\phi_M\left(y_M+e_M+\epsilon Y_M\right)\\&&\qquad\times\psi^\star_1(t_1)\ldots\psi^\star_{l-1}\left(t_{l-1}\right)\psi^\star_l(t_l)
\\ 
&&+\sum_{h=1}^{l-1} \XC_h \XT_{\star,h}
v_1\left(x_1\right)
\ldots v_n\left(x_n\right)\phi_1\left(y_1+e_1+\epsilon Y_1\right)
\ldots\phi_M\left(y_M+e_M+\epsilon Y_M\right)\\&&\qquad\times\psi^\star_1(t_1)\ldots\psi^\star_{l-1}(t_{l-1})\psi^\star_l(t_l)
\\
&&+\XC_l\lambda\XT_{\star,l}v_1\left(x_1\right)
\ldots v_n\left(x_n\right)\phi_1\left(y_1+e_1+\epsilon Y_1\right)
\ldots\phi_M\left(y_M+e_M+\epsilon Y_M\right)\\
&&\qquad\times\psi^\star_1(t_1)\ldots\psi^\star_{l-1}(t_{l-1})\psi^\star_l(t_l) \\
&=&\left( -\sum_{i=1}^n \XA_i \overline\XA_i\XX_i^{r_i}
+\sum_{j=1}^M\XB_j\lambda_j+\sum_{h=1}^{l-1} \XC_h\XT_{\star,h}+\XC_l\lambda\XT_{\star,l}\right) w(x,y,t),
\end{eqnarray*}
thanks to \eqref{REGSWYS-A}, \eqref{jhjadwlgh}, \eqref{QYHA0ow1dk}
and~\eqref{jhjadwlghplop2}.

Consequently, making use of \eqref{NOZABAnd} and~\eqref{alpc}, when~$(x,y)$ is
near the origin and~$t\in(a_1,+\infty)\times\dots\times(a_l,+\infty)$,
we have that
$$
\Lambda_{-\infty} w\left(x,y,t\right)=
\left( -\sum_{i=1}^n |\XA_i |\XX_i^{r_i}
+\sum_{j=1}^M\XB_j\lambda_j+\sum_{h=1}^{l-1} \XC_h\XT_{\star,h}+\lambda\XC_l\XT_{\star,l}\right) w(x,y,t)=0
.$$
This says that~$w\in\mathcal{H}$. Thus, in light of~\eqref{aza} we have that
\begin{equation*}
0=\theta\cdot\partial^K w\left(0\right)=
\sum_{\left|\iota\right|\leq K}
{\theta_{\iota}\partial^\iota w\left(0\right)}=\sum_{|i|+|I|+|\mathfrak{I}|\le K}
\theta_{i,I,\mathfrak{I}}\,\partial_{x}^{i}\partial_{y}^I\partial_{t}^{\mathfrak{I}}w\left(0\right) 
.\end{equation*}
Hence, in view of~\eqref{eq:adoajfiap} and~\eqref{svsc},
\begin{equation}\label{1.64bis}
\begin{split}
0\,&=\,\sum_{|i|+|I|+|\mathfrak{I}|\le K}
\theta_{i,I,\mathfrak{I}}\,
\partial^{i_1}_{x_{1}} {v}_{1}(0)\ldots
\partial^{i_n}_{x_{n}} {v}_{n}(0)\\&\qquad\times
\partial^{I_1}_{y_{1}} {\phi}_{1}(e_1+\epsilon Y_1)\ldots
\partial^{I_M}_{y_{M}} {\phi}_{M}(e_M+\epsilon Y_M)
\,\partial_{t_1}^{\mathfrak{I}_1}\psi_1(0)
\ldots\partial_{t_l}^{\mathfrak{I}_l}\psi_l(0)
\\&=
\,\sum_{|i|+|I|+|\mathfrak{I}|\le K}
\theta_{i,I,\mathfrak{I}}\,\XX_1^{r_1}\ldots\XX_n^{r_n}\,
\partial^{i_1}_{x_{1}} \overline{v}_{1}(0)\ldots
\partial^{i_n}_{x_{n}} \overline{v}_{n}(0)\\&\qquad\times
\partial^{I_1}_{y_{1}} {\phi}_{1}(e_1+\epsilon Y_1)\ldots
\partial^{I_M}_{y_{M}} {\phi}_{M}(e_M+\epsilon Y_M)
\,\partial_{t_1}^{\mathfrak{I}_1}\psi_1(0)
\ldots\partial_{t_l}^{\mathfrak{I}_l}\psi_l(0)
.\end{split}\end{equation}
Moreover, using~\eqref{Mittag}, \eqref{psielle}
and \eqref{pata7UJ:AKKcc},
it follows that
\begin{eqnarray*}
\partial^{\mathfrak{I}_l}_{t_l}\psi_l(0)&=&
\sum_{j=0}^{+\infty} {\frac{
\lambda^j\,\XT_{\star,l}^j\, \alpha_l j(\alpha_l j-1)\dots(\alpha_l j-
\mathfrak{I}_l+1)
(0-a_l)^{\alpha_l j-\mathfrak{I}_l}}{\Gamma\left(\alpha_l j+1\right)}}\\
&=& \sum_{j=0}^{+\infty} {\frac{
\lambda^j\,\XT_{\star,l}^j\, \alpha_l j(\alpha_l j-1)\dots(\alpha_l j-
\mathfrak{I}_l+1)
\,{\epsilon}^{\alpha_l j-\mathfrak{I}_l}}{\Gamma\left(\alpha_l j+1\right)\;
{\XT_l}^{\alpha_l j-\mathfrak{I}_l}}}\\
&=& \sum_{j=1}^{+\infty} {\frac{
\lambda^j\,\XT_{\star,l}^j\, \alpha_l j(\alpha_l j-1)\dots(\alpha_l j-
\mathfrak{I}_l+1)
\,{\epsilon}^{\alpha_l j-\mathfrak{I}_l}}{\Gamma\left(\alpha_l j+1\right)\;
{\XT_l}^{\alpha_l j-\mathfrak{I}_l}}}.
\end{eqnarray*}
Accordingly, by~\eqref{TGAdef}, we find that
\begin{equation}
\label{limmittlc}\begin{split}&
\lim_{\epsilon\searrow 0}\epsilon^{\mathfrak{I}_l-\alpha_l}
\partial^{\mathfrak{I}_l}_{t_l}\psi_l(0)
=\lim_{\epsilon\searrow 0}
\sum_{j=1}^{+\infty} {\frac{
\lambda^j\,\XT_{\star,l}^j\, \alpha_l j(\alpha_l j-1)\dots(\alpha_l j-
\mathfrak{I}_l+1)
\,{\epsilon}^{\alpha_l (j-1)}}{\Gamma\left(\alpha_l j+1\right)\;
{\XT_l}^{\alpha_l j-\mathfrak{I}_l}}}
\\&\qquad=
{\frac{ \lambda\,\XT_{\star,l}\,\alpha_l (\alpha_l -1)\dots(\alpha_l -\mathfrak{I}_l+1)
}{\Gamma\left(\alpha_l +1\right)\;
\XT_l^{\alpha_l -\mathfrak{I}_l}}} =
{\frac{ \lambda\,\XT_l^{\mathfrak{I}_l}\,\alpha_l (\alpha_l -1)\dots(\alpha_l -\mathfrak{I}_l+1)
}{\Gamma\left(\alpha_l +1\right)}}.
\end{split}\end{equation}
Hence, recalling~\eqref{IBARRA2}, we can write~\eqref{1.64bis} as
\begin{equation}
\label{eq:ort:Xc}\begin{split}&
0=\sum_{{|i|+|I|+|\mathfrak{I}|\le K}\atop{|\mathfrak{I}|\le|\overline{\mathfrak{I}}|}}
\theta_{i,I,\mathfrak{I}}\,\XX_1^{r_1}\ldots\XX_n^{r_n}\,
\partial^{i_1}_{x_{1}} \overline{v}_{1}(0)\ldots
\partial^{i_n}_{x_{n}} \overline{v}_{n}(0)\\
&\qquad\qquad\times
\partial^{I_1}_{y_{1}} {\phi}_{1}(e_1+\epsilon Y_1)\ldots
\partial^{I_M}_{y_{M}} {\phi}_{M}(e_M+\epsilon Y_M) \partial^{\mathfrak{I}_1}_{t_1}\psi_{1}(0)\ldots
\partial^{\mathfrak{I}_l}_{t_l}\psi_{l}(0).\end{split}
\end{equation}
Moreover, we define
\begin{equation*}
\Xi:=|\overline{\mathfrak{I}}|-\sum_{h=1}^l {\alpha_h}+\left|I\right|-
\sum_{j=1}^M {s_j}.
\end{equation*}
Then, we
multiply~\eqref{eq:ort:Xc} by $\epsilon^{\Xi}\in(0,+\infty)$, and we
send~$\epsilon$ to zero. In this way, we obtain from~\eqref{limmittl}, used here for~$h\in\{1,\ldots, l-1\}$,
\eqref{limmittlc} and~\eqref{eq:ort:Xc} that
\begin{eqnarray*}
0&=&\lim_{\epsilon\searrow0}
\epsilon^{\Xi}
\sum_{{|i|+|I|+|\mathfrak{I}|\le K}\atop{|\mathfrak{I}|\le|\overline{\mathfrak{I}}|}}
\theta_{i,I,\mathfrak{I}}\,
\,\XX_1^{r_1}\ldots\XX_n^{r_n}\,
\partial^{i_1}_{x_{1}} \overline{v}_{1}(0)\ldots
\partial^{i_n}_{x_{n}} \overline{v}_{n}(0)\\
&&\qquad\times\partial_{y_1}^{I_1}\phi_1\left(e_1+\epsilon Y_1\right)
\ldots\partial_{y_M}^{I_M}\phi_M\left(e_M+\epsilon Y_M\right)
\\ &&\qquad\times\partial^{\mathfrak{I}_1}_{t_1}\psi_1(0)\ldots\partial^{\mathfrak{I}_l}_{t_l}\psi_l(0)
\\ &=&\lim_{\epsilon\searrow0}
\sum_{{|i|+|I|+|\mathfrak{I}|\le K}\atop{|\mathfrak{I}|\le|\overline{\mathfrak{I}}|}}
\epsilon^{|\overline{\mathfrak{I}}|-|\mathfrak{I}|}
\theta_{i,I,\mathfrak{I}}
\,\XX_1^{r_1}\ldots\XX_n^{r_n}\,
\partial^{i_1}_{x_{1}} \overline{v}_{1}(0)\ldots
\partial^{i_n}_{x_{n}} \overline{v}_{n}(0)\\
&&\qquad\times\epsilon^{|I_1|-s_1}\partial_{y_1}^{I_1}\phi_1\left(e_1+\epsilon Y_1\right)
\ldots\epsilon^{|I_M|-s_M}\partial_{y_M}^{I_M}\phi_M\left(e_M+\epsilon Y_M\right)
\\ &&\qquad\times\epsilon^{\mathfrak{I}_1-\alpha_1}\partial^{\mathfrak{I}_1}_{t_1}\psi_1(0)\ldots\epsilon^{\mathfrak{I}_l-\alpha_l}\partial^{\mathfrak{I}_l}_{t_l}\psi_l(0)
\\&=&
\sum_{{|i|+|I|+|\mathfrak{I}|\le K}\atop{|\mathfrak{I}| = |\overline{\mathfrak{I}}|}}
\lambda\,\tilde C_{i,I,\mathfrak{I}}\,\theta_{i,I,\mathfrak{I}}\,\XX_1^{r_1}\ldots\XX_n^{r_n}\,
\partial^{i_1}_{x_{1}} \overline{v}_{1}(0)\ldots
\partial^{i_n}_{x_{n}} \overline{v}_{n}(0)\\\
&&\qquad\times e_1^{I_1}\ldots e_M^{I_M}\,
\left(- e_1 \cdot Y_1 \right)_+^{s_1-|I_1|}
\ldots
\left(- e_M \cdot Y_M\right)_+^{s_M-|I_M|}\XT_1^{\mathfrak{I}_1}\ldots\XT_l^{\mathfrak{I}_l}
,\end{eqnarray*}
for a suitable~$\tilde C_{i,I,\mathfrak{I}}$. We stress that~$\tilde C_{i,I,\mathfrak{I}}\ne0$,
thanks also to~\eqref{alppoo},
applied here with~$\omega_j:=1$, $\tilde{\phi}_{\star,j}:=\phi_j$
and $e_j$ as in \eqref{freee2} for any $j\in\{1,\ldots,M\}$. 

Hence, recalling~\eqref{alp-0c},
\begin{equation}\label{GANAfaic}
\begin{split}
0&\,=\,\sum_{{|i|+|I|+|\mathfrak{I}|\le K}\atop{|\mathfrak{I}| = |\overline{\mathfrak{I}}|}}
C_{i,I,\mathfrak{I}}\,\theta_{i_1,\dots,i_n,I_1,\dots,I_M,\mathfrak{I}_1,\dots,\mathfrak{I}_l}\;
\XX_1^{i_1}\ldots\XX_n^{i_n}\\&\qquad\qquad\times
e_1^{I_1}\ldots e_M^{I_M}\,
\left(- e_1 \cdot Y_1 \right)_+^{s_1-|I_1|}
\ldots
\left(- e_M \cdot Y_M\right)_+^{s_M-|I_M|}
\XT_1^{\mathfrak{I}_1}\ldots\XT_l^{\mathfrak{I}_l}\\
&\,=\,
\left(- e_1 \cdot Y_1 \right)_+^{s_1}
\ldots
\left(- e_M \cdot Y_M\right)_+^{s_M}\\&\qquad\qquad\times
\sum_{{|i|+|I|+|\mathfrak{I}|\le K}\atop{|\mathfrak{I}| = |\overline{\mathfrak{I}}|}}
C_{i,I,\mathfrak{I}}\,\theta_{i,I,\mathfrak{I}}\;
\XX^i\,e^{I}\,
\left(- e_1 \cdot Y_1 \right)_+^{-|I_1|}
\ldots
\left(- e_M \cdot Y_M\right)_+^{-|I_M|}
\XT^{\mathfrak{I}}
,\end{split}\end{equation}
for a suitable~$C_{i,I,\mathfrak{I}}\ne0$.

We observe that the equality in~\eqref{GANAfaic}
is valid for any choice of the free parameters~$(\XX,Y,\XT)$
in an open subset of~$\mathbb{R}^{p_1+\dots+p_n}\times
\mathbb{R}^{m_1+\ldots+m_M}\times\mathbb{R}^l$,
as prescribed in~\eqref{FREExi}, \eqref{FREEmustar}
and~\eqref{eq:FREEy}.

Now, we take new free parameters $\XY_j$ with $\XY_j\in\mathbb{R}^{m_j}\setminus\{0\}$ for any $j=1,\ldots,M$, and perform in \eqref{GANAfaic} the same change of variables done in \eqref{COMPATI}, obtaining that
$$
0=\sum_{{|i|+|I|+|\mathfrak{I}|\le K}\atop{|\mathfrak{I}| = |\overline{\mathfrak{I}}|}}
C_{i,I,\mathfrak{I}}\,\theta_{i,I,\mathfrak{I}}\;
\XX^i\XY^I\XT^{\mathfrak{I}},
$$
for some $C_{i,I,\mathfrak{I}}\ne 0$. 

Hence, 
the second identity in~\eqref{ipop} is obtained as desired, and the proof of Lemma \ref{lemcin} in case~\ref{itm:case2}
is completed.
\end{proof}

\begin{proof}[Proof of \eqref{ipop}, case \ref{itm:case3}]
We divide the proof of case \ref{itm:case3} into two subcases, namely
either 
\begin{equation}\label{SC-va1}
{\mbox{there exists $h\in\{1,\ldots,l\}$ such that $\XC_h\ne 0$,}}\end{equation}
or
\begin{equation}\label{SC-va2}
{\mbox{$\XC_h=0\,$ for every $h\in\{1,\ldots,l\}$.}}\end{equation}
We start by dealing with the case in~\eqref{SC-va1}.
Up to relabeling and reordering the coefficients $\XC_h$, we can assume that
\begin{equation}\label{CNONU}
{\mbox{$\XC_1\ne 0$}}.\end{equation}
Also, thanks to the assumptions given in case~\ref{itm:case3}, we can suppose that
\begin{equation}\label{NEGB}
{\mbox{$\XB_M<0$}},\end{equation}
and, for any $j\in\{1,\ldots,M\}$, we consider $\lambda_{\star,j}$ and $\tilde{\phi}_{\star,j}$ as in \eqref{lambdastarj}. 
Then, we take~$\omega_j:=1$ and~$\phi_j$ as in~\eqref{autofun1},
so that~\eqref{REGSWYS-A} is satisfied.
In particular, here we have that
\begin{equation}\label{7yHSSIKnNSJS}
\lambda_{j}=
\lambda_{\star,j}\qquad{\mbox{ and }}\qquad\phi_j=\tilde\phi_{\star,j}.\end{equation}
We define
\begin{equation}\label{MAfghjkGGZ3} R:= 
\frac{ 1 }{|\XC_1|}\displaystyle\sum_{j=1}^{M-1}|\XB_j|\lambda_{\star,j}.\end{equation}
We notice that, in light of \eqref{CNONU},
the setting in~\eqref{MAfghjkGGZ3} is well-defined. 

Now, we fix a set of free parameters 
\begin{equation}\label{FREEmustar3}
\XT_{\star,1}\in(R+1,R+2),\dots\,\XT_{\star,l}\in(R+1,R+2).\end{equation}
Moreover, we define
\begin{equation}
\label{alp3}
\lambda_M\,:=\,\frac{1}{\XB_M}\left(
-\sum_{j=1}^{M-1}\XB_j\lambda_{\star,j}-\sum_{h=1}^{l}|\XC_h|\XT_{\star,h}\right).\end{equation}
We notice that \eqref{alp3} is well-defined thanks to \eqref{NEGB}.
{F}rom~\eqref{MAfghjkGGZ3} we deduce that
\begin{equation*}
\begin{split}
\sum_{h=1}^l&|\XC_h|\XT_{\star,h}+\sum_{j=1}^{M-1}\XB_j\lambda_{\star,j}\ge
|\XC_1|\XT_{\star,1}-\sum_{j=1}^{M-1}|\XB_j|\lambda_{\star,j} \\
&>|\XC_1|R-\sum_{j=1}^{M-1}|\XB_j|\lambda_{\star,j}=0.
\end{split}
\end{equation*}
Consequently, by~\eqref{NEGB} and~\eqref{alp3},
\begin{equation}
\label{alp-03}
\lambda_M>0.\end{equation}
Now, we define, for any $h\in\left\{1,\ldots,l\right\}$,
\begin{equation*}
\overline{\XC}_h:=
\begin{cases}
\displaystyle\frac{\XC_h}{\left|\XC_h\right|}\quad& \text{ if }\XC_h\neq 0 ,\\
1 \quad &\text{ if }\XC_h=0.
\end{cases}
\end{equation*}
We notice that
\begin{equation}\label{ahflaanhuf}
{\mbox{$\overline{\XC}_h\neq 0$ for all~$h\in\left\{1,\ldots,l\right\}$,}}\end{equation}
and
\begin{equation}\label{ahflaanhuf3}
{\XC}_h\overline{\XC}_h=|{\XC}_h|.\end{equation}
Moreover, 
we consider~$a_h\in(-2,0)$, for every~$h=1,\dots,l$,
to be chosen appropriately in what follows (see~\eqref{pata7UJ:AKK3}
for a precise choice).

Now, for every $h\in\{1,\ldots,l\}$, we define
\begin{equation}
\label{autofun3}
\psi_h(t_h):=E_{\alpha_h,1}(\overline{\XC}_h\XT_{\star,h}(t_h-a_h)^{\alpha_h}),
\end{equation}
where~$E_{\alpha_h,1}$ denotes the Mittag-Leffler function
with parameters $\alpha:=\alpha_h$ and $\beta:=1$ as defined 
in \eqref{Mittag}. By Lemma~\ref{MittagLEMMA},
we know that
\begin{equation}
\label{HFloj}
\begin{cases}
D^{\alpha_h}_{t_h,a_h}\psi_h(t_h)=\overline{\XC}_h\XT_{\star,h}\psi_h(t_h)\quad\text{in }\,(a_h,+\infty) ,\\
\psi_h(a_h)=1 ,\\
\partial^m_{t_h}\psi_h(a_h)=0\quad\text{for any}\,m=1,\dots,[\alpha_h],
\end{cases}
\end{equation}
and we consider again the extension $\psi^\star_h$ given in \eqref{starest}. 
By Lemma~A.3 in~\cite{CDV18}, we know that~\eqref{HFloj}
translates into
\begin{equation}\label{HFloj2}
D^{\alpha_h}_{t_h,-\infty}\psi_h^\star(t_h)
=\overline{\XC}_h\XT_{\star,h}\psi_h^\star(t_h)\,\text{ in every interval }\,I\Subset(a_h,+\infty).
\end{equation}
Now, we 
consider auxiliary parameters~$\XT_h$, $e_j$ and $Y_j$
as in \eqref{TGAdef}, \eqref{econ} and~\eqref{eq:FREEy}. 
Moreover, we introduce an additional set of free parameters 
\begin{equation}
\label{FREEXXXXX}
\XX=(\XX_1,\ldots,\XX_n)\in
\mathbb{R}^{p_1}\times\ldots\times\mathbb{R}^{p_n}.
\end{equation}
We
let~$\epsilon>0$, to be taken small
possibly depending on the 
free parameters.
We take~$\tau\in C^\infty(\mathbb{R}^{p_1+\ldots+p_n},[0+\infty))$ such that
\begin{equation}
\label{otau3}
\tau(x):=\begin{cases}
\exp\left({\,\XX\cdot x}\right)&\quad\text{if }\,x\in B_1^{p_1+\ldots+p_n} ,\\
0&\quad\text{if }\,x\in \mathbb{R}^{p_1+\ldots+p_n}\setminus B_2^{p_1+\ldots+p_n},
\end{cases}
\end{equation}
where 
$$ \XX\cdot x:=\sum_{j=1}^{n} \XX_i\cdot x_i$$
denotes the standard scalar product.

We notice that, for any $i\in\mathbb{N}^{p_1}\times\ldots\times\mathbb{N}^{p_n}$,
\begin{equation}
\label{resczzz}
\partial^{i}_{x}\tau(0)=\partial^{i_1}_{x_1}\ldots\partial^{i_n}_{x_n}\tau(0)=
\XX^{i_{11}}_{11}\ldots\XX^{i_{1p_1}}_{1p_1}\ldots\XX^{i_{n1}}_{n1}\ldots\XX^{i_{np_n}}_{np_n}
=\XX^i.
\end{equation}
We define
\begin{equation}
\label{svs3}
w\left(x,y,t\right):=\tau(x)\phi_1\left(y_1+e_1+\epsilon Y_1\right)\cdot
\ldots\cdot
\phi_M\left(y_M+e_M+\epsilon Y_M\right) 
\psi^\star_1(t_1)\cdot\ldots\cdot\psi^\star_l(t_l),
\end{equation}
where the setting in~\eqref{REGSWYS-A}
has also been exploited. 

We also notice that $w\in C\left(\mathbb{R}^N\right)\cap C_0\left(\mathbb{R}^{N-l}\right)\cap\mathcal{A}$. Moreover,
if
\begin{equation}\label{pata7UJ:AKK3}
a=(a_1,\dots,a_l):=\left(-\frac{\epsilon}{\XT_1},\dots,-\frac{\epsilon}{\,\XT_l}\right)\in\mathbb{R}^l\end{equation}
and~$\left(x,y\right)$ is sufficiently close to the origin
and~$t\in(a_1,+\infty)\times\dots\times(a_l,+\infty)$, we have that
\begin{eqnarray*}&&
\Lambda_{-\infty} w\left(x,y,t\right)\\&=&
\left( \sum_{j=1}^{M} \XB_j (-\Delta)^{s_j}_{y_j}+
\sum_{h=1}^{l} \XC_h D^{\alpha_h}_{t_h,-\infty}\right) w\left(x,y,t\right)\\
&=&\sum_{j=1}^{M} \XB_j 
\tau(x)\phi_1\left(y_1+e_1+\epsilon Y_1\right)
\ldots\phi_{j-1}\left(y_{j-1}+e_{j-1}+\epsilon Y_{j-1}\right)(-\Delta)^{s_j}_{y_j}\phi_j\left(y_j+e_j+\epsilon Y_j\right)\\&&\qquad\times
\phi_{j+1}\left(y_{j+1}+e_{j+1}+\epsilon Y_{j+1}\right)
\ldots\phi_M\left(y_M+e_M+\epsilon Y_M\right)\psi^\star_1\left(t_1\right)
\ldots\psi^\star_l\left(t_l\right)\\
&&+\sum_{h=1}^{l} \XC_h\tau(x)\phi_1\left(y_1+e_1+\epsilon Y_1\right)\ldots\phi_M\left(y_M+e_M+\epsilon Y_M\right)\psi^\star_1\left(t_1\right)
\ldots\psi^\star_{h-1}\left(t_{h-1}\right)\\
&&\qquad\times D^{\alpha_h}_{t_h,-\infty}\psi^\star_h\left(t_h\right)
\psi^\star_{h+1}(t_{h+1})\ldots\psi^\star_l\left(t_l\right)\\
&=&\sum_{j=1}^{M} \XB_j \lambda_j
\tau(x)\phi_1\left(y_1+e_1+\epsilon Y_1\right)
\ldots\phi_M\left(y_M+e_M+\epsilon Y_M\right)\psi^\star_1(t_1)\ldots\psi^\star_l(t_l)
\\ 
&&+\sum_{h=1}^{l} \XC_h\overline{\XC}_h\XT_{\star,h}
\tau(x)\phi_1\left(y_1+e_1+\epsilon Y_1\right)
\ldots\phi_M\left(y_M+e_M+\epsilon Y_M\right)\psi^\star_1(t_1)\ldots\psi^\star_l(t_l) \\
&=&\left( \sum_{j=1}^M\XB_j\lambda_j+\sum_{h=1}^{l} \XC_h\overline{\XC}_h\XT_{\star,h}\right) w(x,y,t),
\end{eqnarray*}
thanks to \eqref{REGSWYS-A} and~\eqref{HFloj2}.

Consequently, making use of \eqref{7yHSSIKnNSJS}, \eqref{alp3} and~\eqref{ahflaanhuf3},
if~$(x,y)$ is near the origin and~$t\in(a_1,+\infty)\times\dots\times(a_l,+\infty)$,
we have that
$$
\Lambda_{-\infty} w\left(x,y,t\right)=
\left( \sum_{j=1}^M\XB_j\lambda_{\star,j}+\XB_M\lambda_M+\sum_{h=1}^{l} |\XC_h|\XT_{\star,h}\right) w(x,y,t)=0
.$$
This says that~$w\in\mathcal{H}$. Thus, in light of~\eqref{aza} we have that
\begin{equation*}
0=\theta\cdot\partial^K w\left(0\right)=
\sum_{\left|\iota\right|\leq K}
{\theta_{\iota}\partial^\iota w\left(0\right)}=\sum_{|i|+|I|+|\mathfrak{I}|\le K}
\theta_{i,I,\mathfrak{I}}\,\partial_{x}^i\partial_{y}^I\partial_{t}^{\mathfrak{I}}w\left(0\right) 
.\end{equation*}
{F}rom this and~\eqref{svs3}, we obtain that
\begin{equation}\label{eq:orthsi3}
0=\sum_{|i|+|I|+|\mathfrak{I}|\le K}
\theta_{i,I,\mathfrak{I}}\,
\partial^{i}_x\tau(0)
\partial^{I_1}_{y_{1}} {\phi}_{1}(e_1+\epsilon Y_1)\ldots
\partial^{I_M}_{y_{M}} {\phi}_{M}(e_M+\epsilon Y_M)
\,\partial_{t_1}^{\mathfrak{I}_1}\psi_1(0)\ldots\partial_{t_l}^{\mathfrak{I}_l}\psi_l(0).\end{equation}
Moreover, using~\eqref{autofun3} and \eqref{pata7UJ:AKK3},
it follows that, for every $\mathfrak{I}_h\in\mathbb{N}$
\begin{eqnarray*}
\partial^{\mathfrak{I}_h}_{t_h}\psi_h(0)&=&
\sum_{j=0}^{+\infty} {\frac{
\overline{\XC}_h^j\,\XT_{\star,h}^j\, \alpha_h j(\alpha_h j-1)\dots(\alpha_h j-
\mathfrak{I}_l+1)
(0-a_h)^{\alpha_h j-\mathfrak{I}_h}}{\Gamma\left(\alpha_h j+1\right)}}\\
&=& \sum_{j=0}^{+\infty} {\frac{
\overline{\XC}_h^j\,\XT_{\star,h}^j\, \alpha_h j(\alpha_h j-1)\dots(\alpha_h j-
\mathfrak{I}_h+1)
\,{\epsilon}^{\alpha_h j-\mathfrak{I}_h}}{\Gamma\left(\alpha_h j+1\right)\;
{\XT_h}^{\alpha_h j-\mathfrak{I}_h}}}\\
&=& \sum_{j=1}^{+\infty} {\frac{
\overline{\XC}_h^j\,\XT_{\star,h}^j\, \alpha_h j(\alpha_h j-1)\dots(\alpha_h j-
\mathfrak{I}_h+1)
\,{\epsilon}^{\alpha_h j-\mathfrak{I}_h}}{\Gamma\left(\alpha_h j+1\right)\;
{\XT_h}^{\alpha_h j-\mathfrak{I}_h}}}.
\end{eqnarray*}
Accordingly, recalling~\eqref{TGAdef}, we find that
\begin{equation}
\label{limmittl3}\begin{split}&
\lim_{\epsilon\searrow 0}\epsilon^{\mathfrak{I}_h-\alpha_h}
\partial^{\mathfrak{I}_h}_{t_h}\psi_h(0)
=\lim_{\epsilon\searrow 0}
\sum_{j=1}^{+\infty} {\frac{
\overline{\XC}_h^j\,\XT_{\star,h}^j\, \alpha_h j(\alpha_h j-1)\dots(\alpha_h j-
\mathfrak{I}_h+1)
\,{\epsilon}^{\alpha_h (j-1)}}{\Gamma\left(\alpha_h j+1\right)\;
{\XT_h}^{\alpha_h j-\mathfrak{I}_h}}}
\\&\qquad=
{\frac{ \overline{\XC}_h\,\XT_{\star,h}\,\alpha_h (\alpha_h -1)\dots(\alpha_h -\mathfrak{I}_h+1)
}{\Gamma\left(\alpha_h +1\right)\;
\XT_h^{\alpha_h -\mathfrak{I}_h}}} =
{\frac{ \overline{\XC}_h\,\XT_h^{\mathfrak{I}_h}\,\alpha_h (\alpha_h -1)\dots(\alpha_h -\mathfrak{I}_h+1)
}{\Gamma\left(\alpha_h +1\right)}}.
\end{split}\end{equation}
Also, recalling~\eqref{IBARRA}, we can write~\eqref{eq:orthsi3} as
\begin{equation}
\label{eq:ort:X3}
0=\sum_{{|i|+|I|+|\mathfrak{I}|\le K}\atop{|I|\le|\overline{I}|}}
\theta_{i,I,\mathfrak{I}}\,\partial^{i}_x\tau(0)
\partial^{I_1}_{y_{1}} {\phi}_{1}(e_1+\epsilon Y_1)\ldots
\partial^{I_M}_{y_{M}} {\phi}_{M}(e_M+\epsilon Y_M) 
\partial^{\mathfrak{I}_1}_{t_1}\psi_{1}(0)\ldots
\partial^{\mathfrak{I}_l}_{t_l}\psi_{l}(0).
\end{equation}
Moreover, we define
\begin{equation*}
\Xi:=\left|\overline{I}\right|-\sum_{j=1}^M {s_j}+|\mathfrak{I}|-\sum_{h=1}^l {\alpha_h}
.\end{equation*}
Then, we
multiply~\eqref{eq:ort:X3} by $\epsilon^{\Xi}\in(0,+\infty)$, and we
send~$\epsilon$ to zero. In this way, we obtain from~\eqref{alppoo},
\eqref{resczzz}, \eqref{limmittl3} and~\eqref{eq:ort:X3} that
\begin{eqnarray*}
0&=&\lim_{\epsilon\searrow0}
\epsilon^{\Xi}
\sum_{{|i|+|I|+|\mathfrak{I}|\le K}\atop{|I|\le|\overline{I}|}}
\theta_{i,I,\mathfrak{I}}\,\partial_{x}^i\tau(0)
\partial_{y_1}^{I_1}\phi_1\left(e_1+\epsilon Y_1\right)
\ldots\partial_{y_M}^{I_M}\phi_M\left(e_M+\epsilon Y_M\right)
\partial^{\mathfrak{I}_1}_{t_1}\psi_1(0)\ldots\partial^{\mathfrak{I}_l}_{t_l}\psi_l(0)
\\ &=&\lim_{\epsilon\searrow0}
\sum_{{|i|+|I|+|\mathfrak{I}|\le K}\atop{|I|\le|\overline{I}|}}
\epsilon^{|\overline{I}|-|I|}
\theta_{i,I,\mathfrak{I}}\,\partial_x^i\tau(0)
\epsilon^{|I_1|-s_1}\partial_{y_1}^{I_1}\phi_1\left(e_1+\epsilon Y_1\right)
\ldots\epsilon^{|I_M|-s_M}\partial_{y_M}^{I_M}\phi_M\left(e_M+\epsilon Y_M\right)
\\ &&\qquad\times\epsilon^{\mathfrak{I}_1-\alpha_1}\partial^{\mathfrak{I}_1}_{t_1}\psi_1(0)\ldots\epsilon^{\mathfrak{I}_l-\alpha_l}\partial^{\mathfrak{I}_l}_{t_l}\psi_l(0)
\\&=&
\sum_{{|i|+|I|+|\mathfrak{I}|\le K}\atop{|I| = |\overline{I}|}}
C_{i,I,\mathfrak{I}}\,\theta_{i,I,\mathfrak{I}}\,
\XX_1^{i_1}\ldots\XX_n^{i_n}\,
e_1^{I_1}\ldots e_M^{I_M}\,
\left(- e_1 \cdot Y_1 \right)_+^{s_1-|I_1|}
\ldots
\left(- e_M \cdot Y_M\right)_+^{s_M-|I_M|}\XT_1^{\mathfrak{I}_1}\ldots\XT_l^{\mathfrak{I}_l}
\\&=&
\left(- e_1 \cdot Y_1 \right)_+^{s_1}
\ldots
\left(- e_M \cdot Y_M\right)_+^{s_M}\\&&\qquad\qquad\times
\sum_{{|i|+|I|+|\mathfrak{I}|\le K}\atop{|I| = |\overline{I}|}}
C_{i,I,\mathfrak{I}}\,\theta_{i,I,\mathfrak{I}}\;
\XX^i\,e^{I}\,
\left(- e_1 \cdot Y_1 \right)_+^{-|I_1|}
\ldots
\left(- e_M \cdot Y_M\right)_+^{-|I_M|}
\XT^{\mathfrak{I}}
,\end{eqnarray*}
for a suitable~$C_{i,I,\mathfrak{I}}\ne0$.

We observe that the latter equality
is valid for any choice of the free parameters~$(\XX,Y,\XT)$
in an open subset of~$\mathbb{R}^{p_1+\ldots+p_n}\times\mathbb{R}^{m_1+\ldots+m_M}\times\mathbb{R}^l$,
as prescribed in~\eqref{eq:FREEy},~\eqref{FREEmustar3} and~\eqref{FREEXXXXX}. 

Now, we take new free parameters $\XY_j$ with $\XY_j\in\mathbb{R}^{m_j}\setminus\{0\}$ for any $j=1,\ldots,M$, and perform in the latter identity the same change of variables done in \eqref{COMPATI}, obtaining that
$$
0=\sum_{{|i|+|I|+|\mathfrak{I}|\le K}\atop{|I| = |\overline{I}|}}
C_{i,I,\mathfrak{I}}\,\theta_{i,I,\mathfrak{I}}\;
\XX^i\XY^I\XT^{\mathfrak{I}},
$$
for some $C_{i,I,\mathfrak{I}}\ne 0$.
This completes the proof of~\eqref{ipop} in case~\eqref{SC-va1} is satisfied.\medskip

Hence, we now focus on the case in which~\eqref{SC-va2} holds true.
For any $j\in\{1,\ldots,M\}$,
we consider the function~$\psi\in H^{s_j}(\mathbb{R}^{m_j})\cap C^{s_j}_0(\mathbb{R}^{m_j})$
constructed in Lemma \ref{hbump} and we call such function~$\phi_j$,
to make it explicit its dependence on~$j$ in this case.
We recall that
\begin{equation}
\label{harmony}
(-\Delta)^{s_j}_{y_j}\phi_j(y_j)=0\quad\text{ in }\,B_1^{m_j}.
\end{equation}
Also, for every $j\in\{1,\ldots,M\}$, we let $e_j$
and~$Y_j$ be as in \eqref{econ} and \eqref{eq:FREEy}.
Thanks to Lemma~\ref{hbump} and Remark~\ref{RUCAPSJD},
for any $I_j\in\mathbb{N}^{m_j}$, we know that
\begin{equation}
\label{psveind}
\lim_{\epsilon\searrow 0}\epsilon^{|I_j|-s_j}\partial_{y_j}^{I_j}\phi_j(e_j+\epsilon Y_j)=\kappa_{s_j}e_j^{I_j}(-e_j\cdot Y_j)_+^{s_j-|I_j|},
\end{equation}
for some $\kappa_{s_j}\ne 0$.

Moreover, for any $h=1,\ldots, l$, we define $\overline{\tau}_h(t_h)$ as
\begin{equation}
\label{tita}
\overline{\tau}_h(t_h):=\begin{cases}
e^{\XT_h t_h}\quad&\text{if}\quad t_h\in[-1,+\infty), \\
\displaystyle e^{-\XT_h}\sum_{i=0}^{k_h-1}\frac{\XT_h^i}{i!}(t_h+1)^i\quad&\text{if}\quad t_h\in(-\infty,-1),\end{cases}
\end{equation}
where
$\XT=(\XT_1,\ldots,\XT_l)\in(1,2)^l$ are free parameters.

We notice that, for any $h\in\{1,\ldots, l\}$ and $\mathfrak{I}_h\in\mathbb{N}$,
\begin{equation}
\label{rescttt}
\partial^{\mathfrak{I}_h}_{t_h}\overline{\tau}_h(0)=\XT_h^{\mathfrak{I}_h}.
\end{equation}
%%	
%%	Moreover, 
%%	% setting $f_h(t_h):=e^{\XT_h t_h}$, 
%%	%	a simple computation shows that
%%	%	%\begin{equation}
%%	%	%\label{tito}
%%	%	%D^{\alpha_h}_{t_h,-1}f_h(t_h)\leq \frac{2^{k_h}}{\Gamma(k_h-\alpha_h+1)}\left(t_h^{k_h-\alpha_h}(e^{2t_h}-1)+(t_h+1)^{k_h-\alpha_h}\right),
%%	%	%\end{equation}
%%	%	%$|D^{\alpha_h}_{t_h,-1}f_h(t_h)|$ is finite for every $t_h>-1$.
%%	%	%Hence, from \eqref{tita}, it follows that 
%%	%	$f_h\in C^{k_h,\alpha_h}_{-1}$,
%%	using Lemma A.3 in \cite{CDV18} with $u(t_h):=e^{\XT_h t_h}$, $a:=-\infty$, $b:=-1$ and $u_\star:=\overline{\tau}_h$, we have that 
%%	\begin{equation}
%%	\label{railey}
%%	\overline{\tau}_h\in C^{k_h,\alpha_h}_{-\infty}\cap C^{\infty}((-1,+\infty)),\,\text{ and}\quad\partial^{k_h}_{t_h}\overline{\tau}_h(t_h)=0\quad\text{in}\quad(-\infty,-1).
%%	\end{equation}
%%	
Now, we
define
\begin{equation}
\label{quiquoqua}
w(x,y,t):=\tau(x)\phi_1(y_1+e_1+\epsilon Y_1)\ldots\phi_M(y_M+e_M+\epsilon Y_M)\overline{\tau}_1(t_1)\ldots\overline{\tau}_l(t_l),
\end{equation}
where the setting of~\eqref{autofun1}, \eqref{otau3} and \eqref{tita} has been exploited.
We have that~$w\in\mathcal{A}$. Moreover, we point out that, since $\tau$, $\phi_1,\ldots,\phi_M$ are 
compactly supported, we have that~$w\in C(\mathbb{R}^N)\cap C_0(\mathbb{R}^{N-l})$, and, using Proposition \ref{maxhlapspan}, for any $j\in\{1,\ldots,M\}$, it holds that~$\phi_j\in C^{\infty}(\mathcal{N}_j)$ for some neighborhood~$\mathcal{N}_j$
of the origin in $\mathbb{R}^{m_j}$.
Hence $w\in C^{\infty}(\mathcal{N})$. 

Furthermore,
using \eqref{harmony}, when~$y$
is in a neighborhood of the origin we have that
\begin{equation*}
\begin{split}
\Lambda_{-\infty} w(x,y,t)&=\tau(x)\left(\XB_1(-\Delta)^{s_1}_{y_1}\phi_1(y_1+e_1+\epsilon Y_1)\right)\ldots\phi_M(y_M+e_M+\epsilon Y_M)\overline{\tau}_1(t_1)\ldots\overline{\tau}_l(t_l) \\
&+\ldots+\tau(x)\phi_1(y_1)\ldots\left(\XB_M(-\Delta)^{s_M}_{Y_M}\phi_M(y_M+
e_M+\epsilon Y_M)\right)\overline{\tau}_1(t_1)\ldots\overline{\tau}_l(t_l)=0,
\end{split}
\end{equation*}
which gives that~$w\in\mathcal{H}$.

In addition, using~\eqref{IBARRA}, \eqref{resczzz} and \eqref{rescttt}, we have that
\begin{eqnarray*}&&
0=\theta\cdot\partial^K w(0)=\sum_{|\iota|\leq K}\theta_{i,I,\mathfrak{I}}
\partial_x^i \partial_y^I \partial_t^{\mathfrak{I}} w(0)
=\sum_{{|\iota|\leq K}\atop{|I| \le |\overline{I}|}}\theta_{i,I,\mathfrak{I}}
\partial_x^i \partial_y^I \partial_t^{\mathfrak{I}} w(0)\\
&&\qquad\qquad=\sum_{{|\iota|\leq K}\atop{|I| \le |\overline{I}|}}\theta_{i,I,\mathfrak{I}}
\,\XX^i\partial_{y_1}^{I_1} \phi_1(e_1+\epsilon Y_1)\ldots
\partial_{y_M}^{I_M}\phi_M(e_M+\epsilon Y_M)\,\XT^\mathfrak{I}.
\end{eqnarray*}
Hence, we set
$$ \Xi:=|\overline{I}|-\sum_{j=1}^M s_j,$$
we multiply the latter identity by~$\epsilon^\Xi$
and we exploit~\eqref{psveind}. In this way, we find that
\begin{eqnarray*}
0 &=& \lim_{\epsilon\searrow0}
\sum_{{|\iota|\leq K}\atop{|I| \le |\overline{I}|}}\epsilon^{|\overline I|-|I|}\theta_{i,I,\mathfrak{I}}
\,\XX^i\,\epsilon^{|I_1|-s_1}\partial_{y_1}^{I_1} \phi_1(e_1+\epsilon Y_1)\ldots\epsilon^{|I_M|-s_M}
\partial_{y_M}^{I_M}\phi_M(e_M+\epsilon Y_M)\,\XT^\mathfrak{I}\\&=&
\sum_{{|\iota|\leq K}\atop{|I| = |\overline{I}|}}\theta_{i,I,\mathfrak{I}}\,\kappa_{s_j}\,
\,\XX^i\,e^{I}\,(-e_1\cdot Y_1)_+^{s_1-|I_1|}\ldots(-e_M\cdot Y_M)_+^{s_M-|I_M|}
\,\XT^\mathfrak{I}\\
&=&
(-e_1\cdot Y_1)_+^{s_1}\ldots(-e_M\cdot Y_M)_+^{s_M}
\sum_{{|\iota|\leq K}\atop{|I| = |\overline{I}|}}\theta_{i,I,\mathfrak{I}}\,\kappa_{s_j}\,
\,\XX^i\,e^{I}\,(-e_1\cdot Y_1)_+^{-|I_1|}\ldots(-e_M\cdot Y_M)_+^{-|I_M|}
\,\XT^\mathfrak{I},
\end{eqnarray*}
and consequently
\begin{equation}\label{7UJHAnAXbansdo}
0=\sum_{{|\iota|\leq K}\atop{|I| = |\overline{I}|}}\theta_{i,I,\mathfrak{I}}\,\kappa_{s_j}\,
\,\XX^i\,e^{I}\,(-e_1\cdot Y_1)_+^{-|I_1|}\ldots(-e_M\cdot Y_M)_+^{-|I_M|}
\,\XT^\mathfrak{I}.
\end{equation}
Now we take
free parameters $\XY\in\mathbb{R}^{m_1+\ldots+m_M}\setminus\{0\}$
and we perform the same change of variables in \eqref{COMPATI}.
In this way, we deduce from~\eqref{7UJHAnAXbansdo} that
\begin{eqnarray*}
0=\sum_{{|i|+|I|+|\mathfrak{I}|\le K}\atop{|I| = |\overline{I}|}}C_{i,I,\mathfrak{I}}\theta_{i,I,\mathfrak{I}}\XX^i\XY^I\XT^\mathfrak{I},
\end{eqnarray*}
for some $C_{i,I,\mathfrak{I}}\ne 0$, and the first claim in \eqref{ipop} is proved
in this case as well.
\end{proof}

\begin{proof}[Proof of \eqref{ipop}, case \ref{itm:case4}]
Notice that if there exists $j\in\{1,\ldots,M\}$ such that $\XB_j\ne 0$, we are in the setting of case \ref{itm:case3}.
Therefore, we assume that $\XB_j=0$ for every $j\in\{1,\ldots,M\}$.

We let $\psi$  
be the function constructed in Lemma~\ref{LF}.
For each $h\in\{1,\dots,l\}$,
we let~$\overline\psi_h(t_h):=\psi(t_h)$,
to make the dependence on $h$ clear and explicit.
Then, by formulas~\eqref{LAp1} and~\eqref{LAp2},
we know that
\begin{equation}
\label{cidivu}
D^{\alpha_h}_{t_h,0}\overline{\psi}_h(t_h)=0\quad\text{ in }\,(1,+\infty)
\end{equation} 
and, for every~$\ell\in\mathbb{N}$,
\begin{equation}\label{FOR2.50}
\lim_{\epsilon\searrow0} \epsilon^{\ell-\alpha_h}\partial^\ell_{t_h}
\overline{\psi}_h(1+\epsilon t_h)=\kappa_{h,\ell}\; t_h^{\alpha_h-\ell}
,\end{equation}
in the sense of distribution,
for some~$\kappa_{h,\ell}\ne0$.

Now,
we introduce a set of auxiliary parameters $\XT=(\XT_1,\ldots,\XT_l)
\in(1,2)^l$, and fix $\epsilon$ sufficiently small
possibly depending on the parameters. Then, we define
\begin{equation}
\label{keanzzz}
a=(a_1,\dots,a_l):=\left(-\frac{\epsilon}{\XT_1}-1,\ldots,-\frac{\epsilon}{\XT_l}-1\right)
\in(-2,0)^l,\end{equation}
and
\begin{equation}
\label{translation}
\psi_h(t_h):=\overline{\psi}_h(t_h-a_h).
\end{equation}
With a simple computation we have that the function
in~\eqref{translation} satisfies
\begin{equation}\label{cgraz}
D^{\alpha_h}_{t_h,a_h}\psi_h(t_h)=D^{\alpha_h}_{t_h,0}
\overline{\psi}_h(t_h-a_h)=0 \quad \text{in }\,(1+a_h,+\infty)=\left(-\frac{\epsilon}{\XT_h},+\infty\right),
\end{equation}
thanks to~\eqref{cidivu}.
In addition, for every~$\ell\in\mathbb{N}$, we have that~$\partial^\ell_{t_h}\psi_h(t_h)=
\partial^\ell_{t_h}\overline{\psi}_h(t_h-a_h)$, and therefore,
in light of~\eqref{FOR2.50} and~\eqref{keanzzz},
\begin{equation}\label{FOR2.5}
\epsilon^{\ell-\alpha_h}\partial^\ell_{t_h}\psi_h(0)=\epsilon^{\ell-\alpha_h}
\partial^\ell_{t_h}\overline{\psi}_h(-a_h)=
\epsilon^{\ell-\alpha_h}\partial^\ell_{t_h}\overline{\psi}_h\left(
1+\frac{\epsilon}{\XT_h}
\right)\to\kappa_{h,\ell}\;\XT_h^{\ell-\alpha_h},
\end{equation}
in the sense of distributions, as~$\epsilon\searrow0$.

Moreover, since for any $h=1,\ldots, l$, $\psi_h\in C^{k_h,\alpha_h}_{a_h}$, we can consider the extension
\begin{equation}\label{estar}
\psi^\star_h(t_h):=\begin{cases}
\psi_h(t_h)&\quad\text{ if }\,t_h\in[a_h,+\infty), \\
\displaystyle\sum_{i=0}^{k_h-1}\frac{\psi_h^{(i)}(a_h)}{i!}(t_h-a_h)^i&\quad\text{ if }\,t_h\in(-\infty,a_h),
\end{cases}
\end{equation}
and, using Lemma A.3 in \cite{CDV18} with $u:=\psi_h$, $a:=-\infty$, $b:=a_h$ and $u_\star:=\psi^\star_h$, we have that 
\begin{equation}
\label{tuck}
\psi^\star_h\in C^{k_h,\alpha_h}_{-\infty}\quad\text{and}\quad D^{\alpha_h}_{t_h,-\infty}\psi_h^\star=D^{\alpha_h}_{t_h,a_h}\psi_h=0\quad\text{in every interval }\,I\Subset\left(-\frac{\epsilon}{\XT_h},+\infty\right).
\end{equation}
Now, we fix a set of free parameters $\XY=\left(\XY_1,\ldots,\XY_M\right)\in\mathbb{R}^{m_1+\ldots+m_M}$, and consider~$\overline\tau\in C^\infty
(\mathbb{R}^{m_1+\ldots+m_M})$, such that
\begin{equation}
\label{otau5}
\overline{\tau}(y):=\begin{cases}
\exp\left({\XY\cdot y}\right)\quad&\text{if }\quad y\in B_1^{m_1+\ldots+m_M} ,\\
0\quad&\text{if }\quad y\in \mathbb{R}^{m_1+\ldots+m_M}\setminus B_2^{m_1+\ldots+m_M},\end{cases}
\end{equation}
where
$$
\XY\cdot y=\sum_{j=1}^M\XY_j\cdot y_j,
$$
denotes the standard scalar product.

We notice that, for any multi-index $I\in\mathbb{N}^{m_1+\ldots m_M}$,
\begin{equation}
\label{rescyyy}
\partial^{I}_{y}\overline{\tau}(0)=\XY^I,
\end{equation}
where the multi-index notation has been used.

Now, we define
\begin{equation}
\label{sinaloa}
w(x,y,t):=\tau(x)\overline{\tau}(y)\psi^\star_1(t_1)\ldots\psi^\star_l(t_l),
\end{equation}
where the setting in \eqref{otau3},
\eqref{estar} and \eqref{otau5} has been exploited.

Using~\eqref{tuck}, we have that, for any $(x,y)$
in a neighborhood of the origin and~$t\in\left(-\frac{\epsilon}{2},+\infty\right)^l$,
\begin{equation*}
\begin{split}
\Lambda_{-\infty} w(x,y,t)&=\tau(x)\overline{\tau}(y)\left(\XC_1 D^{\alpha_1}_{t_1,-\infty}\psi^\star_1(t_1)\right)
\ldots\psi^\star_l(t_l) \\
&+\ldots+\tau(x)\overline{\tau}(y)\psi^\star_1(t_1)\ldots\left(\XC_l D^{\alpha_l}_{t_l,-\infty}\psi^\star_l(t_l)\right)=0.
\end{split}
\end{equation*}
We have that~$w\in\mathcal{A}$, and, since $\tau$ and $\overline{\tau}$ are compactly supported, we
also have that $w\in C(\mathbb{R}^N)\cap C_0(\mathbb{R}^{N-l})$.
Also, from Lemma~\ref{LF}, for any $h\in\{1,\ldots,l\}$, we know that~$\overline{\psi}_h\in C^{\infty}((1,+\infty))$, hence $\psi_h \in C^{\infty}\left(\left(
-\frac{\epsilon}{\XT_h},+\infty\right)\right)$.
Thus, $w\in C^{\infty}(\mathcal{N})$, and consequently~$w\in\mathcal{H}$.

Recalling~\eqref{IBARRA2}, \eqref{resczzz}, and \eqref{rescyyy}, we have that
\begin{equation}\label{LATTER}\begin{split}&
0=\theta\cdot\partial^K w(0)=\sum_{|\iota|\leq K}\theta_{i,I,\mathfrak{I}}
\partial_x^i \partial_y^I \partial_t^{\mathfrak{I}} w(0)
=\sum_{{|\iota|\leq K}\atop{|\mathfrak{I}| \le |\overline{\mathfrak{I}}|}}\theta_{i,I,\mathfrak{I}}
\partial_x^i \partial_y^I \partial_t^{\mathfrak{I}} w(0)\\
&\qquad\qquad=\sum_{{|\iota|\leq K}\atop{|\mathfrak{I}| \le |\overline{\mathfrak{I}}|}}\theta_{i,I,\mathfrak{I}}
\,\XX^i\XY^I\partial_{t_1}^{\mathfrak{I}_1}\psi_1(0)\ldots\partial_{t_l}^{\mathfrak{I}_l}\psi_l(0).
\end{split}\end{equation}
Hence, we set
$$ \Xi:=|\overline{\mathfrak{I}}|-\sum_{h=1}^l \alpha_h
%%%%=\left(|\overline{\mathfrak{I}}|-|\mathfrak{I}|\right)+\left(|\mathfrak{I}|-\sum_{h=1}^l \alpha_h\right)
,$$
we multiply the identity in~\eqref{LATTER} by~$\epsilon^\Xi$
and we exploit~\eqref{FOR2.5}. In this way, we find that
\begin{eqnarray*}
0 &=& \lim_{\epsilon\searrow0}
\sum_{{|\iota|\leq K}\atop{|\mathfrak{I}| \le |\overline{\mathfrak{I}}|}}\epsilon^{|\overline{\mathfrak{I}}|-|\mathfrak{I}|}\theta_{i,I,\mathfrak{I}}
\,\XX^i\,\XY^I\,\epsilon^{\mathfrak{I}_1-\alpha_1}\partial_{t_1}^{\mathfrak{I}_1} {\psi}_1(0)\ldots\epsilon^{\mathfrak{I}_l-\alpha_l}\partial_{t_l}^{\mathfrak{I}_l} {\psi}_l(0)\\&=&
\sum_{{|\iota|\leq K}\atop{|\mathfrak{I}| = |\overline{\mathfrak{I}}|}}\theta_{i,I,\mathfrak{I}}\;
\kappa_{1,\mathfrak{I}_1}\dots
\kappa_{l,\mathfrak{I}_l}\;
\XX^i\,\XY^I\,\XT_1^{\mathfrak{I}_1-\alpha_1}\ldots\XT_l^{\mathfrak{I}_l-\alpha_l}\\
&=&
\XT_1^{-\alpha_1}\ldots\XT_l^{-\alpha_l}
\sum_{{|\iota|\leq K}\atop{|\mathfrak{I}| = |\overline{\mathfrak{I}}|}}\theta_{i,I,\mathfrak{I}}
\;\kappa_{1,\mathfrak{I}_1}\dots
\kappa_{l,\mathfrak{I}_l}\;
\XX^i\,\XY^I\,\XT_1^{\mathfrak{I}_1}\ldots\XT_l^{\mathfrak{I}_l},
\end{eqnarray*}
and consequently
\begin{equation*}
0=\sum_{{|\iota|\leq K}\atop{|\mathfrak{I}| = |\overline{\mathfrak{I}}|}}
\theta_{i,I,\mathfrak{I}}\;
\kappa_{1,\mathfrak{I}_1}\dots
\kappa_{l,\mathfrak{I}_l}\;
\,\XX^i\,\XY^I
\,\XT^\mathfrak{I},
\end{equation*}
and the second claim in \eqref{ipop} is proved
in this case as well.
\end{proof}

\section{Every function is locally $\Lambda_{-\infty}$-harmonic up to a small error,
and completion of the proof of Theorem \ref{theone2}}
\label{s:fourth}

In this section we complete the proof of Theorem \ref{theone2}
(which in turn implies Theorem \ref{theone} via Lemma~\ref{GRAT}).
By standard approximation arguments we can reduce to
the case in which $f$ is a polynomial, and hence, by the linearity
of the operator~$\Lambda_{-\infty}$, to the case in which is a monomial.
The details of the proof are therefore the following:

%	we recall a version of the Stone-Weierstrass Theorem in the $C^\ell$ setting
%	(in our context, this will be useful to reduce approximation
%	problems to the case of polynomials, and hence, by linearity, of monomials).
%	For a short proof of this result, see \cite{MR3626547}.
%	
%	\begin{lemma}
%	\label{stowei}
%	Fix $\ell\in\mathbb{N}$. For any $f\in C^\ell(\overline{B}_1)$
%	and any $\epsilon>0$ there exists a polynomial $P$
%	such that $\|f-P\|_{C^\ell(\overline{B}_1)}\leq\epsilon$.
%	\end{lemma}

\subsection{Proof of Theorem \ref{theone2} when $f$ is a monomial}\label{7UHASGBSBSBBSB}

We prove Theorem~\ref{theone2} under the initial assumption
that $f$ is a monomial, that is
\begin{equation}\label{iorade}
f\left(x,y,t\right)=\frac{x_1^{i_1}\ldots x_n^{i_n}y_1^{I_1}\ldots y_M^{I_M}t_1^{\mathfrak{I}_1}\ldots t_l^{\mathfrak{I}_l}}{\iota!}=\frac{x^iy^It^{\mathfrak{I}}}{\iota!}=
\frac{(x ,y,t)^{\iota}}{\iota!},
\end{equation}
where $\iota!:=i_1!\ldots i_n!I_1!\ldots I_M!\mathfrak{I}_1!\ldots\mathfrak{I}_l!$ and $I_\beta!:=I_{\beta,1}!\ldots I_{\beta,m_\beta}!$, $i_\chi!:=i_{\chi,1}!\ldots i_{\chi,p_\chi}!$ for all $\beta=1,\ldots M$. and $\chi=1,\ldots,n$. 
To this end, we argue as follows. 
We 
consider $\eta\in\left(0,1\right)$, to be taken
sufficiently small with respect to the
parameter~$\epsilon>0$ which has been fixed
in the statement of Theorem~\ref{theone2}, and we define
%\begin{equation}\label{le-a}
%a_1:=-2-\frac3{\eta^{\frac1{\alpha_1}}},\dots,
%a_l:=-2-\frac3{\eta^{\frac1{\alpha_l}}}\qquad a:=(a_1,\dots,a_l)
%\end{equation}
%and
$$ {\mathcal{T}}_\eta(x,y,t):=\left(
\eta^{\frac{1}{r_1}}x_1,\ldots,\eta^{\frac{1}{r_n}}x_n,\eta^{\frac{1}{2s_1}}y_1,\ldots,\eta^{\frac{1}{2s_M}}y_M,\eta^{\frac{1}{\alpha_1}}t_1,\ldots,\eta^{\frac{1}{\alpha_l}}t_l\right).$$
We also define
\begin{equation}\label{iorade2}
\gamma:=\sum_{j=1}^n {\frac{|i_j|}{r_j}}+\sum_{j=1}^M {\frac{\left|I_j\right|}{2s_j}}+\sum_{j=1}^l {\frac{\mathfrak{I}_j}{\alpha_j}}
,\end{equation}
and
\begin{equation}
\label{am}
\delta:=\min\left\{\frac{1}{r_1},\ldots,\frac{1}{r_n},\frac{1}{2s_1},\ldots,\frac{1}{2s_M},\frac{1}{\alpha_1},\ldots,\frac{1}{\alpha_l}\right\}.
\end{equation}
We also take $K_0\in\mathbb{N}$ such that
\begin{equation}
\label{blim}
K_0\geq\frac{\gamma+1}{\delta}
\end{equation}
and we let
\begin{equation}
\label{blo}
K:=K_0+\left|i\right|+\left|I\right|+\left|\mathfrak{I}\right|+\ell=
K_0+\left|\iota\right|+\ell,
\end{equation}
where $\ell$ is the fixed integer given in the statement of Theorem~\ref{theone2}. 

By Lemma \ref{lemcin}, there exist a neighborhood~$\mathcal{N}$
of the origin and a function~$w\in C\left(\mathbb{R}^N\right)\cap C_0\left(\mathbb{R}^{N-l}\right)\cap C^\infty\left(\mathcal{N}\right)\cap\mathcal{A}$ such that
\begin{equation}\label{7UJHAanna}
{\mbox{$\Lambda_{-\infty} w=0$ in $\mathcal{N}$, }}\end{equation}
and such that 
\begin{equation}\label{9IKHAHSBBNSBAA}
\begin{split}&
{\mbox{all the derivatives of $w$ in 0 up to order $K$ vanish,}}\\
&{\mbox{with the exception of $\partial^\iota w \left(0\right)$ which equals~$1$,}}\end{split}\end{equation}
being~$\iota$ as in~\eqref{iorade}. 
Recalling the definition of~$\mathcal{A}$ on page~\pageref{CALSASS},
we also know that
\begin{equation}\label{SPE12129dd}
{\mbox{$\partial^{k_h}_{t_h}w=0$
in~$(-\infty,a_h)$, }}\end{equation}
for suitable~$a_h\in(-2,0)$,
for all~$h\in\{1,\dots,l\}$.

In this way, setting
\begin{equation}
\label{quaqui}
g:=w-f,
\end{equation}
we deduce from~\eqref{9IKHAHSBBNSBAA} that
$$
\partial^\sigma g\left(0\right)=0\quad \text{ for any }
\sigma\in\mathbb{N}^N \text{ with }\left|\sigma\right|\leq K.
$$
Accordingly, in $\mathcal{N}$ we can write
\begin{equation}
\label{quaqua}
g\left(x,y,t\right)=\sum_{\left|\tau\right|\geq K+1} {x^{\tau_1}y^{\tau_2}t^{\tau_3}h_\tau\left(x,y,t\right)},
\end{equation}
for some $h_\tau$ smooth in $\mathcal{N}$, where the multi-index
notation $\tau=(\tau_1,\tau_2,\tau_3)$ has been used. 

Now, we define
\begin{equation}\label{JAncasxciNasd}
u\left(x,y,t\right):=\frac{1}{\eta^\gamma}
w\left({\mathcal{T}}_\eta(x,y,t)\right).
\end{equation}
In light of~\eqref{SPE12129dd}, we notice that
$\partial^{k_h}_{t_h}u=0$
in~$(-\infty,a_h/\eta^{\frac1{\alpha_h}})$, for all~$h\in\{1,\dots,l\}$,
and therefore~$u\in C\left(\mathbb{R}^N\right)\cap C_0(\mathbb{R}^{N-l})\cap C^\infty
\left({\mathcal{T}}_\eta(\mathcal{N})\right)\cap\mathcal{A}$. We also claim that
\begin{equation}\label{DEN}
{\mathcal{T}}_\eta([-1,1]^{N-l}\times(a_1,+\infty)\times\ldots\times(a_l,+\infty))\subseteq {\mathcal{N}}.
\end{equation}
To check this, let~$(x,y,t)\in[-1,1]^{N-l}\times(a_1+\infty)\times\ldots\times(a_l,+\infty)$
and~$(X,Y,T):={\mathcal{T}}_\eta(x,y,t)$.
Then, we have that~$|X_1|=\eta^{\frac{1}{r_1}}|x_1|\le \eta^{\frac{1}{r_1}}$,
~$|Y_1|=\eta^{\frac{1}{2s_1}}|y_1|\le \eta^{\frac{1}{2s_1}}$,
~$T_1=\eta^{\frac{1}{\alpha_1}}t_1> a_1\eta^{\frac{1}{\alpha_1}}>-1$,
provided~$\eta$ is small enough.
Repeating this argument, we obtain that, for small~$\eta$,
\begin{equation}\label{CLOS}
{\mbox{$(X,Y,T)$ is as
close to the origin as we wish.}}\end{equation} 
%On the other hand, recalling~\eqref{le-a},
%we see that
%$$ t_1-a_1=t_1+2+\frac3{\eta^{\frac1{\alpha_1}}}\in\left[ 1+\frac3{\eta^{\frac1{\alpha_1}}},3+\frac3{\eta^{\frac1{\alpha_1}}}\right].$$
%As a consequence,
%$$ \eta^{\frac{1}{\alpha_1}}(t_1-a_1)\in\left[{\eta^{\frac1{\alpha_1}}}+3,\,
%3{\eta^{\frac1{\alpha_1}}}+3
%\right].
%$$
%This and~\eqref{9iwkfhvksc283AN} give that
%$$ T_1=\eta^{\frac{1}{\alpha_1}}(t_1-a_1)+a^\star_1\in
%\left[{\eta^{\frac1{\alpha_1}}}+1,\,
%3{\eta^{\frac1{\alpha_1}}}+3
%\right]\subset[0,10],
%$$
%provided that~$\eta$ is small enough. Repeating this argument
%for~$T_2,\dots,T_l$, we find that
%$$ T\in [0,10]^l.$$
%Accordingly, recalling the notation in~\eqref{CALAK},
%we have that~$(O,T)\in{\mathcal{K}}\subseteq{\mathcal{N}}^\star$.
{F}rom \eqref{CLOS} and the fact that~${\mathcal{N}}$
is an open set, we infer that~$(X,Y,T)\in{\mathcal{N}}$,
and this proves~\eqref{DEN}.

Thanks to~\eqref{7UJHAanna} and~\eqref{DEN}, we have that, in~$B_1^{N-l}\times(-1,+\infty)^l$,
\begin{align*}
&\eta^{\gamma-1}\,\Lambda_{-\infty} u\left(x,y,t\right) \\
=\;&\sum_{j=1}^n {\XA_j\partial_{x_j}^{r_j}w
\left({\mathcal{T}}_\eta(x,y,t)\right)}
+\sum_{j=1}^M {\XB_j(-\Delta)^{s_j}_{y_j} w
\left({\mathcal{T}}_\eta(x,y,t)\right)}
+\sum_{j=1}^l {\XC_jD_{t_h,-\infty}^{\alpha_h}w
\left({\mathcal{T}}_\eta(x,y,t)\right) }\\=\;&\Lambda_{-\infty}w
\left({\mathcal{T}}_\eta(x,y,t)\right)\\
=\;&0.
\end{align*}
These observations establish that $u$ solves the equation in $B_1^{N-l}\times(-1+\infty)^l$ and $u$ vanishes when $|(x,y)|\ge R$,
for some~$R>1$, and thus the claims in~\eqref{MAIN EQ:2}
and~\eqref{ESTENSIONE}
are proved.

Now we prove that $u$ approximates $f$, as claimed in~\eqref{IAzofm:2}. For this, using the monomial structure of $f$ in~\eqref{iorade}
and the definition of $\gamma$ in~\eqref{iorade2}, we have,
in a multi-index notation,
\begin{equation}
\label{monsca}
\begin{split}
&\frac{1}{\eta^\gamma}f\left({\mathcal{T}}_\eta(x,y,t)\right)=\frac{1}{\eta^\gamma\,\iota!} \;
(\eta^{\frac{1}{r}}x)^i (\eta^{\frac{1}{2s}}y)^I \big(\eta^{\frac{1}{\alpha}}t\big)^\mathfrak{I}
=\frac{1}{\iota!} x^i y^I t^\mathfrak{I}=f(x,y,t).
\end{split}
\end{equation}
%$$
%\frac{1}{\eta^\gamma}f\left(\eta^{\frac{1}{r_1}}x_1,\ldots,\eta^{\frac{1}{r_n}}x_n,\eta^{\frac{1}{2s_1}}y_1,\ldots,\eta^{\frac{1}{2s_M}}y_M,\eta^{\frac{1}{\alpha_1}}t_1,\ldots,\eta^{\frac{1}{\alpha_l}}t_l\right)=f\left(x,y,t\right).
%$$
Consequently, by~\eqref{quaqui}, \eqref{quaqua}, \eqref{JAncasxciNasd} and~\eqref{monsca},
\begin{align*}
u\left(x,y,t\right)-f\left(x,y,t\right) 
&=\frac{1}{\eta^\gamma}g\left(\eta^{\frac{1}{r_1}}x_1,\ldots,\eta^{\frac{1}{r_n}}x_n,\eta^{\frac{1}{2s_1}}y_1,\ldots,\eta^{\frac{1}{2s_M}}y_M,\eta^{\frac{1}{\alpha_1}}t_1,\ldots,\eta^{\frac{1}{\alpha_l}}t_l\right) \\
&=\sum_{\left|\tau\right|\geq K+1} {\eta^{\left|\frac{\tau_1}{r}\right|+\left|\frac{\tau_2}{2s}\right|+\left|\frac{\tau_3}{\alpha}\right|-\gamma}x^{\tau_1}y^{\tau_2}t^{\tau_3}h_\tau\left(\eta^{\frac{1}{r}}x,\eta^{\frac{1}{2s}}y,\eta^{\frac{1}{\alpha}}t\right)}
,\end{align*}
where a multi-index notation has been used, e.g. we have written
$$ \frac{\tau_1}{r}:=\left( \frac{\tau_{1,1}}{r_1},\dots,\frac{\tau_{1,n}}{r_n}\right)
\in\mathbb{R}^n.$$
Therefore, for any multi-index $\beta=\left(\beta_1,\beta_2,\beta_3\right)$ with $\left|\beta\right|\leq \ell$,
\begin{equation}
\label{eq:quaquo}\begin{split}
&\partial^\beta\left(u\left(x,y,t\right)-f\left(x,y,t\right)\right)
\\=\,&\partial^{\beta_1}_{x}\partial^{\beta_2}_{y}\partial^{\beta_3}_{t}\left(u\left(x,y,t\right)-f\left(x,y,t\right)\right)
\\=\,&\sum_{\substack{\left|\beta'_1\right|+\left|\beta''_1\right|=\left|\beta_1\right| \\ \left|\beta'_2\right|+\left|\beta''_2\right|=\left|\beta_2\right| \\ \left|\beta'_3\right|+\left|\beta''_3\right|=\left|\beta_3\right| \\ \left|\tau\right|\geq K+1}} {c_{\tau,\beta}\;\eta^{\kappa_{\tau,\beta}}\; x^{\tau_1-\beta'_1}y^{\tau_2-\beta'_2}t^{\tau_3-\beta'_3}\partial_{x}^{\beta''_1}\partial_{y}^{\beta''_2}\partial_{t}^{\beta''_3}h_\tau\left(\eta^{\frac{1}{r}}x,\eta^{\frac{1}{2s}}y,\eta^{\frac{1}{\alpha}}t\right)},
\end{split}\end{equation}
where
$$
\kappa_{\tau,\beta}:=\left|\frac{\tau_1}{r}\right|+\left|\frac{\tau_2}{2s}\right|+\left|\frac{\tau_3}{\alpha}\right|-\gamma+\left|\frac{\beta''_1}{r}\right|+\left|\frac{\beta''_2}{2s}\right|+\left|\frac{\beta''_3}{\alpha}\right|,
$$
for suitable coefficients $c_{\tau,\beta}$. Thus, to complete the proof of~\eqref{IAzofm:2}, we need to show that this quantity is small if so is $\eta$.
To this aim, we use~\eqref{am}, \eqref{blim}
and~\eqref{blo} to see that
\begin{eqnarray*}\kappa_{\tau,\beta}
%	&=&
%	\left|\frac{\tau_1}{r}\right|
%	+\left|\frac{\tau_2}{2s}\right|
%	+\left|\frac{\tau_3}{\alpha}\right|
%	-\gamma+\left|\frac{\beta''_1}{r}\right|
%	+\left|\frac{\beta''_2}{2s}
%	\right|+\left|\frac{\beta''_3}{\alpha}\right|
&\geq&\left|\frac{\tau_1}{r}\right|+\left|\frac{\tau_2}{2s}\right|
+\left|\frac{\tau_3}{\alpha}\right|-\gamma \\
& \geq&\delta\left(\left|\tau_1\right|+\left|\tau_2\right|
+\left|\tau_3\right|\right)-\gamma\\&\geq& K\delta
-\gamma\\&\geq& K_0\delta-\gamma\\&\geq& 1.
\end{eqnarray*}
Consequently, we deduce from \eqref{eq:quaquo} that $\left\|u-f\right\|_{C^\ell\left(B_1^N\right)}\leq C\eta$ for some $C>0$. By choosing $\eta$ sufficiently small with respect to $\epsilon$, this implies the claim in~\eqref{IAzofm:2}. This completes the proof of
Theorem \ref{theone2} when $f$ is a monomial.

\subsection{Proof of Theorem \ref{theone2} when $f$ is a polynomial}\label{AUJNSLsxcrsd}

Now, we consider the case in which~$f$ is a polynomial. In this case, we can write $f$ as
$$
f\left(x,y,t\right)=\sum_{j=1}^J c_jf_j\left(x,y,t\right),
$$
where each $f_j$ is a monomial, $J\in\mathbb{N}$ and $c_j\in\mathbb{R}$ for all $j=1,\ldots, J$.

Let $$c:=\max_{j\in\{1,\dots,J\}} c_j.$$ Then, 
by the work done in Subsection~\ref{7UHASGBSBSBBSB},
we know that the claim 
in Theorem \ref{theone2} holds true for each $f_j$, and so we can find $a_j\in(-2,0)^l$, $u_j\in C^\infty\left(B_1^N\right)\cap C\left(\mathbb{R}^N\right)\cap\mathcal{A}$
and $R_j>1$ such that $\Lambda_{-\infty} u_j=0$ in $B_1^{N-l}\times(-1,+\infty)^l$, $\left\|u_j-f_j\right\|_{C^\ell\left(B_1^N\right)}\leq\epsilon$ and $u_j=0$ if $|(x,y)|\ge R_j$. 

Hence, we set
$$
u\left(x,y,t\right):=\sum_{j=1}^J c_ju_j\left(x,y,t\right),
$$
and we see that
\begin{equation}
\label{pazxc}
\left\|u-f\right\|_{C^\ell\left(B_1^N\right)}\leq\sum_{j=1}^J {\left|c_j\right|\left\|u_j-f_j\right\|_{C^\ell\left(B_1^N\right)}}\leq cJ\epsilon.
\end{equation}
Also, $\Lambda_{-\infty} u=0$ thanks to the linearity of $\Lambda_{-\infty}$ in $B_1^{N-l}\times(-1,+\infty)^l$. Finally, $u$ is supported in $B_R^{N-l}$ in the variables $(x,y)$, being $$R:=\max_{j\in\{1,\dots,J\}} R_j.$$ This proves 
Theorem \ref{theone2} when $f$ is a polynomial (up to replacing $\epsilon$ with $cJ\epsilon$).

\subsection{Proof of Theorem \ref{theone2} for a general $f$}

Now we deal with the case of a general~$f$. To this end, we exploit
Lemma~2 in~\cite{MR3626547} and we see that
there exists a polynomial $\tilde{f}$ such that
\begin{equation}\label{6ungfbnreog}
\|f-\tilde{f}\|_{C^\ell(B_1^N)}\leq\epsilon.\end{equation}
Then, applying the result already proven in Subsection~\ref{AUJNSLsxcrsd}
to the polynomial $\tilde{f}$,
we can find $a\in(-\infty,0)^l$, $u\in C^\infty\left(B_1^N\right)
\cap C\left(\mathbb{R}^N\right)\cap\mathcal{A}$ and $R>1$ such that 
\begin{eqnarray*}
&& \Lambda_{-\infty} u=0 \quad{\mbox{ in }}B_1^{N-l}\times(-1,+\infty)^l, \\&&
u=0 \qquad\quad{\mbox{ if }}|(x,y)|\ge R,\\&&
\partial^{k_h}_{t_h} u=0 \quad\quad{\mbox{if }}t_h\in(-\infty,a_h),\qquad{\mbox{ for all }}h\in\{1,\dots,l\},
\\ {\mbox{and }}&&\|u-\tilde{f}\|_{C^\ell(B_1^N)}
\leq\epsilon.\end{eqnarray*}
Then, recalling~\eqref{6ungfbnreog}, we see that
$$ \|u-f\|_{C^\ell(B_1^N)}\leq\|u-\tilde{f}
\|_{C^\ell(B_1^N)}+\|f-\tilde{f}
\|_{C^\ell(B_1^N)}\leq2\epsilon.$$ Hence, the proof
of Theorem~\ref{theone2} is complete.
\qed

\end{document}